\documentclass[11pt,oneside,notitlepage]{article}
\usepackage{geometry}
\usepackage[ngerman,english]{babel}
\usepackage{amsmath}
\usepackage{amsrefs}
\usepackage{amssymb}
\usepackage{amsthm}
\usepackage{cancel}
\usepackage{caption,subcaption}
\usepackage{cite}
\usepackage{color}
\usepackage{enumitem}
\usepackage{framed}
\usepackage{graphicx}
\usepackage{hyperref}

\usepackage{layout}
\usepackage{mathtools}
\usepackage{mathrsfs}
\usepackage{placeins}
\usepackage{pdfpages}

\usepackage{tikz}
\usepackage{txfonts}
\usepackage[normalem]{ulem}

\newtheorem{theorem}{Theorem}
\newtheorem{assumption}[theorem]{Assumption}

\newtheorem{lemma}[theorem]{Lemma}
\newtheorem{notation}[theorem]{Notation}
\newtheorem{proposition}[theorem]{Proposition}
\newtheorem{remark}[theorem]{Remark}

\newcommand{\red}[1]{{#1}}
\newcommand{\stkout}[1]{\ifmmode\text{\sout{\ensuremath{#1}}}\else\sout{#1}\fi}

\newcommand{\at}{\tilde{a}}

\newcommand{\eps}{\varepsilon}

\newcommand{\Deltild}{\Delta \tilde{\mathbf{x}}}
\newcommand{\Delt}{\Delta \mathbf{x}}
\newcommand{\dx}{dx}
\newcommand{\dt}{dt}
\newcommand{\ddt}{\frac{\rm{d}}{\dt}}
\newcommand{\supp}{\operatorname{supp}}

\newcommand{\half}{\frac{1}{2}}

\newcommand{\N}{\mathcal{N}}

\newcommand{\Lam}{L}

\newcommand*\xbar[1]{%
   \hbox{%
     \vbox{%
       \hrule height 0.5pt 
       \kern0.5ex
       \hbox{%
         \kern-0.2em
         \ensuremath{#1}%
         \kern-0.1em
       }%
     }%
   }%
}
\DeclarePairedDelimiter\abs{\lvert}{\rvert}
\DeclarePairedDelimiter\norm{\lVert}{\rVert}

\newcommand{\energy}{\ensuremath{E}}
\newcommand{\energygap}{\ensuremath{\mathcal{E}}}
\newcommand{\E}{\energygap}

\newcommand{\dissipation}{\ensuremath{{D}}}
\newcommand{\D}{\dissipation}
\newcommand{\W}{{{V}}}
\newcommand{\Wb}{\bar{{V}}}
\newcommand{\Wt}{\widetilde{{V}}}
\newcommand{\Wtz}{W_0}
\newcommand{\Wm}{{V}_-}
\newcommand{\V}{{V}}

\newcommand{\leqsim}{\ensuremath{\lesssim}}

\newcommand{\co}{e_*}

\newcommand{\loc}{\mathrm{loc}}

\newcommand{\R}{\mathbb{R}}

\newcommand{\Hpkt}{{\dot{H}^{-1}}}

\newcommand{\wt}{\tilde{w}}
\newcommand{\Et}{\tilde{\E}}
\newcommand{\ft}{\tilde{f}}

\newcommand{\Tt}{\tilde{T}}

\definecolor{darkred}{rgb}{0.9,0.1,0.1}

\def\Xint#1{\mathchoice
{\XXint\displaystyle\textstyle{#1}}%
{\XXint\textstyle\scriptstyle{#1}}%
{\XXint\scriptstyle\scriptscriptstyle{#1}}%
{\XXint\scriptscriptstyle\scriptscriptstyle{#1}}%
\!\int}
\def\XXint#1#2#3{{\setbox0=\hbox{$#1{#2#3}{\int}$}
\vcenter{\hbox{$#2#3$}}\kern-.5\wd0}}

\def\dashint{\Xint-}

\numberwithin{equation}{section}   
\numberwithin{theorem}{section}

\begin{document}

\title{
	Optimal relaxation of bump-like solutions of the one-dimensional Cahn--Hilliard equation}
\author{Sarah Biesenbach\begin{footnote}{RWTH Aachen University, E-mail address: \href{mailto:biesenbach@eddy.rwth-aachen.de}{biesenbach@eddy.rwth-aachen.de}.}\end{footnote},
Richard Schubert\begin{footnote}{RWTH Aachen University, E-mail address: \href{mailto:schubert@math1.rwth-aachen.de}{schubert@math1.rwth-aachen.de}.}\end{footnote},
 and
	Maria G. Westdickenberg\begin{footnote}{RWTH Aachen University, E-mail address: \href{mailto:maria@math1.rwth-aachen.de}{maria@math1.rwth-aachen.de}.}\end{footnote}
}
\date{\today}
\maketitle
\begin{abstract}
\begin{center}
REVISED VERSION INCORPORATING THE ERRATUM ON LEMMA 2.1 \\
AND WITH A CORRECTION TO LEMMA 2.8 
\end{center}

In this paper we derive optimal relaxation rates for the Cahn-Hilliard equation  on the one-dimensional torus and the line. We consider initial conditions with a finite (but not small) $L^1$-distance to an appropriately defined bump. The result extends the relaxation method developed previously for a single transition layer (the ``kink'') to the case of two transition layers (the ``bump''). As in the previous work, the tools include Nash-type inequalities, duality arguments, and Schauder estimates. For both the kink and the bump, the energy gap is translation invariant and its decay alone cannot specify to which member of the family of minimizers the solution converges. Whereas in the case of the kink, the conserved quantity singles out the longtime limit, in the case of a bump, a new argument is needed. On the torus, we quantify the \red{(initially algebraic and ultimately exponential)} convergence to the bump that is the longtime limit; on the line, the bump-like states are merely metastable and we quantify the initial algebraic relaxation behavior.\\



\noindent Copyright information: This is an original manuscript of an article published by Taylor \& Francis in Communications in Partial Differential Equations on 14 November, 2021, available at: https://doi.org/10.1080/03605302.2021.1987458 and an erratum published by Taylor \& Francis in the same journal on 05 May, 2023, available at: https://doi.org/10.1080/03605302.2023.2202732.
\end{abstract}

\section{Introduction}\label{S:intro}
In this paper we use a nonlinear, energy-based method to study relaxation and metastability of the one-dimensional Cahn-Hilliard equation
\begin{align}\label{ch}
\left\{
\begin{aligned}
	u_t&=-\bigl(u_{xx}-G'(u)\bigr)_{xx} && \mbox{for }\, t\in(0,\infty),&& x\in I,\\
 u&=u_0 &&\mbox{for }\, t=0,&&x\in I,
\end{aligned}\right.
\end{align}
on an interval $I\subset \R$, where $G$ is a double-well potential with nondegenerate absolute minima at $\pm 1$ (cf.\ Assumption \ref{ass:G};  a canonical choice is $G(u)=\frac{1}{4}(1-u^2)^2$).
Introduced by Cahn and Hilliard in the mid-twentieth century \cite{CH}, \eqref{ch} is a well-known model for phase separation and is used widely in materials science and other application areas. The Cahn-Hilliard equation, and its second-order cousin, the Allen-Cahn equation, are also canonical examples of \emph{metastability} in $d=1$. Metastability, also called dynamic metastability to distinguish it from the (very different) phenomenon of stochastic metastability, refers to a system that exhibits a fast relaxation to a lower dimensional manifold of phase space followed by slow changes near the manifold.

Surprisingly few sharp results on metastability exist within the analysis literature. Most of the results that do exist characterize behavior for initial data in a small neighborhood of the slow manifold (``well-prepared'' initial data), for which methods based on linearization can be employed. However the rapid relaxation to the slow manifold is essential for metastability, is observed in physical and numerical experiments, and poses the most nonlinear behavior and hence perhaps the most interesting challenge for the analysis. Metastability of the $1$-d Allen-Cahn equation was pointed out in unpublished notes of John Neu and analyzed for well-prepared initial data in \cites{CP,FH,BK}. Order-one (or poorly prepared) initial data were analyzed in \cites{C,OR}; see also \cite{W}. Metastability of the $1$-d Cahn-Hilliard equation was observed in \cite{EF} and analyzed for well-prepared initial data in \cites{ABF,BH,BX1,BX2}; see Subsection~\ref{ss:lit} for more detail. In order to capture the metastability of poorly prepared initial data, one needs to understand the initial, algebraic relaxation. One method to do so is introduced in \cite{OW} and used to establish metastability in \cite{SW}, however as explained in more detail below, an $L^1$-based approach is better suited for coarsening. Such a method is developed in \cite{OSW} to explain algebraic relaxation to a kink on the line; we will recall this result shortly and explain how the current paper extends that work.

Our interest in this paper is in initial data that is order-one away from a ``bump.'' We make this question precise in two settings:

\noindent
\uline{Problem 1}: We study the problem on $[-\Lam,\Lam]$ subject to periodic boundary conditions (from now on, the ``torus with sidelength $L$'') and the mean constraint
\begin{align}
\dashint_{[-\Lam,\Lam]} u\, \dx=m\in[-3/4,3/4].\label{meanconstr}
\end{align}
(The choice  of $3/4$ is to make things concrete; any fixed cut-off less than $1$ would do.)

\noindent
\uline{Problem 2}: We study the problem on the real line
subject to $-1$ boundary conditions at $\pm \infty$ and the integral constraint
\begin{align}
  \int_{\R}(u+1)\,dx=2L.\label{intconstr}
\end{align}
\begin{figure} \begin{center}\includegraphics[width=11cm,height=3cm,keepaspectratio]{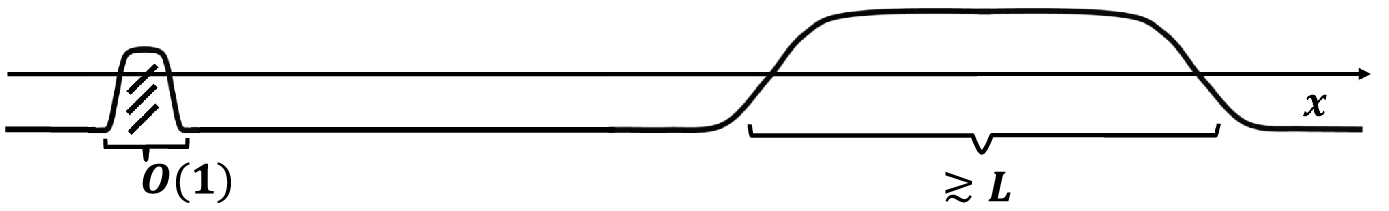}\end{center}
	\caption{Our method handles configurations with order one disturbances far away from the large bump.}\label{fig:2}
\end{figure}
In both cases $L$ represents the lengthscale of the bump. We are interested in initial data such as that depicted in Figure~\ref{fig:2}, for which an order-one $L^1$ disturbance occurs at an arbitrary location in the system. Whereas the HED method of \cite{OW} relies on the initial $\dot{H}^{-1}$ distance of the disturbance, which for the example in Figure~\ref{fig:2} would be large,
we seek relaxation results that are \emph{independent of the distance between the disturbance and the bump}. Profiles such as that depicted in Figure~\ref{fig:2} arise for instance as intermediate snapshots during coarsening of multi-bump solutions.

Problem 1 exhibits a multiscale relaxation: The solution converges  initially to ``a'' bump (on a $2$-parameter slow manifold) and subsequently to ``the'' bump (on a $1$-parameter slow manifold) whose lengthscale is fixed by \eqref{meanconstr}. These statements are made more precise below.
Although both \eqref{meanconstr} and \eqref{intconstr} are preserved by the evolution, the missing compactness in Problem 2 allows the solution to lose its bump-like structure in the longtime limit, and the solution converges to the constant state $\bar{u}\equiv -1$. Nonetheless we capture relaxation towards the bump-like state. See the discussion below for more explanation.

The Cahn-Hilliard equation is the gradient flow with respect to the $\Hpkt$ metric of the scalar Ginzburg-Landau energy
\begin{align*}
	E(u)=\int_I \frac{1}{2}u_x^2 +G(u)\,\dx.
\end{align*}
We will make use of the gradient flow structure of \eqref{ch} only in the form
\begin{align}
\dot{E}=-\D\qquad
\text{and}\qquad\norm{u(t,\cdot)-u(s,\cdot)}_{\dot{H}^{-1}(I)}\leq\int_s^t \D^\frac{1}{2}\, \,d\tau,\label{diff}
\end{align}
where $D$ is the dissipation and $\norm{\cdot}_{\dot H^{-1}(I)}$ is appropriately defined.
We emphasize that our results are nonperturbative in the sense that \uline{the initial data is not assumed to be close to the set of stable (or metastable) states}. We assume only that the $L^1$ distance of the initial data $u_0$ to a bump is \uline{finite} (not small) and that $u_0$ is bounded away from a two-bump state in the sense that there exists $\epsilon >0$ such that
\begin{align}
E(u_0) \leq 4\co-\epsilon,\label{e0bd}
\end{align}
where $\co$ represents the energy of the minimizer on $\R$ subject to $\pm 1$ boundary conditions.
Recall that in light of \eqref{diff}, the energy bound is preserved by the flow:
\begin{align}
  E(u)\leq 4\co-\epsilon\quad\text{for all }t>0.\label{diffcons2}
\end{align}
Establishing that the boundedness of the $L^1$ distance is also preserved by the flow lies at the heart of the method; see Subsection~\ref{ss:commentsmethod} below.

We now comment in more detail on the  $L^1$- framework for relaxation introduced in \cite{OSW}, on which our work is based.
The evolution equation is challenging (and interesting) because it is fourth-order  and the linearization on $\R$ has no spectral gap, which means that at most algebraic-in-time convergence to equilibrium can be expected for the problem on the line.
In \cite{OSW} an $L^1$-framework for relaxation was introduced and applied for initial data that was order-one away from a so-called kink $v$. The (normalized or centered) kink
is the minimizer of the energy subject to $\pm 1$ boundary conditions at $\pm\infty$ and $v(0) = 0$. As stated above, the constant $\co$ denotes the energy of the kink.
Clearly any shifted kink $v_a:=v(\cdot-a)$ is again an energy minimizer subject to $\pm 1$ boundary conditions, so that there is a one-parameter family of minimizers.

In \cite{OSW} the authors consider initial data such that
\begin{align}
	E(u_0)&\leq 3\co-\epsilon\qquad\text{(for some $\epsilon>0$)}\notag
	\end{align}
and
\begin{align*}
\int_\R\,(u_0-v)\,\dx=0\quad\text{while at the same time}\quad  \int_\R \abs{u_0-v}\,\dx<\infty.
\end{align*}
The first condition above is just a normalization; the second is the assumption of finite $L^1$ distance, and the goal of \cite{OSW} is to find optimal decay rates based on this assumption.
The idea behind the method is to establish and exploit $L^1$ control for all positive times. Defining
 $v_c:=v(\cdot-c)$ as the $L^2$-closest shifted kink to the solution $u$ at time $t\geq 0$ and the excess mass functional as
\begin{align*}
  \V:=\int_\R\abs{u-v_c}\,\dx,
\end{align*}
it is shown that
\begin{align}
\V\lesssim \V_0+1\qquad\text{for all times $t\geq 0$ },\label{Vbd}
\end{align}
and hence that the energy gap to the kink decays according to
\begin{align*}
	E(u)-e_*\lesssim \min\left\{E(u_0)-e_*,\frac{{\V_0}^{2}+1}{t^{\half}}\right\}.
\end{align*}
  For an explanation of the $\lesssim$ notation used here and throughout the paper, see Notation \ref{n:lesssim} below.

Energetic control alone gives no information about the kink to which the solution converges, since the ``valley floor'' of the energy landscape is flat and all kinks have the same energy. However because of the conservation law
\begin{align}
  \ddt \int_\R u-v\,\dx=0\label{cons}
\end{align}
and the ``nondegeneracy''
\begin{align*}
  \int_{\R} v-v_a\,\dx\neq 0 \qquad\text{for all $a\neq 0$},
\end{align*}
precisely the centered kink is selected as the long-time limit. Moreover,
from the representation formula
\begin{align*}
2c=  \int_{\R} v-v_c\,\dx,
\end{align*}
one can easily parlay \eqref{cons} and $\int_{\R}u_0-v\,\dx=0$
into the estimate
\begin{align*}
 \abs{c(t)} \lesssim \V(t)
\end{align*}
 for the  shift, which in light of \eqref{Vbd} is bounded uniformly in time.
The absence of such a simple selection mechanism and representation formula for bump-like initial data raises new questions.

Let us elucidate this difficulty first in Problem 1. Let $w$ denote the \uline{(normalized or centered) bump function}, by which we mean the function on the torus that minimizes the energy subject to the mean constraint~\eqref{meanconstr} and shifted so that its maximum is at the origin. (Here and throughout we take periodicity on the torus into account.)
Existence and properties of bump-like minimizers of $E$ with fixed mean go back to the work of Carr, Gurtin, and Slemrod \cite{CGS}. As for the kinks, every shifted bump $w_a:=w(\cdot-a)$ is again a minimizer. We will show that the solution of the Cahn-Hilliard equation converges to a shifted minimizer $w_{c_*}$, however the mean constraint does not pick out the shift. Indeed, although
\begin{align}
  \int_{[-\Lam,\Lam]} (u_0-w)\,\dx=0\quad\Rightarrow \quad\int_{[-\Lam,\Lam]} \big(u-w\big)\,\dx=0\;\;\text{for all }t\in(0,\infty),\label{notitw}
\end{align}
this conservation law does not select a longtime limit, since due to
\begin{align}
  \int_{[-\Lam,\Lam]}\big(w-w(\cdot-a)\big)\,\dx=0\;\;\text{for all }a\in [-L,L]\label{ohnow},
\end{align}
it follows from \eqref{notitw} in addition that
\begin{align*}
  \int_{[-\Lam,\Lam]}\big(u-w(\cdot-a)\big)\,\dx=0\;\;\text{\uline{for all} }a\in[-L,L]\;\text{and}\;\,t\in(0,\infty).
\end{align*}
Hence in this problem we face a situation in which there is a \emph{continuum of energy minimizers and identifying the limit configuration is nontrivial}.

Similarly, for Problem 2, we define the sharp interface function $\chi$ such that
\begin{align*} 
\chi(x)=\begin{cases}
1&\mbox{for } \, x\in\left[-\frac L2,\frac L2\right],\\
-1&\text{otherwise},
\end{cases}
\end{align*}
and express our integral constraint~\eqref{intconstr} as
\begin{align*}
\int_{\R} u-\chi\,\dx=0.
\end{align*}
We will see that the solution spends a long time ``not far from'' a shifted interface function $\chi_a=\chi(\cdot-a)$, but again the conservation law
\begin{align*}
  \int_{\R} (u_0-\chi)\,\dx=0\quad\Rightarrow \quad\int_{\R} \big(u-\chi\big)\,\dx=0\;\;\text{for all }t\in(0,\infty)
\end{align*}
is no help in identifying the shift. Here the situation is even more dire: After spending a long time near a shifted interface function, the solution  eventually converges to $\bar{u}\equiv -1$. Hence we seek  to characterize not a stable limit point $w_{c_*}$ but a \emph{metastable ``limit'' point}.

\subsection{The one-parameter slow manifold on the torus}\label{ss:wc} As for the energy on $\R$ subject to $\pm 1$ boundary conditions, the energy on the torus subject to the mean constraint~\eqref{meanconstr} has a \emph{one-parameter family of minimizers}, the so-called bumps mentioned above. It is not hard to check that, for fixed mean $m\in(-1,1)$ and $\Lam>0$ sufficiently large, the normalized bump function $w$ introduced above is smooth, has exactly two zeros separated by a distance order $\Lam$, and satisfies the Euler-Lagrange equation
\begin{align} \label{el}
-w_{xx} + G'(w) = \lambda,
\end{align}
where $\lambda$ is the Lagrange-multiplier connected to the integral constraint. (Here and throughout, we use ``smooth'' to mean $C^\infty$.) The other global minimizers are given by $w_a:=w(x-a)$, parameterized by the translation $a$, and, for reasons that will become clearer after reading Theorem~\ref{t:torus}, we will refer to the collection of all such functions as \emph{the one-dimensional slow manifold on the torus} and denote it by $\N_1(L)$.

Corresponding to a solution $u$ of the Cahn-Hilliard equation at  time $t\geq 0$, we define the shifted bump $w_{c}(x)\coloneqq w(x-c)$ as the \emph{$L^{2}$-projection of $u$ onto the set of bumps}, i.e.,
\begin{align*}
  \norm{u-w_c}_{L^2([-L,L])}\leq \norm{u-w_a}_{L^2([-L,L])}\quad\text{for all }a\in [-L,L],
\end{align*}
and we refer to the associated $c=c(t)$ as ``the shift.'' Notice for future reference that $w_{c}$ satisfies the Euler-Lagrange equation
\begin{align}\label{elL2}
\int_{[-\Lam,\Lam]} (u-w_{c})w_{cx}\,\dx=\int_{[-\Lam,\Lam]} uw_{cx}\,\dx=0.
\end{align}
The shift is not necessarily continuous in time and we neither need nor assume uniqueness of the shift. (However the shift becomes unique and moves continuously after an initial relaxation time.)
\begin{notation}\label{not:cdel}
  It will be convenient to define
  \begin{align*} 
\Delta c (t)\coloneqq \sup_{0 \leq t' \leq t} \, \abs{c(t')-c(0)}.
	\end{align*}
\end{notation}

The one-parameter family of shifted bumps will turn out to be an attracting, lower-dimensional manifold for the evolution: Even for initial data $u_0$ such that $\norm{u_0-w_{c(0)}}_{L^2([-L,L]}\gtrsim 1$, the solution $u$ of the Cahn-Hilliard equation is drawn into a small neighborhood of the one-dimensional manifold and converges to it as $t\to\infty$. Theorem~\ref{t:torus} proves quantitative versions of these remarks. Although, as explained above, the integral constraint does not control the shift, we will see in Theorem~\ref{t:torus} that the limiting shift is not too far (measured in terms of the initial $L^1$-distance) from the initial shift.

\subsection{The two-parameter slow manifold} \label{ss:slowmfld}

For both problems, we will measure the \emph{initial relaxation of the solution} in terms of its distance to a \emph{two-parameter family of glued kink profiles}. We begin with the problem on the line. The slow manifold consists of functions such that for some $q\in\R$ and some centers $\alpha, \beta\in \R$, there holds $\tilde{w}\approx v(\cdot-\alpha)$ for $x<q$ and $\tilde{w}\approx -v(\cdot-\beta)$ for $x>q$, and the function is smoothly interpolated
 near $q$. See Figure~\ref{fig:1} for an illustration.
	\begin{figure} \begin{center}\includegraphics[width=11cm,height=3cm,keepaspectratio]{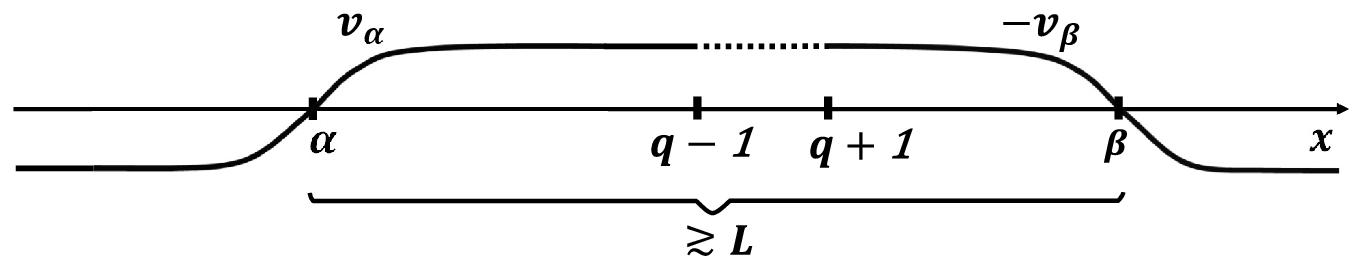}\end{center}
	\caption{A function $\wt\in\N(q,L)$ on the line, constructed from $v_{\alpha}$ and $-v_{\beta}$}\label{fig:1}
\end{figure}

To make things precise: For $L\geq L_0\gg 1$, we define $\N(L)$ to be the collection of functions $\tilde{w}$ such that:
\begin{itemize}
  \item There exists $q, \alpha, \beta \in\R$ with $q-\alpha\geq \frac{L}{2}$ and $\beta-q\geq \frac{L}{2}$.
  \item Away from $q$, the function is equal to a shifted kink:
  \[\tilde{w}=\begin{cases}
    v(\cdot-\alpha)& \text{on }(-\infty,q-1),\\
    -v(\cdot-\beta)&\text{on }(q+1,\infty).
  \end{cases}\]
  \item Near $q$, the function is smoothly interpolated so that $\wt \in C^3(\R)$,   and there holds
\begin{align*}
  \norm{\wt}_{L^2((q-1,q+1))}^2& \lesssim  \max\{v_\alpha(q-1),-v_\beta(q+1)\}, \\ \norm{\wt^{(k)}}_{L^2((q-1,q+1))}^2&\lesssim  \max\{1-v_\alpha(q-1),1{+}v_\beta(q+1)\} \quad \mbox{for } k=1,\,2,\,3,
\end{align*}
and
$ \norm{\wt-1}_{H^1((q-1,q+1))}$
are exponentially small, i.e., bounded above by $\exp(-L/C)$ for some $C<\infty$.
\end{itemize}
Occasionally we will want to refer specifically to glued kink profiles with a fixed gluing point $q$, which we will denote by $\N(q,\Lam)$. Almost always we will reduce to considering $\N(0,\Lam)$.

We define $\N(L)$ and $\N(0,L)$ similarly on the torus. Specifically, $\N(L)$ are functions  given by gluing together an increasing and decreasing kink. Bearing in mind periodicity, the size of the torus, and the mean condition, we choose $C<\infty$ and define $\N(L)$ as functions for which there exist $q\in[-L,L]$ and  $\alpha,\, \beta$ such that $\min\{\abs{\alpha-q},\abs{\alpha-(q+L)}\}\geq L/C$ and the same for $\beta$ (for instance $C=16$ will do).
Away from $q$ and $q+L (=q-L)$, the function is set equal to
\[\tilde{w}=\begin{cases}
    v(\cdot-\alpha)& \text{on }(q-L+1,q-1),\\
    -v(\cdot-\beta)&\text{on }(q+1,q+L-1),
  \end{cases}\]
with smooth interpolation near $q$ and $q+L$ as on $\R$.

Because there are two degrees of freedom ($\alpha$ and $\beta$), the two-parameter slow manifold (in contrast to the one-parameter slow manifold introduced above) includes functions with different distances between the zeros. We will see in Theorem~\ref{t:torus} that solutions of the Cahn-Hilliard equation on the torus relax \emph{quickly} to a member of the two-parameter slow manifold and then over time adjust the length between zeros and approach a member of the one-parameter slow manifold.

For both Problems $1$ and $2$, we will consider initial data $u_0$ with a finite $L^1$ distance to a function $\wt_0\in \N(L)$.
In the following, let $I$ denote the real line or the torus, as the case may be, and $I_-:=(-\infty,0]$ or $I_-:=(-L,0]$ and analogously for $I_+$. For $t>0$ and the solution $u$ of the Cahn-Hilliard equation, we will define an associated function $\wt$ by minimizing
\begin{align*}
	\norm{v_\alpha-u}_{L^2(I_-)} \quad  &\mbox{and} \quad \norm{v_\beta+u}_{L^2(I_+)},
\end{align*}
and we will denote the minimizing points and kinks by $\tilde{a}$ and $v_{\tilde{a}}$, and $\tilde{b}$ and $-v_{\tilde{b}}$, respectively. We will deduce
\begin{align}
	\int_{I_-} (u-v_{\tilde{a}})v_{{\tilde{a}} x}\,\dx =   0 = \int_{I_+} (u+v_{\tilde{b}})v_{{\tilde{b}} x}\, \dx ,\label{elg}
\end{align}
and establish appropriate bounds on $\tilde{a}$ and ${\tilde{b}}$. Notice that $v_{\tilde{a}}$ and $-v_{\tilde{b}}$ satisfy the Euler-Lagrange equation
\begin{align} \label{elv}
	-v_{xx}+G'(v)=0.
\end{align}
We will refer to $\wt$ as ``the $L^2$ closest glued kink profile,'' although the discerning reader will note that we minimize over the  half-lines and not the full space. We do it this way so that we can make use of \eqref{elg}.

We will see that for a long time, $\tilde{w}$ remains well-defined and that the $L^1$-distance of $u$ to $\tilde{w}$ remains controlled.
We will work with the quantities
\begin{alignat}{2}
\ft(t,x)&\coloneqq u(t,x)-\wt(t,x), &\notag\\
\Et(t) & \coloneqq \energy(u) - \energy(\wt) &\quad& \text{(the energy gap),} \nonumber \\
\D(t) & \coloneqq \int\Bigl(\bigl(u_{xx}-G'(u) \bigl)_x\Bigr)^2\,\dx &\quad & \nonumber \text{(the dissipation),}\\
\text{and }\quad \Wt(t)&\coloneqq \int \abs{u(t,x)-\wt(t,x)}\,\dx&\quad&\text{(the excess mass or volume)}. \nonumber
\end{alignat}
We typically omit the argument(s) from the left-hand side, except in cases in which we want to emphasize the dependence. Similarly, we sometimes write $u(t)$ (or similar) instead of $u(t,x)$ when it is the time-dependence that we wish to highlight.

\begin{notation}\label{not:zeros}
  It will also be convenient to define
  \begin{align*} 
\Delta \mathbf{\tilde{x}}(t) \coloneqq \sup_{0 \leq t' \leq t} \, \abs{\mathbf{\tilde{x}}(t')-\mathbf{\tilde{x}}(0)},
	\end{align*}
where $\mathbf{\tilde{x}}=\{\tilde{a},\tilde{b}\}$ are the zeros of $\wt$ and the distance between zeros is defined as in Notation~\ref{not:zeros2}.
\end{notation}

\subsection{Result for Problem 1}\label{ss:result1}


On the torus, we will initially work with the two-parameter slow manifold and the quantities $\ft$, $\Et$, $\Wt$ listed above. Once the energy gap has become sufficiently small, we will ``change horses'' and rely on the quantities
\begin{alignat}{2}
f_{c}(t,x) &\coloneqq u(t,x)-w_{c}(t,x),  &\notag\\
\E(t) & \coloneqq \energy(u) - \energy(w_c)=\energy(u)-\energy(w) &\quad& \text{(the energy gap),} \nonumber \\
\D(t) & \coloneqq \int\Bigl(\bigl(u_{xx}-G'(u) \bigl)_x\Bigr)^2\,\dx &\quad & \nonumber \text{(the dissipation),}\\
\text{and }\quad \W(t)&\coloneqq \int \abs{u(t,x)-w_{c}(t,x)}\,\dx&\quad&\text{(the excess mass or volume).}\nonumber
\end{alignat}
Furthermore we denote the initial energy gap
 and excess mass by
\begin{align*}
\E_0 \coloneqq \E(0)=E(u_0)-E(w).
\quad\text{ and }\quad \Wt_0\coloneqq\int\abs{u_0-\wt(0)}\,\dx.
\end{align*}

Our main result for Problem 1 is:
\begin{theorem}\label{t:torus}
There exists $C>0$ with the following property. For any $W_0<\infty$ and $\epsilon>0$, there exists $\Lam_0<\infty$ such that the following holds true. Let $u$ be a smooth solution of the Cahn-Hilliard equation \eqref{ch} on the torus with sidelength $\Lam\geq \Lam_0$. Suppose that the initial data $u_0$ satisfies
\begin{align} \label{star}
E(u_0)&\leq 4\co-\epsilon,\qquad\dashint_{[-\Lam,\Lam]}u_0\,\dx=m,\qquad \int_{[-\Lam,\Lam]}\abs{u_0-w}\,\dx\leq W_0,
\end{align}
where $m\in[-3/4,3/4]$ and $w$ is the centered bump. Then for all $t>0$, there holds
	\begin{align}\label{slowdowntorus}
	\W\lesssim W_0+1,\qquad \Deltild\lesssim W_0+1, \qquad \Delta c\lesssim W_0+1.
	\end{align}
Moreover, the metastable behavior and convergence to equilibrium can be characterized as follows. There exist times $0\leq T_0\leq T_1\leq T_2$ such that $$T_0 \lesssim W_0^4+1, \qquad  T_1\lesssim L^2,\qquad T_2\lesssim L^2\log W_0$$
and the following holds true.
 For $t\leq T_1$, the energy gap decays
according to
	\begin{align}\label{decayWtorus}
	\E\sim\abs{\Et}\lesssim \min \left\{ \E_0,\frac{W_0^2+1}{t^{1/2}}
\right\}.	\end{align}
In particular, at time $T_0$ the solution is close in energy and $L^\infty$ to the two-parameter slow manifold $\N(0,L)$:
\begin{align}
  \E(T_0)\sim\abs{\Et(T_0)}\ll 1,\qquad ||u(T_0)-\wt||_{L^\infty([-L,L])}\ll 1,\label{eq0}
\end{align}
where here $\ll$ requires that $T_0$ and $L_0$ be chosen sufficiently large. At time $T_1$, the solution is algebraically close in energy and $L^\infty$ to the two-parameter slow manifold:
\begin{align}
  \E(T_1)\sim\abs{\Et(T_1)}\ll \frac{W_0^2+1}{L},\qquad ||u(T_1)-\wt||_{L^\infty([-L,L])}\ll\frac{W_0+1}{L^{1/2}}.\label{egL}
\end{align}
On $[T_1,T_2]$, the energy gap to the two-parameter slow manifold and the dissipation decay exponentially in time:
\begin{align*}
  \Et+D\lesssim \frac{W_0^2+1}{L}\exp(-(t-T_1)/(CL^2)).
\end{align*}
At time $T_2$, the solution is (uniformly in $W_0$) algebraically close in energy and $L^\infty$ to the one-parameter slow manifold:
\begin{align}
  \E(T_2)\sim\abs{\Et(T_2)}\ll \frac{1}{L},\qquad ||u(T_2)-w_c||_{L^\infty([-L,L])}\ll\frac{1}{L^{1/2}}.\label{118b}
\end{align}
For  $t\geq T_2$, the energy gap to the one-parameter slow manifold and the dissipation decay exponentially in time:
\begin{align}
  \E+\D\lesssim \frac{1}{L}\exp(-(t-T_2)/(CL^2)),\label{expy}
\end{align}
and the solution converges exponentially in time to the one-parameter slow manifold:
\begin{align}
  ||u(t)-w_c||_{L^\infty([-L,L])}\lesssim \frac{1}{L^{1/2}}\exp(-(t-T_2)/(CL^2)).\label{expylinf}
\end{align}
Moreover the shift $c(t)$ converges exponentially in time to a fixed location $c_*\in [-L,L]$.
\end{theorem}
One could summarize by saying that we capture metastability of $\N(L)$ followed by stable relaxation/convergence to $\N_1(L)$. \red{The initial relaxation is algebraic; later stage relaxation is exponential (with  transition timescale $\lesssim L^2$).}
\begin{remark}[No prefactor estimate]
The content of \eqref{expy} should not be overestimated: We do not pinpoint the precise constant in the exponential factor, and in \eqref{expy} we do not want to suggest that we recover the subexponential prefactor. The estimate in \eqref{expy} is only meant to indicate that the bound is initially $1/L$ and, once time is large enough, exponentially small in $L$.
\end{remark}
\begin{remark}[Crossover time]
A comparison with \cite{SW} reveals the following behavior: For initial data $u_0$ such that $1\lesssim W_0:=\norm{u_0-w}_{L^1([-L,L])}\ll \norm{u_0-w}^2_{\dot{H}^{-1}([-L,L])}=:H_0<\infty$ (as expected for a generic coarsening event), the rate~\eqref{decayWtorus} is best initially but is overtaken by $H_0/t$ decay when
\begin{align*}
  T\sim \frac{H_0^2}{W_0^4}.
\end{align*}
\end{remark}
\begin{remark} \label{r:norm}
  Choosing the bump with its maximum at zero in~\eqref{star} should be viewed as a normalization. If instead some shifted bump $w(\cdot-a)$ were considered, we would work instead with $\N(a,L)$. In contrast the second condition in~\eqref{star} should be viewed as setting the lengthscale $L$ of the stable bump to which the solution converges.
\end{remark}
\begin{remark}[Uniqueness of the zeros] \label{r:EDsmall}
Once the energy gap and the dissipation are small, one can show as in \cite[Lemma~2.13]{SW} that the solution $u$ has exactly two zeros. From uniqueness of the zeros one can deduce that they move continuously in time, and according to Lemma~\ref{l:u2zeros}, the zeros are a distance of order $L$ apart.
\end{remark}
\begin{remark}[Dissipation decay] \label{r:algebraic_D}
From Lemma \ref{l:dissy_bound}, one obtains
\begin{align*}
	\D(t) \lesssim  \frac{ W_0^4+1}{t}\qquad \text{on }[1,T_1],
\end{align*}
but the rate does not appear to be optimal.
\end{remark}

\subsection{Result for Problem 2}
On the line, the bump-like states converge in the long-time limit to $u\equiv -1$. However we will show that the two-parameter slow manifold $\N(\Lam)$ is \emph{metastable}: For initial data such that the $L^1$-distance to some $\wt_0\in\N(\Lam)$ is bounded (but not small), the solution is drawn into a small neighborhood of $\N(\Lam)$ and stays there for a long time. By translation invariance, we may assume that $\wt_0\in \N(0,L)$.
Our main result for Problem 2 is the following.
\begin{theorem}\label{t:mainresultuniversal}
For any $W_0 < \infty$ and $\epsilon>0$, there exists $L_0<\infty$ such that for any $L\geq L_0$ the following holds true.
 Let $u$ be a smooth solution of the Cahn-Hilliard equation \eqref{ch} on $\R$ with initial data $u_0$.
 Suppose that
 	\begin{align*}
	E(u_0) \leq 4\co-\epsilon, \quad \int_{\R}(u_0+1)\,\dx=2L,
\end{align*}
and there exists $\wt_0\in\mathcal{N}\red{(0,2L)}$ such that
\begin{align}
 \int_\R\abs{u_0-\wt_0}\,\dx\leq \Wtz.\label{W_uni}
\end{align}
Then there exist times $0\leq T_0\leq T_1$ such that
\begin{align*}
T_0\lesssim \Wtz^4+1,\qquad  T_1\lesssim L^2
\end{align*}
and for all $t\in (0,T_1]$, there exists $\wt\in\red{\N(0,L)}$ and there holds
\begin{align*}
\int_\R \abs{u-\wt}\,\dx\lesssim W_0+1,\qquad \Deltild\lesssim \Wtz+1.
\end{align*}
Moreover for all $t\in (0,T_1]$, the energy gap decays at the optimal algebraic rate
	\begin{align}\label{edecay_uni}
	\abs{E(u)-2e_*}\lesssim \min \left\{ E(u_0)-2e_*,\frac{\Wtz^2+1}{t^{1/2}} \right\}.
	\end{align}
In particular, at time $T_0$
the solution is close in energy and $L^\infty$ to the two-parameter slow manifold:
\begin{align}
  \abs{\Et(T_0)}\ll 1,\qquad ||u(T_0)-\wt||_{L^\infty(\R)}\ll 1,\notag
\end{align}
(where $\ll$ requires that $T_0$ and $L_0$ be sufficiently large) and at time $T_1$
this improves to the algebraic closeness :
\begin{align}
  \Et(T_1)\ll \frac{ W_0^2+1}{L},\qquad ||u(T_1)-\wt||_{L^\infty(\R)}\ll\frac{ W_0+1}{L^{1/2}}.\label{egL22}
\end{align}
Also for all $t\in [T_0,T_1]$, $u$ has exactly two zeros $\{a(t),b(t)\}$   and for all times $s<t$ in this interval, changes are controlled via
\begin{align}
  \abs{a(t)-a(s)}+\abs{b(t)-b(s)}\lesssim 
  \left(\max_{[s,t]}\Et^{1/2}\right)
(t-s)^{1/4}+1.\label{da}
\end{align}

Finally, algebraic closeness to the two-parameter slow manifold, i.e. \eqref{egL22},
persists for a timescale at least of order $L^2/( W_0^4+1)$.
\end{theorem}

\begin{remark}[Improved bounds]
One can improve the bound on the energy gap by iterating and arguing as in Lemma~\ref{l:motionofzeros}, but the resulting bound on the motion of the zeros is not summable. To obtain good control on the motion of the zeros, one may need to impose additional constraints on the initial data that give information about how far the disturbance is from the bump. For instance if one assumes in addition that $\norm{u_0-\wt_0}_{\dot{H}^{-1}(\R)}^2\lesssim L$, one could combine the $L^1$ relaxation result from this paper with the method of \cite{OW}.
\end{remark}
\begin{remark}[Dissipation decay]
As in Remark~\ref{r:algebraic_D},
one obtains from Lemma~\ref{l:dissy_bound} that
\begin{align*}
	\D(t) \lesssim  \frac{ W_0^4+1}{t}\qquad \text{on }[1,T_1],
	\end{align*}
but the rate does not appear to be optimal.
\end{remark}

Theorem~\ref{t:mainresultuniversal} establishes metastability of $\N(L)$ on $\R$.
Since there is no state with zero dissipation and $-1$ boundary conditions at $\pm\infty$ other than $\bar{u}\equiv -1$, the solution converges in the longtime limit to $\bar{u}$ (and the excess mass from the integral constraint ``leaks'' to infinity). Once the energy has become strictly less than $2e_*$, it is not hard to obtain the following quantitative result. We present the result for initial data $u_0$, but it is natural to imagine this data being an intermediate-time state in a longer coarsening process.
We will use the notation
\begin{align*}
 \Wm (t)&\coloneqq \int \abs{u(t,x)+1}\,\dx&\quad&\text{(the excess mass or volume relative to $-1$).}
 \end{align*}

\begin{theorem}\label{t:mainresultubar}
	For any  $\epsilon >0$, the following holds true.
 Let $u$ be a smooth solution of the Cahn-Hilliard equation \eqref{ch} on the line with initial data $u_0$.
 Suppose that $u_0$ satisfies
 	\begin{align*}
	E(u_0) \leq 2\co-\epsilon
	\qquad
and \qquad
 \Wm (0)<\infty . 
\end{align*}
Then the solution converges to $\bar{u}\equiv -1$ in $H^1(\R)$ and for all $t\geq 0$, there holds
\begin{align*}
  \Wm \lesssim \Wm (0)+1 \quad \mbox{and} \quad  E(u)\lesssim \min\left\{E(u_0),\frac{\Wm (0)^2+1}{t^{1/2}}\right\}.
\end{align*}
\end{theorem}
The proof of Theorem~\ref{t:mainresultubar} is also included in Section~\ref{S:1} below.
\begin{remark}[Boundary conditions]\label{rem:boundedness}
  Notice that $E < 4e_*-\epsilon$ and either $\Wt < \infty$ or $\Wm<\infty$ imply that $u$ is bounded in $L^\infty(\R)$ (using the trick of Modica and Mortola) and satisfies $-1$ boundary conditions at $\pm \infty$. Indeed, the energetic and $L^1$ bounds imply that
  $G(u)\gtrsim (u+1)^2$ outside of a large compact set, so that $(u+1)\in H^1(\R)$ and
  \begin{align*}
    \sup_{(M,\infty)}\abs{u+1}^2\leq\int_M^\infty (u+1)^2+((u+1)_x)^2\,\dx=o(1)_{M\uparrow\infty},
  \end{align*}
  and similarly on $(-\infty,-M)$.

\end{remark}

\subsection{Additional comments on the literature}\label{ss:lit}
Metastability of the Cahn-Hilliard equation was already observed numerically by Elliott and French in \cite{EF}; see also \cite{M}. Analytical results on the metastability include \cite{ABF,BH,BX1,BX2}. In particular, slowness of the metastable phase was observed under varying smallness assumptions on the initial data.	
Bronsard and Hilhorst \cite{BH} use an elementary energy method to show that solutions for
initial data algebraically close in $L^1$ to the slow manifold remain trapped algebraically close to the slow manifold for an algebraically long time.
Bates and Xun \cites{BX1,BX2} examine the behavior of a solution of the Cahn-Hilliard equation on a bounded spatial domain with initial data that is algebraically close in $H^2$ to the slow manifold. By exploiting the spectrum of the linearized Cahn-Hilliard operator, they prove that the solution relaxes towards the slow manifold and stays close for an exponentially long time.

As explained above, capturing metastability of the Cahn-Hilliard equation for poorly prepared initial data requires understanding the initial, algebraic relaxation towards the slow manifold. Algebraic rates of convergence for the Cahn-Hilliard equation on the line and initial data close (in some form) to a kink profile has been studied in \cites{BKT,CCO,H}. In \cite{OW}, the authors demonstrate that closeness of the initial data to the slow manifold is not necessary: They consider initial data that is merely a finite distance from the slow manifold in $\Hpkt$ and establish optimal algebraic relaxation under this assumption.

The fact that exponential closeness to the slow manifold is not only \emph{propagated} but also \emph{generated}, i.e., that poorly prepared initial data is drawn into an exponentially small neighborhood of the slow manifold, is demonstrated for the Allen-Cahn equation via a sharp energy method by Otto and Westdickenberg (née Reznikoff) in \cite{OR}. Scholtes and Westdickenberg \cite{SW} combine this technique with the algebraic relaxation from \cite{OW} to prove metastability of the Cahn-Hilliard equation on the torus: They show that
initial data that is an order-one distance in energy and $\Hpkt$ to the slow manifold is drawn into an exponentially small neighborhood of the slow manifold and remains trapped there for an exponentially long time.

As described above, \cite{OSW} drops the $\Hpkt$ condition in favor of an $L^1$ distance, since this yields a more generic condition for coarsening. In this work, we again consider initial data that is merely order-one away in $L^1$ from a limit point or metastable configuration (see Figure~\ref{fig:2}). The challenges of our setting compared to that of \cite{OSW} have been explained earlier in the introduction.

\subsection{Comments and method}\label{ss:commentsmethod}
Problems 1 and 2---and the tools we use to analyze them---are closely related, so we collect some comments that apply to both.
For an extended discussion of optimal relaxation rates for the Cahn-Hilliard equation on the line, we refer to \cite{OW}. The energetic relaxation in~\eqref{decayWtorus} and~\eqref{edecay_uni} is optimal and is the same rate as for the heat equation on the line with $L^1$ initial data.
The main idea for both problems is to combine
\begin{itemize}
  \item energy--energy--dissipation (EED) relationships,
  \item a Nash-type inequality of the form
  $\E\lesssim \D^{1/3}(\W+1)^{4/3},$
  \item and a duality argument that controls the growth of $\W$.
\end{itemize}
In particular, since on the torus $E(w_c)$ is independent of the shift $c$, the differential equality \eqref{diff} improves to $\dot \E=-\D$ and
the Nash-type inequality  and an integration in time yield
   $$ \E\red{(T)} \lesssim \frac{\sup_{t\leq T} \W(t)^2+1}{\red{T}^{1/2}}, $$
so that controlling $\W$ is critical. The strategy for the two-parameter slow manifold of glued kink profiles is similar, but since the energy is not constant, we also need to control exponentially small error terms.

\subsection{Conventions and organization}
Our assumptions on the potential $G$ are as follows.
\begin{assumption}\label{ass:G}
	The double-well potential $G$
	is assumed to satisfy
	\begin{align*}
	\begin{aligned}
	&\bullet &&G\in C^2 \mbox{ is even,} && \\
	&\bullet &&G(\pm1) = 0\;\text{ and }\; G(u) > 0 \; \mbox{ for } u \neq \pm 1, \\
	&\bullet &&G'(u) \leq 0 \; \mbox{ for } u\in[0,1], && \\
	&\bullet &&G''(\pm1)  >0. &&
	\end{aligned}
	\end{align*}
\end{assumption}

\begin{notation} \label{n:lesssim}
	Throughout the paper we use the notation
	\begin{align*} A \lesssim B \end{align*}
for positive quantities $A$ and $B$ if there exists a universal constant $C\in(0,\infty)$ depending at most on the potential $G$ and $\epsilon$ from \eqref{e0bd}, such that
$L \geq L_0$ implies $A \leq CB$. If $A \lesssim B$ and $B \lesssim A$ we write $A \sim B$.
	
Moreover, we  write
\begin{align*} A \ll B \end{align*}
if for every $\eps >0$ there is a parameter value sufficiently large so that $A \leq \eps B$ as long as the parameter is at least that large. Typically the parameter is $L\geq L_0$. Exceptions are indicated.		
The notation
$A \approx B $
is used when $\abs{A-B} \ll 1$.

When there is a constant (typically in an exponential) that is universal (depending at most on $G$) but which we are not interested in tracking precisely, we denote this constant by $C$. Hence the value of $C$ may vary from line to line.
\end{notation}
\begin{notation}
  In specifying constants, we will always emphasize the ``worst case scenario'' in the sense that we write that there exists
  $
    \epsilon>0
  $
  to mean that there exists $\epsilon\in (0,\infty)$, when the important point to keep in mind is that $\epsilon$ is positive and analogously
  $
    L_0<\infty
  $
  to mean that there exists $L_0\in(0,\infty)$, when the important point to keep in mind is that $L_0$ is finite.
\end{notation}
\begin{notation}\label{not:zeros2}
  	For two sets of zeros $\mathbf{x}(t)=\{a(t),b(t)\}$, $\mathbf{x}(s)=\{a(s),b(s)\}$, we use $\abs{\mathbf{x}(t)-\mathbf{x}(s)}$ to denote the distance between the zeros:	
	\begin{align*}
	\abs{\mathbf{x}(t)-\mathbf{x}(s)}\coloneqq \max\left\{|a(t)-a(s)|,|b(t)-b(s)| \right\},
	\end{align*}
bearing in mind periodicity on the torus.
\end{notation}

\begin{notation}
  	We will occasionally consider integral functionals like $\E$ restricted to a subinterval $I$ of the original domain of definition. In this case we will denote this restriction via a subscript, e.g., by $\E_I$.
\end{notation}

\subsubsection*{Organization.}

The energy--energy--dissipation relationships that are central to our method are stated, together with elementary consequences, in Subsection~\ref{ss:eed}.
General statements that are important on both the line and  the torus are collected in Subsection~\ref{ss:prelim}. Problem 1 and  the proof of Theorem~\ref{t:torus} are contained in
Section~\ref{S:2}.
Section~\ref{S:1} tackles Problem 2 and the proofs of Theorems~\ref{t:mainresultuniversal} and \ref{t:mainresultubar}.
Remaining proofs are given in Section~\ref{S:aux} and the appendix, with the following rationale.

\red{The proofs of Lemmas~\ref{l:eedwt},~\ref{l:basiceedtorusw}, and~\ref{l:basiceed-}---as well as those of Lemmas \ref{l:nash}~--~\ref{l:dissy}---are relegated to the appendix, since they closely resemble proofs from \cites{OR,OW,OSW}. The proofs of Lemma~\ref{l:EDtorus}  and the lemmas from Subsection~\ref{ss:prelim} are not similar to previously published results and are included in Subsection~\ref{ss:prelimpf}. Subsection~\ref{ss:propproof} contains the proofs for the control of the excess mass.}

\section{Basic estimates}\label{s:mainresults}
We begin with some preliminary estimates that are central to our work. Unless otherwise specified, we will work on \emph{either} the real line or the torus with sidelength $L$. We will use $I_-$ to denote $[-L,0]$ or $(-\infty,0]$, and $I_+:=-I_-$.
Integrals appearing below without an explicit domain of integration are understood to be over the full domain.

\subsection{EED and elementary consequences} \label{ss:eed}

On both the real line and the torus, we will track the solution initially in time using EED with respect to a  glued kink profile $\wt$ (cf. Subsection~\ref{ss:slowmfld}).

\begin{lemma}[Energy gap and dissipation estimates compared to $\wt$] \label{l:eedwt}
There exists $C<\infty$ with the following property.	For every $W_0 < \infty$ and $\epsilon >0$, there exists  $L_0<\infty$ such that for any $L \geq L_0$, one has the following. Consider any smooth function $u$  such that
\begin{align*}
E(u)&\leq 4\co-\epsilon
\end{align*}
and such that there is an $L^2$-closest glued kink profile $\wt \in \N(0,L)$ and
\begin{align}
  \int\abs{u-\wt}\,\dx\leq W_0. \label{EEDassumptwt}
\end{align}
Then there holds
	\begin{align}
		\int \ft^2+\ft_x^2\,\dx &\lesssim \abs{\Et} + \exp \left( -L/C\right) \label{eq:Eerrorlow}\\
		\abs{\Et} &\lesssim \int \ft^2+\ft_x^2\,\dx +  \exp \left( -L/C\right), \label{eq:Eerrorup} \\
		\mbox{and} \qquad \int \ft_x^2+\ft_{xx}^2+\ft_{xxx}^2\,\dx & \lesssim  \D + \exp \left( -L/C\right). \label{eq:Derror}
	\end{align}
\end{lemma}
On the torus, the Hardy inequality makes it possible to estimate $\Et$ by $L^2D$ up to an exponential error.
\begin{lemma}[Energy-energy-dissipation estimate on the torus]\label{l:EDtorus}
There exists $C<\infty$ with the following property.	For every $W_0 < \infty$ and $\epsilon >0$, there exists  $L_0<\infty$ such that the following holds true. Let $u$ be a smooth function on the torus with sidelength $L\geq L_0$ satisfying
\begin{align*}
E(u)&\leq 4\co-\epsilon
\end{align*}
and such that there exists an $L^2$-closest glued kink profile $\wt \in \N(L)$ with
\begin{align*}
  \int\abs{u-\wt}\,\dx\leq W_0. 
\end{align*}
Then there holds
\begin{align*}
\Et\lesssim L^2 D+\exp(-L/C).
\end{align*}
\end{lemma}	

On the torus, once the energy gap $\Et$ has become small enough, we change perspective and compare $u$ to the  $L^2$-closest shifted bump $w_c$. This is made possible by the following lemma.
\begin{lemma}[Energy gap and dissipation estimates compared to $w_c$] \label{l:basiceedtorusw}
For every $W_0 < \infty$ and $\epsilon >0$, there exists  $L_0<\infty$ and $\eps>0$ such that the following holds true.  Let  $u$ be a smooth function on the torus with sidelength $L\geq L_0$ satisfying
\begin{align} \label{EEDwcA}
E(u)&\leq 4\co-\epsilon,\qquad
  \int_{[-L,L]}\abs{u-w_c}\,\dx\leq W_0, \qquad\text{and}\qquad \E \leq \frac{\eps}{L}.
\end{align}
Then there holds
	\begin{eqnarray}
	\int_{[-\Lam,\Lam]} f_c^2 + f_{cx}^2 \, \dx & \lesssim & \E \;\; \lesssim \int_{[-\Lam,\Lam]} f_c^2 + f_{cx}^2 \; \dx \label{eq:energyestimatewctorus} \\
	\mbox{and} \quad \quad \quad \int_{[-\Lam,\Lam]} f_{cx}^2 + f_{cxx}^2+f_{cxxx}^2 \, \dx & \lesssim & \D. \label{eq:dissipationestimatewctorus}
	\end{eqnarray}
\end{lemma}	

Notice for future reference that because of \eqref{diffcons2}, there holds
	\begin{align} \notag
	\E \leq \E_0 \lesssim 1 \quad \mbox{and} \quad \abs{\Et} \lesssim 1 \qquad\;\;\mbox{for all } t \geq 0,
	\end{align}
and hence, using the elementary estimate
\begin{align} \label{elementary}
\sup \abs{f} \lesssim \left(\int f^2\,\dx \right)^{1/4} \left( \int f_x^2\,\dx \right)^{1/4},
\end{align}
the conditions of Lemma~\ref{l:eedwt} lead to
	\begin{eqnarray}\label{eq:basiceedconsequence}
	\norm{\ft}_{L^{\infty}} &\lesssim & \min \left\{ \abs{\Et}^{1/2}, \abs{\Et}^{1/4} \D^{1/4} \right\} +  \exp \left( -L/C\right) \overset{\eqref{diffcons2}}{\lesssim} 1.  \notag\\
	\norm{ \ft_{xx}}_{L^{\infty}} &\lesssim& \D^{1/2} +  \exp \left( -L/C\right).	
	\end{eqnarray}
Under the conditions of Lemma~\ref{l:basiceedtorusw}, we have also
\begin{align}\label{eq:f_c_infty}
\norm{f_c}_{L^{\infty}([-L,L])} \lesssim \min \left\{\E^{1/2},\E^{1/4}\D^{1/4} \right\} \overset{\eqref{diffcons2}}{\lesssim} 1.
\end{align}

For the comparison to $\bar{u}\equiv -1$ on the line in Theorem~\ref{t:mainresultubar}, the estimates simplify significantly.
\begin{lemma}[Energy and dissipation estimates when comparing to $-1$] \label{l:basiceed-}
	Let $\epsilon >0$. Let $u$ be a smooth function satisfying
\begin{align*}
E(u)&\leq 2\co-\epsilon\qquad\text{and}\qquad \int_\R \abs{u+1}\,\dx <\infty.
\end{align*}
Then there holds
	\begin{eqnarray}
	\int_{\R} (u+1)^2 + u_x^2 \, \dx & \lesssim & E \;\; \lesssim \int_{\R} (u+1)^2 + u_x^2 \, \dx, \label{eq:energyestimate-} \\
	\int_{\R} u_x^2 + u_{xx}^2 +u_{xxx}^2 \, \dx & \lesssim & \D. \nonumber 
	\end{eqnarray}
\end{lemma}	

We prove Lemma \ref{l:EDtorus} in Subsection~\ref{ss:prelimpf}. We relegate the proofs of the other lemmas (Lemma~\ref{l:eedwt}, Lemma~\ref{l:basiceed-}, and Lemma~\ref{l:basiceedtorusw}) to the appendix since they follow by adapting the proofs from \cite{OR,OW}.
	
\subsection{Additional general results}\label{ss:prelim}	

We list here some additional results that are independent of the domain and hence apply to both Problems 1 and 2. The proofs are deferred to Subsection~\ref{ss:prelimpf} and the appendix.

We begin with the following remark about the slow manifolds.
\begin{remark} \label{r:expt}
The energy of functions in the slow manifolds is exponentially close to that of two kinks:
For $\wt\in\N(L)$ and $w\in\N_1(L)$, there holds:
\begin{align*}
  \abs{E(\tilde{w})-2e_*}+\abs{E(w)-2e_*}\lesssim \exp(-L/C).
\end{align*}
We will often use this fact to estimate energy differences by exponentially small terms. These estimates are well-known and versions of them can be found for instance in \cites{CGS,CP,OR}.
\end{remark}

Within the proofs of the main theorems, it will be necessary to show that $u$ has (at least) two well-separated zeros and an $L^2$ closest glued kink profile $\wt\in \N(L)$.
\begin{lemma}[$L^1$ closeness to the slow manifolds] \label{l:u2zeros}
	For every $W_0 < \infty$, there exists  $L_0<\infty$ such that for  $L \geq L_0$, the following holds true. Let $\wt_0 \in \mathcal{N}(0, 2 L)$ denote a glued kink profile with zeros $\tilde{\mathbf{x}}_0=\{\tilde{a}_0,\tilde{b}_0\}$. Let $u$ be a smooth function such that $E(u)\leq 4e_*-\epsilon$ and
	\begin{align}
	\int \abs{u-\wt_0}\,\dx \leq W_0.\label{onew}
	\end{align}
	Then $u$ has (at least) two zeros $\mathbf{x} =  \{ a, b \}$ such that $a < b$ are well-separated  and  not far from the zeros of $\wt_0$:
	\begin{align*}
	\abs{a}, \abs{b} \gtrsim L, \quad  \quad  \abs{b - a} \gtrsim L,\qquad \abs{\mathbf{x}-\tilde{\mathbf{x}}_0} \lesssim W_0 + W_0^{1/2}.
	\end{align*}
	In addition, there exists an $L^2$-closest glued kink profile $\wt\in \mathcal{N}(0,L)$ with zeros $\tilde{\mathbf{x}}=\{\tilde{a}, \tilde{b}\}$ and such that
	\begin{align*}
		\int \abs{u-\wt}^2\,\dx +\int\abs{u-\wt}\,dx\lesssim W_0+1, \qquad
	\abs{\mathbf{x}-\tilde{\mathbf{x}}} \lesssim W_0 + 1.
	\end{align*}
On the torus, we obtain analogous bounds if $\wt_0$ is replaced by an element of $\N_1(L)$ or if $\wt$ is replaced by $w_c$. For instance, if \eqref{onew} holds with respect to some $w_d$ with zeros $\mathbf{x}_d$, then
\begin{align*}
  \int\abs{u-w_d}^2\,dx+\int\abs{u-w_d}\,dx\lesssim W_0+1,\qquad \abs{\mathbf{x}-\mathbf{x}_d}\lesssim W_0+1.
\end{align*}
\end{lemma}
\begin{remark}
	Although the solution $u$ of \eqref{ch} may (initially) have \uline{additional} zeros, we will regularly refer to the zeros identified in Lemma~\ref{l:u2zeros} as \uline{the zeros} $\mathbf{x}=\{a,b\}$ of $u$ and we define
\begin{align*}
  \Delt(t) \coloneqq \sup_{0 \leq t' \leq t} \, \abs{\mathbf{x}(t')-\mathbf{x}(0)}.
\end{align*}
\end{remark}

Next, we introduce two tools to control the motion of the zeros: a rough, a priori bound (which will allow us to carry out our duality argument) and a finer bound, which will be useful once we have established smallness of the energy gap.
The proofs are identical for the line and the torus.
We remind the reader of Notation~\ref{not:zeros}.
\begin{lemma}[Rough control of the zeros] \label{l:motionofzeros} 	
	For every $W_0 < \infty$ and $\epsilon >0$, there exists  $L_0<\infty$
	such that for $L \geq L_0$, the following holds true. Suppose that $u$ is a smooth solution of the Cahn-Hilliard equation satisfying \eqref{diffcons2} and that 
	there exists $\wt_0\in\N(L)$ such that for each $t\in[0,T]$ there holds
$\int\abs{u-\tilde{w}_0}\,\dx\leq W_0$.
 Then for  all $0\leq t\leq T$,
	the motion of the zeros of $\wt$  is controlled via
\begin{align} \label{eq:shiftcontrolini}
	\abs{\tilde{\mathbf{x}}(T)-\tilde{\mathbf{x}}(t)} \lesssim
\left(\max_{[t,T]}\Et^{1/2}\right)
(T-t)^{1/4}+1.
	\end{align}
\end{lemma}

\begin{lemma}[Fine control of the zeros]\label{l:fine}
There exist $\eps>0$, $C<\infty$ with the following property. For every $W_0 < \infty$, and $\epsilon >0$, there exists  $L_0<\infty$
such that for $L \geq L_0$, the following holds true. Suppose that $u$ is a smooth solution of the Cahn-Hilliard equation satisfying \eqref{diffcons2} and that on all of $[0,T]$,
there exists $\wt\in\N(L)$ such that
$\int\abs{u-\tilde{w}}\,\dx\leq W_0$. We assume in addition for Problems 1 and 2 that the zeros $a,\,b$ of the solution are absolutely continuous in time and $\Delta \mathbf{\tilde{x}}(T) \lesssim L^{1/2}$.  Then there holds
\begin{align} \label{eq:expwt}
  \Delt(T)+\Deltild(T)\lesssim \Wt_T+\frac{1}{L^{3/2}}\int_0^T \Et(t)^{1/2}\,dt,
\end{align}
for any $T\le \exp(L/C)$, and similarly for $\Wt_T$ and $\Et$ replaced by $\W_T$ and $\E$.
On the torus, if  $\E_0\leq \eps/L$, then the shift is controlled in terms of the energy gap for all $T>0$ via
\begin{align} \label{ccontrol}
	\Delta c(T) \lesssim  L^{1/2}\E(0)^{1/2} + \frac{1}{L^{3/2}}  \int_0^T\E(t)^{1/2}\,\dt.
\end{align}
\end{lemma}

The following Lemmas~\ref{l:nash}--\ref{l:intdissipationbound} appeared already as \cite[Lemmas 3.3--3.5]{OSW}, but for completeness of presentation, we include the (short) proofs in the appendix.
In these lemmas, we speak generally of quantities $\E$, $\D$, $\W$, and $f$; i.e., the reader should forget the earlier definitions and consider only the properties specified in the given lemma.

The first lemma is a Nash-type inequality \cite[Lemma 3.3]{OSW}, which bounds the energy gap in terms of the excess mass and the dissipation. It is similar to the nonlinear estimates used as a first step by Nash in his classical derivation of heat kernel bounds and plays a central role in the $L^1$-framework.
\begin{lemma}[Nash-type inequality] \label{l:nash}
Suppose that the quantities $\E,\D,\W$, and $f\in H^1$ satisfy
\begin{align}  
	\E&\lesssim 1, \label{eq:nashE1} \\
	\int f_x^2\, \dx \leqsim\; \;
\E  & \leqsim \int f^2+f_x^2\,\dx, \label{eq:nashE} \\
	\int f_x^2\, \dx&\leqsim \;\D , \label{eq:nashD} \\
	\W&=\int \abs{f}\,\dx. \label{eq:nashW}
\end{align}
Then there holds
\begin{align*}  
	\E \lesssim \D^{1/3} (\W + 1)^{4/3}.
\end{align*}
\end{lemma}
We fix a time horizon $T<\infty$ and  define
\begin{equation} \label{eq:definitionWT}
\W_T \coloneqq \sup_{t\leq T} \W(t)+1 .
\end{equation}
(We will occasionally use that $\W_T \geq 1$.) A simple ODE lemma controls the energy gap in terms of this quantity.
\begin{lemma} \label{l:odenash}
	Let $T<\infty$ and $\E,\,\D,\,\W:[0,T]\to [0,\infty)$. Suppose that on $[0,T]$, $\E$ and $\D$ satisfy the differential equality
	\begin{align} \label{eq:odenashpre1}
	\ddt \E= -\D
	\end{align}	
	and  the algebraic inequality
	\begin{align} \label{eq:odenashpre2}
	\E \lesssim \D^{1/3}\left( \W + 1 \right)^{4/3}.
	\end{align}	
	Then $\E$ is bounded above by
	\begin{align} \label{eq:odenashresult}
	\E(t) \lesssim \min \left\{\E_0,\frac{\W_T^2 }{t^{1/2}}\right\} \;\; \mbox{for all } t\in [0,T].
	\end{align}
\end{lemma}

Next we introduce an integral bound on the dissipation, which we will need for our results on $L^1$ control.
\begin{lemma}[Integral dissipation bound] \label{l:intdissipationbound}
	Let $T<\infty$ and a constant $\W_T<\infty$ be given, and suppose that  $\E,D:[0,T]\to [0,\infty)$ are related on $[0,T]$ by
	\begin{align} \label{eq:intdisspre}
	\ddt \E= -\D \quad \mbox{and} \quad \E \lesssim \frac{\W_T^2 }{t^{1/2}}.
	\end{align}	
	Then for any $\frac{2}{3} <\gamma \leq 1$ there holds
	\begin{align} \label{eq:intdissipationboundres}
	\int_0^T \D^{\gamma}(t) \, \text{d}t \lesssim \W_T^{4(1-\gamma)}.
	\end{align}
\end{lemma}

Finally, we introduce a differential inequality for the dissipation, and, as a corollary, an upper bound on the dissipation in terms of the energy gap, which we will later use to deduce uniqueness of the zeros. The following lemma is an adaptation of \cite[estimate (1.24) of Lemma 1.4]{OW}.
\begin{lemma}[Differential inequality for the dissipation]\label{l:dissy}
There exist $C<\infty$ and $\eps>0$ with the following property. For every $W_0 < \infty$, and $\epsilon >0$, there exists $L_0<\infty$ such that for $L \geq L_0$, the following holds true. For Problem 1 or 2, suppose that $u$ is a smooth solution of the Cahn-Hilliard equation satisfying \eqref{diffcons2} and such that for all $t\in [0,T]$, there exists $\wt\in\N(L)$ with $\int\abs{u-\wt}\,\dx\leq W_0$ and $\norm{u-\wt}_{L^\infty}\le \eps$. Then there holds
\begin{align} \label{eq:dissyexp}
  \ddt\D\lesssim \D^{3/2}+\exp(-L/C).
\end{align}
For Problem 1, if in addition $\E\leq \eps/L$, then
\begin{align} \label{eq:dissywc}
  \ddt\D\lesssim \D^{3/2}.
\end{align}
\end{lemma}

As in \cite[(1.15), or, more precisely (1.32)]{OW}, the differential inequality for $\D$ is converted (together with $\dot{E}=-D$) into the following upper bound.
\begin{lemma}[Dissipation bound by energy]\label{l:dissy_bound}
Let $T<\infty$ and suppose that  $\E,\,\D:[0,T]\to [0,\infty)$ are related on $[0,T]$ by
\begin{align}
  \ddt \E= -\D \quad \mbox{and} \quad \ddt\D\lesssim \D^{3/2}.\label{assumeddiffq}
\end{align}
Then, for all $t\in (0,T)$, $D$ satisfies the bound
\begin{align*}
	D(t)\lesssim \max\left\{\E^2\left(\frac t2\right),\frac{\E\left(\frac t2\right)}{t}\right\}.
	\end{align*}
\end{lemma}

The proofs of Lemmas \ref{l:nash}--\ref{l:dissy_bound}  are included in the appendix, because they are similar to (albeit in the case of Lemma \ref{l:dissy}, somewhat more complicated than) the proofs of the corresponding results in \cite{OSW} and \cite{OW}.
The other proofs are given in Subsection \ref{ss:prelimpf}.

\section{Problem 1: The large torus}\label{S:2}
On the torus, we will use two slow manifolds: Initially, we will use glued kink profiles $\tilde{w}\in\N(L)$; then, once the energy gap has become sufficiently small, we will ``change horses'' and compare $u$ to shifted bumps $w_c$ (see Subsection~\ref{ss:wc}), which attract the solution in the limit $t\to\infty$.

Initially in time, we will use the Nash-type inequality. The  central challenge for this argument is to control $\Wt_T$ up until a time $T\sim L^2$. This control is derived in the following proposition.
\begin{proposition}[$L^1$ control on the torus] \label{p:L1boundtorus} There exists a $\delta\in (0,1)$ with the following property.
	For every $W_0 < \infty$, $W_1<\infty$, and $\epsilon >0$, there exists  $L_0<\infty$
	such that for $L \geq L_0$ and $T=\delta L^2$, the following holds true. Let $u$ be a smooth solution of \eqref{ch} on the torus with sidelength $L$ and initial data satisfying \eqref{star}. Suppose that on $[0,T]$, $u$ has an $L^2$-closest
glued kink profile
$\wt\in\red{\mathcal{N}(0,L)}$
and that for all $t\in[0,T]$, there holds
\begin{align} \label{eq:torusL1boundassumpt}
\red{\Wt_T\leq W_1,}\quad	\Et(t) \geq 0, \quad \Et \sim \int_{[-L,L]} \ft^2+\ft_x^2\,\dx \quad \mbox{and} \quad \Et \lesssim \frac{\Wt_T^2}{t^{\frac12}}.
\end{align}
Then the excess mass is controlled via
	\begin{align} 
	\Wt_T \lesssim \Wtz + 1.\label{Wtz}
	\end{align}
\end{proposition}
\red{The point is that the derived bound~\eqref{Wtz} is independent of $W_1$. The reader should think of $W_1$ as an artificial bound (that will be used in the buckling argument for the proof of Theorem~\ref{t:torus}).}

 Proposition~\ref{p:L1boundtorus} is an extension of \cite[Proposition~3.6]{OSW}, which was inspired by the duality argument of  Niethammer and Vel\'{a}zquez in their work on existence of self-similar solutions of Smoluchowski’s equation, cf. \cite[Subsection~1.3]{NV}. We include the proof of Proposition~\ref{p:L1boundtorus} in Subsection~\ref{ss:propproof}.

As soon as the solution is close enough to the one-parameter slow manifold $\N_1(L)$ for Lemma~\ref{l:basiceedtorusw} to apply, we will use exponential decay in time of the energy gap. Indeed, if $u_0$ satisfies \eqref{EEDwcA}, then the
Poincar\'{e} inequality
	\begin{align} \label{poincaretorus}
	\int_{[-L,L]} f_c^2\,\dx \lesssim L^2 \int_{[-L,L]} f_{cx}^2\,\dx
	\end{align}
and Lemma~\ref{l:basiceedtorusw} imply
	\begin{align*}
	\E \overset{\eqref{eq:energyestimatewctorus}}\lesssim \int_{[-L,L]} f_c^2 + f_{cx}^2\,\dx \overset{\eqref{poincaretorus}}\lesssim (1+L^2) \int_{[-L,L]} f_{cx}^2\,\dx \overset{\eqref{eq:dissipationestimatewctorus}}\lesssim L^2\D,
	\end{align*}
and an integration in time yields
\begin{align}
	\E\lesssim \E_0\exp\left(-\frac{t}{CL^2}\right)\qquad\text{for $t\geq 0$}. \label{eq:torusexpdecayE}
	\end{align}
Lemma \ref{l:dissy_bound} then yields
\begin{align}
 D\lesssim \E_0\exp\left(-\frac{t}{CL^2}\right)\qquad\text{for $t\geq 1$}. \label{dissyexp2p0}
\end{align}

\begin{proof}[Proof of Theorem~\ref{t:torus}]	

	\underline{Step $1$}:  We begin by restricting to times for which the energy gap is not too small, which we will use in order to be able to absorb the exponential error terms in \eqref{eq:Eerrorlow}--\eqref{eq:Derror}.
We define
\begin{align*}
\hat{T}_1 &\coloneqq
\inf \left\{ t\geq 0 : \E \leq \frac{\eps}{L} \right\},
\end{align*}
where $\eps>0$ is chosen at least as small as the implicit constants in \eqref{EEDwcA} and
\eqref{118b} (where for the second estimate in \eqref{118b}, the energy bound is combined with \eqref{eq:f_c_infty}).
According to \eqref{star} and Lemma \ref{l:u2zeros}, $\wt_0:=\wt(0)$ exists, belongs to $\N(0,L)$, and satisfies $\Wt_0\lesssim W_0+1$.
To be sure of the existence of $\wt$ for future times, we define
\begin{align*}
    \hat{T}_2&\coloneqq\inf\left\{t\geq 0\colon \norm{u(t)-\wt_0}_{L^1({[-L,L]})}\geq C_1(\Wtz+1) \right\},
\end{align*}
where $C_1<\infty$ is a universal constant to be specified below. According to Lemma~\ref{l:u2zeros}, for all $t\leq\hat{T}_2$, the $L^2$ closest glued kink profile $\wt\in\N(0,L)$ exists and satisfies
\begin{align}
  \sup_{t\leq \hat{T}_2} \left(\abs{{\textbf{x}}(t)-{\textbf{x}}(0)}+
  \abs{\tilde{\textbf{x}}(t)-\tilde{\textbf{x}}(0)}\right)\lesssim C_1(\Wtz+1) .\label{C1dex}
\end{align}
A triangle inequality, the definition of $\hat{T}_2$, and properties of $\wt$ give in addition
\begin{align}
\Wt\lesssim C_1(W_0+1).  \label{Vaddition}
\end{align}
One goal in the buckling argument will be to improve from \eqref{C1dex} and \eqref{Vaddition} to  estimates that are independent of $C_1$.
Finally, for an arbitrary $\Tt<\infty$, we define the (finite) time
	 \begin{align*}
	 	T \coloneqq \min \{\Tt, \hat{T}_1, \hat{T}_2\}.
	 \end{align*}

Using the properties of the energy on the slow manifold (cf. Remark \ref{r:expt}), we observe that for all $t, s\in [0,T]$, there holds
\begin{eqnarray}
  \abs{E(\tilde{w}(t))-E(\tilde{w}(s))}+\abs{E(\wt(t))-E(w_c(t))}&\leq&\exp(-(L-\Deltild(T))/C)\notag\\
  &\leq& \exp(-L/C),\label{absorb0}
\end{eqnarray}
where in \eqref{absorb0} we have applied \eqref{C1dex} and increased $C$ to absorb the $W_0$- and $C_1$-dependent exponential factor.
From this estimate
we deduce
\begin{align}
\Et>0,\quad\Et\gtrsim\frac{\eps}{L},\quad\text{and }\; \E(t)\sim\Et(t)\sim \Et(s)\qquad\text{for all $t, s\in[0,T]$.}\label{etoo}
\end{align}
In particular, if $\hat{T}_1=0$, then we can set $T_0=T_1=T_2=0$ and the first part of the theorem is satisfied trivially. Hence we may assume without loss of generality that $\hat{T}_1>0$.

Recalling the properties of $\wt$ and the lower bound from~\eqref{etoo}, we deduce from Lemmas~\ref{l:eedwt} and~\ref{l:nash} that
\begin{align}
 \tilde\E&\lesssim (\D+\exp(-L/C))^{1/3}(\Wt+1)^{4/3}\notag\\
 &\lesssim  \D^{1/3}(\Wt+1)^{4/3}+ \exp(-L/C)(\Wt+1)^{4/3}\qquad\text{on all of }\,[0,T]. \label{EDWpuretorus1}
\end{align}

Recalling \eqref{Vaddition} and \eqref{etoo}, we estimate the second term by
\begin{align*}
  \exp(-L/C)(\Wt+1)^{4/3}\overset{\eqref{Vaddition}}\lesssim\exp(-L/C)\left(C_1(W_0+1)\right)^{4/3}\overset{\eqref{etoo}}\ll \Et
\end{align*}
whence \eqref{EDWpuretorus1} improves to
\begin{align}
 \tilde\E\lesssim \D^{1/3}(\Wt+1)^{4/3}\qquad\text{on all of }\,[0,T]. \label{EDWpuretorus}
\end{align}

Next we introduce the fixed-in-time function $\tilde{w}(T)$, so that
\begin{align}
  \bar\E:=E(u)-E(\tilde{w}(T))\quad\text{satisfies}\quad \ddt\bar\E=-D\quad\text{and}\quad \bar\E\geq 0,\;\;\bar\E(t)\sim \tilde{\E}(t),\label{Ebartorus}
\end{align}
where for the inequality we have used monotonicity of the energy and the positivity of $\tilde\E(T)$ from \eqref{etoo}.
Since
\begin{align}
  \bar\E\overset{\eqref{Ebartorus}}\lesssim\tilde\E\overset{\eqref{EDWpuretorus}}\lesssim \D^{1/3}(\Wt+1)^{4/3},\label{nashlocal}
\end{align}
an application of Lemma~\ref{l:odenash} yields
\begin{align}
  \bar{\E}\lesssim\frac{\Wt_{T}^2}{t^{1/2}}\qquad\text{and hence}\qquad  \Et\overset{\eqref{Ebartorus}}\lesssim\frac{\Wt_{T}^2}{t^{1/2}}\qquad\text{for all }\quad t\leq T.\label{firstbd}
\end{align}
We now have a two-part argument. Initially and up to times $t\leq \min\{T,\delta L^2\}$ (where $\delta$ is the constant from Proposition~\ref{p:L1boundtorus}), we combine~\red{\eqref{Vaddition} and} \eqref{firstbd} with Proposition~\ref{p:L1boundtorus}
to obtain
\begin{align}
\Wt_{T}\lesssim \Wtz+1,\qquad\text{so that }\qquad\Et\lesssim\frac{\Wtz^2+1}{t^{1/2}}\qquad \text{for }t\leq \min\{T,\delta L^2\}.\label{univ3torus}
\end{align}

We would now like to deduce
\begin{align}\label{eq:zerocontrol}
\Delt+\Deltild\lesssim \Wtz+1.
\end{align}
Initially and until $\Et\ll 1$ (according to the smallness desired in \eqref{eq0}), we use Lemma \ref{l:motionofzeros}. We observe that \eqref{univ3torus} implies this smallness is achieved at a time at most order $\Wtz^4+1$.
This defines the time $T_0$ announced in the theorem.

Notice that smallness of the energy gap on $[T_0,T]$ and \eqref{eq:basiceedconsequence} give $\norm{u-\wt}_{L^\infty}\ll 1$ on $[T_0,T]$. We want also smallness of the dissipation, from which we will infer that the zeros are unique and, hence, move continuously. For this, we start with the coarse lower bound on the dissipation
\begin{align*}
  D\overset{\eqref{nashlocal}}\gtrsim \frac{\Et^3}{(\Wt+1)^4}\overset{\eqref{etoo},\eqref{Vaddition}}\gtrsim \frac{\eps^3}{L^3C_1^4(W_0+1)^4}\gg\exp(-L/C),
\end{align*}
where we have used that $\eps$ and $C_1$ are universal constants and that $L_0$ is chosen sufficiently large with respect to $W_0$.
Substituting this lower bound into \eqref{eq:dissyexp} for $t\leq \hat{T}_2$ yields the differential inequality
\begin{align*}
  \ddt D\lesssim D^{3/2},
\end{align*}
and Lemma \ref{l:dissy_bound} gives smallness in terms of $\Et$.
Thus we obtain continuity of the zeros. Using the rough bound
\eqref{C1dex} and continuity to justify an application of Lemma \ref{l:fine}, we obtain from the lemma and the energy gap estimate \eqref{univ3torus}  the uniform control \eqref{eq:zerocontrol} up until $\min\{T,\delta L^2\}$.

Finally, we argue that the constraint from $\hat{T}_2$ is slack:
Using
\eqref{eq:zerocontrol} and properties of the $\wt$, we deduce
\begin{align}
  \int_{[-L,L]}\abs{u-\wt_0}\,dx\lesssim \int_{[-L,L]}\abs{u-\wt}\,dx+\int_{[-L,L]}\abs{\wt-\wt_0}\,dx
  \overset{\eqref{univ3torus},\eqref{eq:zerocontrol}}\lesssim \Wtz+1.\label{littleC2}
\end{align}
Therefore choosing the constant $C_1$ in the definition of $\hat{T}_2$ large enough with respect to the implicit constant in this bound gives $\hat{T}_2>\min\{T,\delta L^2\}$.

We now turn to the second part of the two-part argument. If $T\leq \delta L^2$, then we set $T_1=T_2=T$ and move to Step 2, but if $T>\delta L^2$, then we set $T_1=\delta L^2$ and we need a separate estimate on $[T_1,T]$.
\red{We begin by deducing from \eqref{univ3torus} that
\begin{align}
  \Et(T_1)\lesssim \frac{\Wtz^2+1}{L}.\label{WoL}
\end{align}}
We will now use the exponential decay of $\Et$. Indeed, invoking Lemma \ref{l:EDtorus} and recalling \eqref{etoo} and \eqref{Ebartorus}, we observe that
\begin{align*}
  \ddt\bar\E=-D\quad\text{and}\quad \bar\E\lesssim L^2 D.
\end{align*}
An integration in time (together with \eqref{Ebartorus}) yields
\begin{eqnarray}
\Et(t)&\lesssim& \Et(T_1)\exp\left({-(t-T_1)/C L^2}\right)\notag\\
&\overset{\eqref{WoL}}\lesssim &\frac{\Wtz^2+1}{L}\exp\left({-(t-T_1)/C L^2}\right) \qquad\text{ for $t\in[T_1,T]$},\label{biggy}
\end{eqnarray}
and Lemma \ref{l:dissy_bound} provides the bound on the dissipation. Inserting \eqref{biggy} into Lemma \ref{l:fine} implies \eqref{eq:zerocontrol} on $[T_1,T]$.
For $\Wt$ we use the Cauchy-Schwarz inequality to estimate
\begin{align}
  \Wt\leq L^{1/2}\left(\int_{[-L,L]}(u-\wt)^2\,dx\right)^{1/2}
  \lesssim  L^{1/2}\Et^{1/2}\overset{\eqref{biggy}}\lesssim W_0+1,\notag
\end{align}
and arguing as in \eqref{littleC2} implies $\hat{T}_2>T$ for $C_1$ large enough.
From the energetic decay and by choosing $\tilde{T}$ large enough, we conclude
\begin{align}
\E(T) =\frac{\eps}{L} \qquad \text{and}\qquad T\lesssim L^2\log\Wtz,\label{ETB}
\end{align}
and we set $T_2:=T$.
We now proceed to the next step, in which the attraction to the one-parameter slow manifold takes over.



\uline{Step 2}:	Now that we want to consider the one-parameter slow manifold, it makes sense to start by considering what has transpired in relation to $\N_1(L)$ up until time $T_2$.
From Lemma~\ref{l:u2zeros}, $\Wt_{T_2}\lesssim W_0+1$, and \eqref{eq:zerocontrol} on $[0,T_2]$, we obtain
\begin{align*}
\W\lesssim W_0+1,\quad \text{and }\quad \Delta c\lesssim W_0+1\qquad\text{on }[0,T_2].
\end{align*}
Next it remains to continue beyond $T_2$.
Because of the first part of \eqref{ETB}, we can apply \eqref{eq:torusexpdecayE} with initial time $T_2$
and \eqref{expy} follows. In addition, EED (from Lemma~\ref{l:basiceedtorusw}) and the Cauchy-Schwarz inequality give control of $\W$ for all $t\geq T_2$, concluding the proof of the first part of \eqref{slowdowntorus}. Using \eqref{expy} in Lemma \ref{l:fine} concludes the proof of the third part of \eqref{slowdowntorus}, and this together with a triangle inequality and Lemma~\ref{l:u2zeros} completes the proof of the second item in \eqref{slowdowntorus}. The estimate \eqref{expylinf} follows in the usual way from~\eqref{expy} and~\eqref{eq:f_c_infty}.

It remains only to establish exponential convergence of the shift, for which we again turn to Lemma~\ref{l:fine}. Using \eqref{eq:torusexpdecayE} with initial time $s$ for $t>s\geq T_2$, we estimate
\begin{eqnarray*}
	\abs{c(t)-c(s)} & \overset{\eqref{ccontrol}}\lesssim &
	L^{1/2}\E(s)^{1/2} + \frac{1}{L^{3/2}}  \int_s^t\E(\tau)^{1/2}\,d\tau\\
&\overset{\eqref{eq:torusexpdecayE}}\lesssim&
L^{1/2}\E(s)^{1/2} +\frac{1}{L^{3/2}} \int_s^t\E(s)^{1/2}\exp\left(-\frac{\tau-s}{CL^2}\right)^{1/2}\,d\tau\\
&\lesssim & L^{1/2}\E(s)^{1/2}\\
& \overset{\eqref{expy}}\lesssim &\exp\left(-\frac{s-T_2}{CL^2}\right),
\end{eqnarray*}	
which suffices.

\end{proof}

\section{Problem 2: The bump on the real line} \label{S:1}
We begin by collecting the supporting results that we will use for the proof of Theorem~\ref{t:mainresultuniversal}.
Because Lemma~\ref{l:odenash} yields the decay of the energy gap only in terms of the supremum of $\Wt$, the important question for our main result is whether $\Wt$ in fact remains bounded.
\begin{proposition}[$L^1$ control on the line] \label{p:L1bound}There exists $\delta\in(0,1)$ with the following property.
	For every $W_0 < \infty$, $W_1<\infty$, and $\epsilon >0$, there exists  $L_0<\infty$
	such that for $L \geq L_0$, the following holds true.  Let $u$ be a smooth solution of \eqref{ch} on the line with initial data satisfying \eqref{e0bd} and \eqref{W_uni}, and suppose there is a $\wt\in\red{\mathcal{N}(0,L)}$ that is the $L^2$-closest glued kink profile. Suppose moreover that for all $t\in[0,T]$, there holds
	\begin{align*}
		\red{\Wt_T\leq W_1,}\quad	\Et\geq 0,\quad \Et\sim \int_\R \ft^2+\ft_x^2\,\dx, \quad \mbox{and} \quad 
		\Et \lesssim \frac{\Wt_T^2}{t^{1/2}}. 
	\end{align*}
	\begin{enumerate}
\item If $T\leq \delta L^2$ or
	
\item if $T\lesssim L^2$ and $\Et\lesssim (\Wtz^2+1)/L$, then

\end{enumerate}
\begin{align*} 
	\Wt_T \lesssim \Wtz + 1.
	\end{align*}

\end{proposition}
We also obtain $L^1$ control on the distance to $\bar{u}\equiv -1$ when the energy is below $2\co$.
\begin{proposition}[$L^1$ control on the line compared to $\bar{u}\equiv -1$] \label{p:L1boundsmallE}
Let $\epsilon >0$ and suppose that $u$ is a smooth solution of \eqref{ch} on the line with initial data satisfying $E(u(0))\le 2\co-\epsilon$ and $\Wm (0)<\infty$.
Then for all $t>0$, there holds
\begin{align}\label{eq:W-}
\Wm (t)\lesssim \Wm (0)+1.
\end{align}
\end{proposition}

The proof of Theorem~\ref{t:mainresultuniversal} relies on EED, the Nash-type decay, $L^1$ control, and a buckling argument.
\begin{proof}[Proof of Theorem~\ref{t:mainresultuniversal}]
Our goal is to control the evolution until the energy gap has become algebraically small, and the proof is similar to the first part of the proof of Theorem~\ref{t:torus}. We define
\begin{align*}
  \hat{T}_1:=\inf\left\{t\geq 0\colon E(u(t))-2e_*\leq \frac{\eps(\Wtz^2+1)}{L}\right\} 
\end{align*}
for $\eps>0$ arbitrarily small. If $\hat{T}_1=0$, then we can set $T_1=0$ and the theorem is satisfied trivially. Hence we may assume without loss that $\hat{T}_1>0$.
Next we set
\begin{align*}
  \hat{T}_2:=\inf\left\{t\geq 0\colon \norm{u(t)-\wt_0}_{L^1({\R})}\geq C_1(\Wtz+1) \right\},
\end{align*}
for a universal constant $C_1$ to be specified below.
Finally, we set
\begin{align*}
  T:=\min\{\hat{T}_1,\hat{T}_2,\delta L^2\},
\end{align*}
where $\delta$ is the constant from Proposition~\ref{p:L1bound}.

Arguing exactly as in the proof of Theorem~\ref{t:torus} yields
\begin{align}
  \Et\lesssim\frac{\Wt_{T}^2}{t^{1/2}}\qquad\text{for all }\quad t\leq T.\notag
\end{align}
In addition, from Proposition~\ref{p:L1bound} (the first part) and Lemmas~\ref{l:motionofzeros} and~\ref{l:fine},  we obtain
\begin{align}
\Wt_{T}\lesssim \Wtz+1,\qquad\Et\lesssim\frac{\Wtz^2+1}{t^{1/2}},\qquad\text{and}\qquad \Deltild\lesssim \Wtz+1.\label{univ3}
\end{align}
Using the first and third bound above and choosing the constant $C_1$ in the definition of $\hat{T}_2$ large enough, we obtain $T<\hat{T}_2$.
From the energetic decay, we can already conclude the existence of the time $T_0$ from the statement of the theorem and also that
\[\abs{E(u(T))-2e_*}\lesssim\frac{\Wtz^2+1}{\sqrt{\delta}L}
\lesssim\frac{\Wtz^2+1}{L},\]
since $\delta$ is a universal constant.

Without loss of generality, we may assume that $T<\hat{T}_1$.
To continue beyond time $\delta L^2$,
we set
\begin{align*}
  S:=\min\{\hat{T}_1,\hat{T}_2,CL^2\},
\end{align*}
for any fixed constant $C<\infty$.
Using $\hat{T}_2$ to control $\Deltild$ roughly as in \eqref{C1dex}, properties of $\wt$, and monotonicity of the energy, we observe that
\begin{align*}
  \Et\lesssim \frac{\Wtz^2+1}{L}\qquad \text{for all }t\in [\delta L^2,S].
\end{align*}
Hence we may argue as above, this time using the second condition in Proposition~\ref{p:L1bound}, from which we deduce \eqref{univ3} for all $t\in[\delta L^2,S]$. As usual, this yields
\begin{align*}
  \Et(S)\leq \frac{\eps ( W_0^2+1)}{L}, \qquad S\lesssim L^2
\end{align*}
for $C$ large enough with respect to $\eps$, and we set $T_1:=S$.

Finally, Lemma~\ref{l:motionofzeros} gives~\eqref{da} and also implies that the algebraic closeness to the slow manifold persists at least until $t-T_1\gtrsim L^2/( W_0^4+1)$.

\end{proof}

For the convergence to $\bar{u}\equiv -1$, the argument is much simpler.
\begin{proof}[Proof of Theorem~\ref{t:mainresultubar}]
Here no buckling is necessary and a direct application of Proposition~\ref{p:L1boundsmallE}, Lemma~\ref{l:basiceed-}, Lemma~\ref{l:nash}, and Lemma~\ref{l:odenash} yields the result.
\end{proof}

\section{Auxiliary proofs} \label{S:aux}
The proof of Lemma \ref{l:EDtorus} and of the smaller auxiliary results from Section~\ref{s:mainresults}  are given in Subsection~\ref{ss:prelimpf}.
Then in Subsection~\ref{ss:propproof} we establish the more involved Propositions~\ref{p:L1boundtorus} and~\ref{p:L1bound}.

\subsection{Proofs from Subsections~\ref{ss:eed}-\ref{ss:prelim}}\label{ss:prelimpf}

\begin{proof}[Proof of Lemma \ref{l:EDtorus}]
We use
\begin{align*}
\int_{[-L,0]} (u-v_{\tilde a})v_{\tilde ax}\,\dx =0
\end{align*}
and Lemma \ref{l:hardy} to estimate
\begin{align*}
\frac{1}{L^2+1}\int_{[-L,0]} (u-v_{\tilde a})^2\,\dx \lesssim \int_{[-L,0]} \frac{1}{1+(x-\tilde a)^2}(u-v_{\tilde a})^2\,\dx\lesssim \int_{[-L,0]} (u-v_{\tilde a})_x^2\,\dx.
\end{align*}
Combining this with the corresponding estimate on $[0,L]$ and the form of $\wt$, we obtain
\begin{align*}
\int_{[-L,L]} \ft^2\,\dx \lesssim L^2\int_{[-L,L]} \ft_x^2\,\dx+\exp(-L/C).
\end{align*}
The EED estimates from Lemma \ref{l:eedwt} then yield
\begin{align*}
\Et\lesssim L^2 D+\exp(-L/C).
\end{align*}
\end{proof}

\begin{proof}[Proof of Lemma~\ref{l:u2zeros}]

	We begin with the result on $\R$.

	\underline{Step 1}: Without loss we consider the first zero $\tilde{a}_0$ of $\wt_0$. It suffices to show that $u$ has a zero in a neighborhood of $\tilde{a}_0$ of size $R = C \max \, {\{ W_0, W_0^{\frac12} \}}$, where $C$ is a universal constant. By choosing $L_0$ sufficiently large with respect to $W_0$, we may assume that
	\begin{align*}
		\tilde{a}_0+3R\leq \frac{\tilde{a}_0 + \tilde{b}_0}{2}.
	\end{align*}
	
	We assume for a contradiction that $u$ does not have a zero in $(\tilde{a}_0-3R,\tilde{a}_0+3R)$. Then $u$ has a fixed sign on that interval; we assume without loss that the sign is negative. We observe that $\wt_0 > 0$ on $(\tilde{a}_0, \tilde{b}_0)$. 
	Since $u$ is negative on $(\tilde{a}_0-3R,\tilde{a}_0 + 3R)$, there holds
	\begin{align*} 
		\int_{\tilde{a}_0+R}^{\tilde{a}_0+2R} \wt_0 \, \dx \leq \int_{\tilde{a}_0}^{\tilde{a}_0+3R} (\wt_0 - u) \, \dx {\leq W_0}.
	\end{align*}
	Let $x_0$ denote the point in  $(\tilde{a}_0, \frac{1}{2}(\tilde{a}_0+\tilde{b}_0))$ such that $\wt_0(x_0) = \half$. We now consider the two cases: (i) $x_0> \tilde{a}_0+\frac{3}{2}R$ and (ii) $x_0\leq \tilde{a}_0+\frac{3}{2}R$. In case (i), we use $\wt_0(x)\gtrsim x-\tilde{a}_0$ on $(\tilde{a}_0+R,x_0)$ to estimate
	\begin{align*}
		\int_{\tilde{a}_0+R}^{\tilde{a}_0+2R} \wt_0 \, \dx \geq \int_{\tilde{a}_0+R}^{\min\{\tilde{a}_0+2R, x_0\}} \wt_0 \, \dx \gtrsim R^2.
	\end{align*}
	In case (ii), we use the monotonicity of $w$ on $( \tilde{a}_0+R,\frac{1}{2}(\tilde{a}_0+\tilde{b}_0))$ to deduce
	\begin{align*}
		\int_{\tilde{a}_0+R}^{\tilde{a}_0+2R} \wt_0 \, \dx \geq \int_{\max\left\{x_0,\tilde{a}_0 +R\right\}}^{\tilde{a}_0+\frac{3}{2}R} \wt_0 \, \dx \gtrsim R.
	\end{align*}
	Therefore $R \lesssim \max \, \{W_0,{W_0^{\frac{1}{2}}}\}$, which contradicts the definition of $R$ for $C$ large enough.
	
	\underline{Step 2}: For the estimates with respect to $\wt$, we use the properties of the glued kink profiles and $L^2$  closeness to $v_{\tilde{a}}$ 
	to estimate
	\begin{eqnarray}
		\lefteqn{\int_{\R} (u-\wt)^2\,\dx} \nonumber \\
		& \lesssim & \int_{\R^-} (u-v_{\tilde{a}})^2\,\dx + \int_{{\R^+}} (u+v_{\tilde{b}})^2\,\dx + 1 \leq \int_{{\R}_-} (u-v_{a_0})^2\,\dx + \int_{{\R^+}} (u+v_{b_0})^2\,\dx +1 \nonumber \\
		& \lesssim &  \int_{\R} (u-\wt_0)^2\,\dx + 1 \lesssim  W_0+1, \nonumber
	\end{eqnarray}
where we have used the $L^\infty$ bound on $u$ and $\wt_0$ .
An argument similar to Step 1 then yields
	\begin{align*}
		\abs{\mathbf{x}-\tilde{\mathbf{x}}} \lesssim W_0+ 1,
	\end{align*}
from which we also deduce $L^1$ control (via a triangle inequality and properties of the glued kink profiles).
This completes the proof on $\R$.

	The proof of the analogous result on the torus (with respect to $\wt$) is almost the same, except that one has to choose $L_0$ large enough with respect to $W_0$ so that
	\begin{align*}
		\tilde{a}+3 R\leq \frac{\tilde{a}+\tilde{b}}{2} \qquad\text{and also}\qquad  \tilde{a}-3R\geq \frac{-2L+\tilde{b}+\tilde{a}}{2},
	\end{align*}
	in order to keep away from both midpoints (bearing periodicity in mind).
	
Now suppose that \eqref{onew} holds on the torus and consider the distance to the $L^2$-closest bump $w_c$ with zeros $\textbf{x}_c=\{a_c,b_c\}$.
Let $\ell$ and $\tilde{\ell}_0$ denote the distances between the zeros of $w_c$ and $\wt_0$, respectively. We shift $w_c$ so that it shares its left zero with that of $\wt_0$ and denote the shifted bump by $w_d$. There holds
	\begin{align*}
			\abs{\ell-\tilde{\ell}_0} \sim \left\vert \int_{[-L,L]} \wt_0-w_d\,\dx \right\vert =  \left\vert \int_{[-L,L]} \wt_0-u\,\dx \right\vert
		\lesssim  W_0,
	\end{align*}
where we have used the conservation law $0=\int_{[-L,L]} u-w_c\,dx=\int_{[-L,L]}u-w_d\,dx$.
From here we deduce
\begin{align*}
w_c\in\N_1(L)\qquad\text{and}\qquad	\int_{[-L,L]} (\wt_0-w_d)^2\,\dx \lesssim W_0+1,
\end{align*}
and hence, with a triangle inequality, we obtain
	\begin{align*}
		\int_{[-L,L]} (u-w_c)^2\,\dx \leq \int_{[-L,L]} (u-w_d)^2\,\dx
\lesssim W_0+1.
	\end{align*}
The $L^1$ bound and the bound on $\abs{\mathbf{x}-\mathbf{x}_c}$ follow as usual.
A similar (but simpler) argument yields the result when $\wt_0$  in \eqref{onew} is replaced by some bump $w_d\in\N_1(L)$.
\end{proof}

\begin{proof}[Proof of Lemma~\ref{l:motionofzeros}]
	We will give the proof on the line. The modification for the torus presents no new difficulties.

According to the assumption about $\wt_0$ and Lemma~\ref{l:u2zeros}, there exists for all $t\in[0,T]$ an $L^2$-closest profile $\wt=\wt(t)\in \mathcal{N}(0,L)$. We will localize around one of the zeros of $\wt$. To fix ideas, assume without loss of generality that the larger shift is in the zero $\tilde{a}$ of $\wt$.  According to Lemma~\ref{l:u2zeros}, there holds
\begin{align}
  \abs{\at(T)-\at(t)}\lesssim W_0+1,\label{assumL}
\end{align}
and our goal is to establish \ref{eq:shiftcontrolini} (in which there is no $W_0$-dependence).

Using \eqref{assumL}, we will estimate
	\begin{align} \label{zerosviawt}
		\abs{\tilde{a}(T)-\tilde{a}(t)} \lesssim \left|\int_{\R}(\wt(T)-\wt(t))\varphi\,\dx\right| +1,
	\end{align}
for a cut-off function $\varphi:\R\to [0,1]$ such that
	\begin{align*}
	\varphi&=\begin{cases}  1 & \mbox{ on } \left[ \tilde{a}(T)-\tilde{L},\tilde{a}(T)+\tilde{L}\right], \\  0 & \mbox{ on } [\tilde{a}(T)-2\tilde{L}, \tilde{a}(T)+2\tilde{L}]^c, \end{cases} \quad
	\mbox{and} \quad \abs{\varphi_{x}}\lesssim \frac{1}{\tilde{L}},\,\abs{\varphi_{xx}}\lesssim \frac{1}{\tilde{L}^2}.
	\end{align*}
The estimate \eqref{zerosviawt} is valid as long as
\begin{align}
\max\{1,|a(t)-a(T)|\}\ll \tilde{L}\ll L.
\label{needs}
\end{align}
We use the triangle inequality to estimate
\begin{align}
 \lefteqn{ \left|\int_{\R}(\wt(T)-\wt(t))\varphi\,\dx\right|}\notag\\
 &\leq \left|\int_{\R}(\wt(T)-u(T))\varphi\,\dx\right|+
  \left|\int_{\R}(u(T)-u(t))\varphi\,\dx\right|+
  \left|\int_{\R}(u(t)-\wt(t))\varphi\,\dx\right|.\label{bigtri}
\end{align}
For the first and last terms, we use the bounded support of $\varphi$ to estimate
\begin{eqnarray}
  \left|\int_{\R}(\wt(T)-u(T))\varphi\,\dx\right|+
  \left|\int_{\R}(u(t)-\wt(t))\varphi\,\dx\right|&\overset{\eqref{eq:Eerrorlow}}\lesssim & \tilde{L}^{1/2}(\abs{\Et(t)}^{1/2}+\abs{\Et(T)}^{1/2}).\label{sumtf}
\end{eqnarray}
For the middle term on the right-hand side of \eqref{bigtri}, we use duality,~\eqref{diff}, and the definition of $\varphi$ to estimate
	\begin{eqnarray}
	\left|\int_{\R} (u(T) - u(t))\varphi \,\dx\right| &\leq&\left(\int \varphi_{x}^2\,\dx\right)^{1/2} \norm{u(T)- u(t)}_{\Hpkt(\R)}\nonumber \\
	&\overset{\eqref{diff}}\lesssim&\left(\frac{1}{\tilde{L}}\right)^{1/2} \int_t^T \D^{1/2}\,\rm{d}s\notag\\
&\lesssim&
	 \abs{E(u(T))-E(u(t))}^{1/2} \left(\frac{T-t}{\tilde{L}}\right)^{1/2}\notag\\
&\lesssim & \left(\abs{\Et(t)}^{1/2}+\abs{\Et(T)}^{1/2}+\exp(-L/C)\right)\left(\frac{T-t}
{\tilde{L}}\right)^{1/2},\label{58simpl}
	\end{eqnarray}
where we have used \eqref{assumL} and the properties of $\wt$.
Substituting \eqref{sumtf} and \eqref{58simpl} into \eqref{bigtri} and then into \eqref{zerosviawt} gives
\begin{align*}
  \abs{\at(T)-\at(t)}\lesssim \max_{s\in[t,T]}\abs{\Et(s)}^{1/2}\left(\tilde{L}^{1/2}
  +\left(\frac{T-t}{\tilde{L}}\right)^{1/2}\right)+1.
\end{align*}
Optimizing in $\tilde{L}$ under the constraint \eqref{needs} (and recalling $\Et\lesssim 1$) gives \eqref{eq:shiftcontrolini}.
\end{proof}

\begin{proof}[Proof of Lemma~\ref{l:fine}]
\underline{Step 1:} Proof of \eqref{eq:expwt}. We estimate $\abs{\tilde a(T)-\tilde a(0)}$. The estimates for $\tilde b$ are identical and the estimate for $\Delt$  follows by the triangle inequality and Lemma~\ref{l:u2zeros}. The starting point are the equations
\begin{align*}
	u_t + [u_{xx}-G'(u)]_{xx} & = 0,\\
	-v_{\tilde{a}xx} + G'(v_{\tilde{a}}) &= 0, 
	\\
	\mbox{and} \quad v_{\tilde{a} t} &= -\tilde{a}_tv_{\tilde{a}x}, 
\end{align*}
which we use to write
\begin{align} \label{eq:uvtildea}
	(u-v_{\tilde{a}})_t - \tilde{a}_tv_{\tilde{a}x} - \Big( -(u-v_{\tilde{a}})_{xx} + G'(v_{\tilde{a}} + (u-v_{\tilde{a}}))-G'(v_{\tilde{a}}) \Big)_{xx}=0.
\end{align}	
We assume without loss $\tilde{a}(0)\leq \tilde{a}(T)$. We define a test function $\varphi:I_-\to [0,1]$ that satisfies
	\begin{align*}
		\varphi \equiv 1 \mbox{ on }
		I_- \setminus \left(\tilde{a}(0)+\frac{\tilde L}{2},0\right) \quad \mbox{and} \quad \varphi \equiv 0 \mbox{ on } \left(\tilde{a}(0)+\tilde L,0\right),
	\end{align*}
where $\tilde L\sim L$ is small enough so that $\tilde a(0)+\tilde L<0$, and
	\begin{equation} \label{etaL2}
		 \norm{\varphi_{xx}}_{L^2(I_-)} \lesssim \frac{1}{L^{3/2}},  \quad \mbox{and} \quad \norm{\varphi_{xxx{x}}}_{L^2(I_-)} \lesssim \frac{1}{L^{7/2}}.
	\end{equation}
Multiplying \eqref{eq:uvtildea} with $\varphi$, integrating, and rearranging terms leads to
\begin{eqnarray}
	\lefteqn{\int_0^T\int_{I_-}   \tilde{a}_tv_{\tilde{a}x}\varphi \,\dx\,\dt
	 =  \int_0^T\int_{I_-}	(u-v_{\tilde{a}})_{t}\varphi\,\dx\,\dt } \nonumber \\
&&\quad-  \int_0^T\int_{I_-}\Big( -(u-v_{\tilde{a}})_{xx} + G'(v_{\tilde{a}}
+ (u-v_{\tilde{a}}))-G'(v_{\tilde{a}}) \Big)_{xx}\varphi\,\dx \,\dt.  \label{eq:uvtildea2}
\end{eqnarray}
Because of the assumption $\Delta \mathbf{\tilde{x}} (T) \lesssim L^{1/2}$,  there holds $\int_{I_-}  v_{\tilde{a}x}\varphi \,\dx \approx 1$, and the left-hand side of \eqref{eq:uvtildea2} can be approximated as
 \begin{align*}
 	\int_0^T\int_{I_-} \tilde{a}_tv_{\tilde{a}x}\varphi \,\dx\,\dt = 	\int_0^T\tilde{a}_t \int_{I_-} v_{\tilde{a}x}\varphi \,\dx\,\dt \approx \int_0^T \tilde{a}_t\,\dt = 	\tilde{a}(T)-\tilde{a}(0).
 \end{align*}
We substitute this into the left-hand side of \eqref{eq:uvtildea2} and estimate
\begin{eqnarray}
\lefteqn{	\tilde{a}(T) - \tilde{a}(0) } \nonumber \\
		&\approx & \int_0^T \int_{I_-} (u-v_{\tilde{a}})_{t}\varphi \,\dx \,\dt + \int_0^T\int_{I_-}  \Big( - (u-v_{\tilde{a}})_{xx}+G''(v_{\tilde{a}})(u-v_{\tilde a}) \Big)_{xx}\varphi  \,\dx \,\dt\nonumber \\
		&\lesssim & \left\vert\int_{I_-} (u(T)-v_{\tilde{a}}(T))\varphi \,\dx\right\vert+\left\vert\int_{I_-} (u(0)-v_{\tilde{a}}(0))\varphi \,\dx\right\vert  \nonumber \\
&&+ \left\vert
 \int_0^T\int_{I_-}  (u-v_{\tilde{a}})_{xx}\varphi_{xx} \,\dx\,\dt \right\vert
		+ \sup_{\abs{\tau} \leq 1} G''(\tau)\int_0^T\int_{I_-}\abs{u-v_{\tilde{a}}} \abs{\varphi_{xx}}   \,\dx \,\dt \label{cutshift} \\
		&\overset{\eqref{etaL2}}\lesssim &  \Wt_T + \frac{1}{L^{3/2}}  \int_0^T\left(\int_{I_-}  (u-v_{\tilde{a}})^2\,\dx\right)^{1/2}\,\dt \lesssim \Wt_T + \frac{1}{L^{3/2}}  \int_0^T\left(\int_{I_-} \ft^2\,\dx\right)^{1/2}\,\dt, \nonumber
\end{eqnarray}
where in the last step we used $T \exp(-L/C)\lesssim 1\leq\Wt_T$ to absorb the integrated exponential error. Next we employ \eqref{eq:Eerrorlow} to estimate the inner integral and  absorb the error in the same way. Combining the estimates for $\tilde a$ and $\tilde b$ yields \eqref{eq:expwt}.

\underline{Step 2:} Proof of \eqref{ccontrol}. Proceeding analogously and applying the EED estimates from Lemma~\ref{l:basiceedtorusw} yields
\begin{eqnarray}
	\lefteqn{	\abs{c(T) - c(0)} } \nonumber \\
	&\overset{\eqref{cutshift}}\lesssim & \left\vert\int_{I_-} (u(T)-w_c(T))\varphi \,\dx\right\vert + \left\vert\int_{I_-} (u(0)-w_c(0))\varphi \,\dx\right\vert  + \left\vert \int_0^T\int_{I_-}  (u-w_c)_{xx}\varphi_{xx} \,\dx\,\dt \right\vert \nonumber \\
	& &+ \sup_{\abs{\tau} \leq 1} G''(\tau)\int_0^T\int_{I_-}\abs{u-w_c} \abs{\varphi_{xx}}   \,\dx \,\dt \nonumber \\
	&\overset{\eqref{etaL2}}\lesssim &  L^{1/2}\E(0)^{1/2} + \frac{1}{L^{3/2}}  \int_0^T\E(t)^{1/2}\,\dt. \nonumber
\end{eqnarray}
\end{proof}

\subsection{Proofs of control of the excess mass}\label{ss:propproof}

We use a duality argument in order to prove Propositions~\ref{p:L1boundtorus},~\ref{p:L1bound}, and \ref{p:L1boundsmallE}.

In order to carry out the duality argument, we will need semigroup estimates for the dual of the linearized operator $\partial_t -G''(1) \partial_x^2 + \partial_x^4$ on a variable domain. In \cite{OSW} the authors use Schauder theory to treat the motion perturbatively. Motion of the boundary at rate $t^{1/2}$ (on large temporal scales) would be critical; in our application on the torus, as in \cite{OSW}, it is sufficient to take $t^{1/4}$.

For our decay estimates, it will be convenient to introduce the following notation.
\begin{notation}
	We use the following notation for two quantities $A$ and $B$:
	\[ A \vee B = \max \{ A, B \} \quad \mbox{and} \quad A \wedge B = \min \{ A, B\}.  \]
\end{notation}
We recall the variable domain estimates from \cite{OSW}.
\begin{proposition}[{\cite[Proposition~3.1]{OSW}}] \label{p:schauderest}
	Let $T<\infty$. For every $C_1 \in(0,\infty)$ there exists a constant $\Lambda\in [1,\infty)$ large with respect to $C_1^4$ such that the following holds true. Consider the curve
	\begin{align*}
	\gamma(t) \coloneqq \tilde{a}(T) - C_1(T-t+\Lambda)^{1/4}\;\;\mbox{for all } t\in[0,T].
	\end{align*}
	Let $\zeta$ satisfy the backwards problem
	\begin{equation} \label{eq:zetapde}
	\begin{cases}
	\zeta_t + G''(1) \zeta_{xx} - \zeta_{xxxx} = 0 & \mbox{ on } \; t \in [0,T), \; x\in(-\infty,\gamma(t)),  \\
	\zeta = \zeta_{xx} = 0 & \mbox{ for } \; t \in [0,T), \; x=\gamma(t),  \\
	\zeta = \psi & \mbox{ for } \; t=T, \; x\in (-\infty,0],     \end{cases}
	\end{equation}
	where $\psi$ satisfies $\norm{\psi}_{L^{\infty}} \leq 1$. There holds
	\begin{eqnarray}
	\norm{\zeta}_{L^{\infty}} & \lesssim &   1   \label{eq:zeta} \\
	\norm{\zeta_{x}}_{L^{\infty}} & \lesssim & \frac{1}{(T-t)^{1/2}} \wedge \frac{1}{(T-t)^{1/4}},      \label{eq:zetax}  \\
	\norm{\zeta_{xx} }_{L^{\infty}} & \lesssim &  \frac{1}{T-t} \wedge \frac{1}{(T-t)^{1/2}},      \label{eq:zetaxx}  \\
	\norm{ \zeta_{xxx} }_{L^{\infty}} & \lesssim &  \frac{1}{(T-t)^{3/2}} \wedge \frac{1}{(T-t)^{3/4}},    \label{eq:zetaxxx}  \\
	\norm{ \zeta_{xxxx} }_{L^{\infty}} & \lesssim & \frac{1}{(T-t)^2} \wedge \frac{1}{T-t}.      \label{eq:zetaxxxx}
	\end{eqnarray}
\end{proposition}

We begin with the proof of Proposition~\ref{p:L1boundtorus}.
\begin{proof}[Proof of Proposition~\ref{p:L1boundtorus}]
	\underline{Step 0}: Set-up.

Analogously as in \cite[Proof of Proposition~3.6]{OSW}, we define the curves
\begin{align*}
\begin{aligned} \gamma_{1-}(t) &\coloneqq \tilde{a}(T) - C_1(T-t+\Lambda)^{1/4},&\gamma_{1+}(t)& \coloneqq \tilde{a}(T) + C_1(T-t+\Lambda)^{1/4}, \\
\gamma_{2-}(t) &\coloneqq \tilde{b}(T) - C_1(T-t+\Lambda)^{1/4},& \gamma_{2+}(t) & \coloneqq \tilde{b}(T) + C_1(T-t+\Lambda)^{1/4},
\end{aligned} 
\end{align*}
for order one constants $C_1$ and $\Lambda$ to be specified below.
It will also be convenient to choose $L_0$ sufficiently large so that
\red{\begin{align}
 \Lambda\leq  L^2. \label{lambdacondition}
\end{align}}

Choosing $C_1$ large enough with respect to the implicit constant in \eqref{eq:shiftcontrolini} and fixing a corresponding  $\Lambda$ from Proposition~\ref{p:schauderest}, 
we observe that  the zeros stay bounded within the moving curves, since
\begin{align} \label{bumpinequalitygammapmeps}
\tilde{a}(t) - \gamma_{1-}(t) \geq (T-t+\Lambda)^{1/4} \quad \mbox{and} \quad \gamma_{1+}(t) - \tilde{a}(t) \geq (T-t + \Lambda)^{1/4}
\end{align}
and similarly for $\tilde{b}$ and $\gamma_{2\pm}$.
Notice also that the definition of the $\gamma$-curves (including \eqref{lambdacondition}), the fact \red{$\tilde{w}\in\N(0,L)$}, and the cut-off time $\delta L^2$ imply  that the $\gamma$-curves stay well-separated in the sense that
\red{\begin{align}
	\min\Big\{\gamma_{2-}(t),-\gamma_{1+}(t)\Big\}\gtrsim L\qquad\text{for all }t\in[0,\delta L^2].\label{gammafarapart}
	\end{align}}

As in \cite{OSW}, we define a simpler stand-in for $\Wt$ via
\begin{align} \label{Wtilde}
\Wb \coloneqq \int_{-L}^{\gamma_{1-}(t)} \abs{u(t)+1} \,\dx + \int_{\gamma_{2+}(t)}^{L} \abs{u(t)+1} \,\dx + \int_{\gamma_{1+}(t)}^{\gamma_{2-}(t)} \abs{u(t)-1}\,\dx,
\end{align}
where the third integral is set to zero for times such that $\gamma_{1+}(t)\geq\gamma_{2-}(t)$.
The following lemma relates $\Wt$ and $\Wb$.
\begin{lemma}[Link between $\Wt$ and $\Wb$] \label{l:linkWW}
	Let $W_0 < \infty$, and $\epsilon >0$. 
	There exists  $L_0<\infty$ such that for $L \geq L_0$, the following holds true. Let $u$ be a smooth solution of \eqref{ch} with initial data satisfying \eqref{e0bd} and \eqref{W_uni} and let $\wt\in\mathcal{N}(L)$ be the corresponding glued kink profile. 

	Then for $t\in [0,T]$ there holds
	\begin{align} \label{eq:linkWWres}
	\Wb(t) \lesssim \Wt(t)+1 \; \mbox{ for all } t\in [0,T] \quad \mbox{and} \quad \Wt(T) \lesssim \Wb(T)+1.
	\end{align}
\end{lemma}

\begin{proof}[Proof of Lemma \ref{l:linkWW}]
	We begin with the first estimate in \eqref{eq:linkWWres}. We will give the argument for the first term in \eqref{Wtilde}; the bounds for the other two terms are obtained similarly. Using the triangle inequality, \eqref{bumpinequalitygammapmeps}, and the fact that $\wt$ can be approximated by $\pm 1$  outside an order-one neighborhood of its zeros, we obtain
	\begin{eqnarray}
	\int_{-L}^{\gamma_{1-}(t)} \abs{u(t)+1} \,\dx
	&\leq& \int_{-L}^{\gamma_{1-}(t)}\abs{u(t)-\wt(t)}\,\dx
	+ \int_{-L}^{\gamma_{1-}(t)} \abs{\wt(t)+1} \,\dx  \nonumber  \\
	&\overset{\eqref{bumpinequalitygammapmeps}}\lesssim&\Wt(t)+ \int_{-L}^{\tilde{a}(t)-(T-t+\Lambda)^{1/4}} \abs{\wt(t)+1} \,\dx \lesssim \Wt(t)+1. \nonumber 
	\end{eqnarray}
	For the second estimate in \eqref{eq:linkWWres}, we calculate
	\begin{eqnarray}
	\lefteqn{ \int_{[-L,L]} \abs{u(T) - \wt(T)}\,\dx} \nonumber \\
	& \overset{\eqref{Wtilde}}\leq & \Wb(T)  + \int_{-L}^{\gamma_{1-}(T)} \abs{\wt(T) + 1}\,\dx + \int_{\gamma_{2+}(T)}^{L} \abs{\wt(T) + 1}\,\dx + \int_{\gamma_{1+}(T)}^{\gamma_{2-}(T)} \abs{\wt(T) - 1}\,\dx\nonumber \\
	& & +\int_{\gamma_{1-}(T)}^{\gamma_{1+}(T)}\abs{u(T)-\wt(T)}\,\dx + \int_{\gamma_{2-}(T)}^{\gamma_{2+}(T)}\abs{u(T) - \wt(T)}\,\dx \nonumber \\
	& \lesssim & \Wb(T) +  1 +  \int_{\gamma_{1-}(T)}^{\gamma_{1+}(T)}\abs{u(T) - \wt(T)}\,\dx + \int_{\gamma_{2-}(T)}^{\gamma_{2+}(T)} \abs{u(T) - \wt(T)}\,\dx. \nonumber
	\end{eqnarray}
	For the last two terms on the right-hand side, we use the Cauchy-Schwarz inequality and our choice of $C_1$ and $\Lambda$ to estimate
	\begin{eqnarray}
	\lefteqn{\int_{\gamma_{1-}(T)}^{\gamma_{1+}(T)} \abs{u(T) - \wt(T)}\,\dx + \int_{\gamma_{2-}(T)}^{\gamma_{2+}(T)} \abs{u(T) - \wt(T)}\,\dx}\nonumber \\
	&\lesssim& \left(C_1\Lambda^{1/4}\right)^{1/2}\norm{\ft(T)}_{L^2([-L,L])}
	\overset{\eqref{eq:torusL1boundassumpt}}\lesssim
\Et^{1/2}\overset{\eqref{diffcons2}}\lesssim 1.\notag
	\end{eqnarray}
\end{proof}

\uline{Step 1}: Main idea.
We now explain the main idea, which is to use duality to express the third term in $\Wb$ as
	\begin{align*}
	\int_{\gamma_{1+}(t)}^{\gamma_{2-}(t)} \abs{u(t) - 1}\, \dx = \sup \left\{  \int_{\gamma_{1+}(t)}^{\gamma_{2-}(t)} \psi (u(t)-1)  \, \dx \mid \psi: \red{(\gamma_{1+}(t), \gamma_{2-}(t))} \to \R \, \mbox{ such that } \norm{\psi}_{L^{\infty}} \leq 1  \right\}.
	\end{align*}
	Bearing in mind periodicity, the sum of the first two terms in $\Wb$ can be treated in the same way. Choosing $\psi$ as the terminal data for~\eqref{eq:zetapde}, we will develop useful bounds for
\begin{align*}
  \ddt\int_{\gamma_{1+}(t)}^{\gamma_{2-}(t)}   (u(t)-1)\zeta\,\dx.
\end{align*}
 For this, the semigroup estimates from Proposition~\ref{p:schauderest} play an essential role. An integration in time and taking the supremum over the terminal data---and the analogous arguments for the remaining two terms in $\Wb$---will lead to
	\begin{eqnarray}
	\Wb(T)
  	&\lesssim& \begin{cases} T +\Wtz +1& \mbox{for all }T, \\
  \Big(\frac{\ln T}{T^{1/4}}+ \frac{T^{1/4}}{L^{1/2}} \Big)\Wt_T +\Wt_T^{2/3} + \Wtz+1 & \mbox{for } 2\leq T \leq \delta L^2, \\
   \end{cases} \nonumber
	\end{eqnarray}
	 for any $\delta\in(0,1)$.
	 We will use the first estimate on $[0,T_1]$ and the second estimate on $[T_1,\delta L^2]$, choosing $T_1$ large enough and $\delta$ small enough so that we can absorb the  linear in $\Wt_T$ terms from the right-hand side. Combining this with Lemma~\ref{l:linkWW} and Young's inequality, we obtain
\begin{align*}
 \bar{\W}(T)+ \Wt(T)\lesssim\Wtz+1.
\end{align*}

	\underline{Step 2}: Estimates. We now turn in more detail to
\begin{align*} 
		\int_{\gamma_{1+}(t)}^{\gamma_{2-}(t)} \abs{u(t)-1}\, \dx ,
		\end{align*}
\red{for which we will introduce a cut-off function in order to work with only one moving boundary at a time.}
We will consider ``short'' times (for which the estimate is allowed to grow with $T$) and large times $2\leq T\leq \delta L^2$.
	
We identify the torus with the interval $[-L,L]$ on $\R$.
\red{Because of the separation from~\eqref{gammafarapart},} we can define $\varphi: \R \to [0,1]$ such that the derivatives of $\varphi$ are supported on an order $L$ interval around the origin and
\begin{align*}
		\varphi \equiv 0 \mbox{ on }
		\left(-\infty,\gamma_{1+}(t)\right]\quad \mbox{and} \quad \varphi \equiv 1 \mbox{ on } \left[\red{-2},\infty \right)\qquad\text{for all }t\leq T\leq \delta L^2, 	
		\end{align*}
		while
		\begin{equation} \label{etaLiL2}
		\norm{\varphi_{x}}_{L^2(\R)} \lesssim \frac{1}{L^{1/2}}, \quad \norm{\varphi_{xx}}_{L^2(\R)} \lesssim \frac{1}{L^{3/2}}, \quad \mbox{and} \quad \norm{\varphi_{xxx}}_{L^2(\R)} \lesssim \frac{1}{L^{5/2}}.
		\end{equation}
We use $\varphi$ to write
		\begin{align*}
		\int_{\gamma_{1+}(t)}^{\gamma_{2-}(t)}\abs{u(t)-1}\,\dx= \int_{-\infty}^{\gamma_{2-}(t)}\abs{u(t)-1}\varphi\,\dx + \int_{\gamma_{1+}(t)}^{\infty}\abs{u(t)-1}(1-\varphi)\,\dx.
		\end{align*}
We will provide estimates for the first term on the right-hand side; the estimates for the second term are analogous.
Using duality in the form
\begin{eqnarray}
		\lefteqn{\int_{-\infty}^{\gamma_{2-}(t)} \abs{u(t) - 1}\varphi \, \dx }\nonumber \\
		&=& \sup \left\{  \int_{-\infty}^{\gamma_{2-}(t)} \psi (u(t)-1) \varphi \, \dx \mid \psi: (-\infty, \gamma_{2-}(t)) \to \R \, \mbox{ s.t. } \norm{\psi}_{L^{\infty}((-\infty,\gamma_{2-}(t)))} \leq 1  \right\} \nonumber
		\end{eqnarray}
and defining $\zeta$ as in Proposition~\ref{p:schauderest} as usual, it suffices to establish estimates for
\begin{align*}
\int_0^T  \left|\ddt \int_{-\infty}^{\gamma_{2-}} \zeta (u-1) \varphi \, \dx\right|\,\dt.
\end{align*}

		For simplicity of notation, we will in the remainder denote $\gamma \coloneqq \gamma_{2-}$ and not explicitly indicate the domain over which norms are taken when it is clear from the context.
		The properties of $\zeta$, $\wt$, and $\varphi$ and integration by parts yield
		\begin{eqnarray}
		\lefteqn{\ddt \int_{-\infty}^{\gamma} \zeta (u-1) \varphi \, \dx} \nonumber \\
		& \overset{\zeta(\gamma)=0}= & \int_{-\infty}^{\gamma} \zeta_t (\wt-1) \varphi \, \dx + \int_{-\infty}^{\gamma} \zeta_{xx} \left(G'(u) - G'(\wt)-G''(1)\ft \right) \varphi \, \dx \nonumber \\
			& & - \left[ \zeta_x  \left(G'(u) - G'(\wt)-\ft_{xx}\right) \varphi \right]_{x=\gamma} + \left[ \zeta_{xxx}\ft \varphi\right]_{x=\gamma} \nonumber \\
			& & + \int_{-\infty}^{\gamma} \zeta_{xx} \left( 2 \ft_{x} \varphi_x + \ft \varphi_{xx} \right) \, \dx  + \int_{-\infty}^{\gamma} \left( 2 \zeta_x \varphi_x + \zeta \varphi_{xx} \right) \left( G'(u) - G'(\wt) - \ft_{xx} \right) \, \dx \nonumber \\
			& & + \int_{(-\infty,\gamma]\cap (-1,1)} \zeta_{xx}\varphi  (G'(\wt)-\wt_{xx})\,\dx
\nonumber \\
			& \eqqcolon & J_1 + J_2 + J_3 + J_4 + J_5 + J_6 + J_7. \nonumber		
		\end{eqnarray}
		The following arguments are similar to those from \cite[Proposition~3.6]{OSW}. While the terms $J_1,...,J_4$ can essentially be treated as $I_1,...,I_4$ in \cite[Proof of Proposition~3.6]{OSW}, the new terms $J_5$ and $J_6$ appear because of the presence of the cut-off function $\varphi$, and $J_7$ appears since on $(-1,1)$, $\wt$ does not satisfy $-\wt_{xx}+G'(\wt)=0$.
For completeness and the convenience of the reader, we include all of the estimates. Often the argument for the short time estimate and the argument for the time integral over the so-called terminal layer $[T-1,T]$ for $T\geq 2$ will be the same; in that case we point this out briefly and leave it to the reader to check.
We remark for reference below that
		\begin{align} \label{eq:bumpfact}
		\int_{-\infty}^{\gamma} \abs{\wt_{x}}\varphi + \abs{\wt-1}\varphi_x + \abs{\wt-1}\varphi\, \dx \lesssim 1.
		\end{align}

		\underline{Term $J_1$}: For order one times, we use integration by parts to deduce
		\begin{eqnarray}
			\lefteqn{\left\vert \int_{-\infty}^{\gamma} \zeta_t(\wt-1) \varphi \, \dx  \right\vert  =  \left\vert \int_{-\infty}^{\gamma} \Big(\zeta_{xxxx}-\zeta_{xx}G''(1) \Big)(\wt-1) \varphi \, \dx  \right\vert} \nonumber \\
			& = & \Big\vert \zeta_{xxx}(\gamma)(\wt(\gamma)-1)\varphi(\gamma) - \int_{-\infty}^{\gamma} \zeta_{xxx} \Big( \wt_{x}\varphi + (\wt-1)\varphi_x \Big) \, \dx \nonumber \\
			& & - \int_{-\infty}^{\gamma} \zeta_{xx}G''(1)(\wt-1) \varphi \, \dx \Big\vert \nonumber \\
			& \lesssim & \norm{\zeta_{xxx}}_{L^{\infty}} \left(1 +  \int_{-\infty}^{\gamma} \abs{\wt_{x}\varphi + (\wt-1)\varphi_x} \, \dx \right) + \norm{ \zeta_{xx}}_{L^{\infty}} \int_{-\infty}^{\gamma} \abs{\wt-1}\varphi\, \dx \nonumber \\
			&  \overset{\eqref{eq:zetaxx},\, \eqref{eq:zetaxxx}, \eqref{eq:bumpfact}}{\lesssim} & \frac{1}{(T-t)^{1/2}} + \frac{1}{(T-t)^{3/4}}, \label{J11}
		\end{eqnarray}
		so that
			\begin{align*}
			\int_0^T \abs{J_1} \, \dt  \overset{\eqref{eq:zetaxx},\eqref{eq:zetaxxx},\eqref{eq:bumpfact}}\lesssim T^{1/2} + T^{1/4}.
		\end{align*}
		
		For the terminal layer $[T-1,T]$, we estimate as in \eqref{J11} and use the Schauder estimates from Proposition~\ref{p:schauderest} to derive
			\begin{align*}
			\left\vert \int_{-\infty}^{\gamma} \zeta_t(\wt-1) \varphi \, \dx  \right\vert
			 \overset{\eqref{eq:zetaxx},\, \eqref{eq:zetaxxx}, \eqref{eq:bumpfact}}{\lesssim} \frac{1}{(T-t)^{1/2}} + \frac{1}{(T-t)^{3/4}}, \nonumber
		\end{align*}
		so that
		\begin{align*}
			\int_{T-1}^T \left\vert J_1  \right\vert \, \dt \lesssim 1.
		\end{align*}
		
		For $T \geq 2$, we use
		Proposition~\ref{p:schauderest} to estimate
		\begin{align*}
			\left\vert \int_{-\infty}^{\gamma} \zeta_t (\wt-1) \varphi \, \dx\right\vert \lesssim \norm{\zeta_t}_{L^{\infty}} \int_{-\infty}^{\gamma} \abs{\wt-1}\varphi \, \dx  \overset{\eqref{eq:zetapde}, \,\eqref{eq:zetaxx},\,\eqref{eq:zetaxxxx}}{\lesssim} \frac{\exp\left( -\frac{\tilde{a}(T)-\gamma}{C}\right)}{T-t}  \lesssim  \frac{\exp\left( -\frac{(T-t)^{1/4}}{C}\right)}{T-t},
		\end{align*}
		where we have used exponential tails of $\abs{\wt-1}$ order one away from $\tilde{a}(T)$ and \eqref{bumpinequalitygammapmeps}.
		Integrating in time over $[0,T-1]$ gives
		\begin{align}
			\int_0^{T-1} \left\vert J_1  \right\vert \, \dt  \lesssim \int_0^{T-1} \frac{1}{(T-t)^{3/4}} \exp\left( -\frac{(T-t)^{1/4}}{C}\right)\, \dt \lesssim 1. \nonumber
		\end{align}

		\underline{Term $J_2$}: We reformulate $J_2$ as
		\begin{align*}
		\abs{J_2} = \left\vert \int_{-\infty}^{\gamma} \zeta_{xx} \Big(G'(u) - G'(\wt)-G''(\wt)\ft  + \left(G''(\wt)-G''(1)\right)\ft \Big) \varphi \, \dx \right\vert \nonumber		
		\end{align*}
		and use the Cauchy-Schwarz inequality,  square integrability of $\abs{\wt-1}$ on $(-L, \gamma)\cap \supp(\varphi)$, and the properties of $G$ to obtain
		\begin{eqnarray}
		\lefteqn{\left\vert\int_{-\infty}^{\gamma}\zeta_{xx} ( G''(\wt)-G''(1) )\ft \varphi \, \dx \right\vert} \nonumber \\
		& \lesssim &  \norm{\zeta_{xx}}_{L^{\infty}} \left( \int_{-\infty}^{\gamma} (G''(\wt) - G''(1))^2\varphi^2\, \dx \right)^{1/2} \left( \int_{-\infty}^{\gamma}  \ft^2 \, \dx \right)^{1/2} \notag\\
		&\lesssim& \norm{\zeta_{xx}}_{L^{\infty}} \left( \int_{-\infty}^{\gamma}  \ft^2 \, \dx \right)^{1/2}.\label{1bumpI2eq1}
		\end{eqnarray}
		For the remaining term, Taylor's formula, the $L^{\infty}$ bound on $u$ and $\wt$, and boundedness of $G'''$ on compact intervals yield
		\begin{eqnarray} \label{1bumpI2eq2}
		\lefteqn{\left\vert\int_{-\infty}^{\gamma} \zeta_{xx} \left( G'(u) - G'(\wt)-G''(\wt)\ft\right) \varphi \, \dx\right\vert} \nonumber \\
		& \lesssim & \norm{\zeta_{xx}}_{L^{\infty}} \int_{-\infty}^{\gamma}  \abs{ G'(u) - G'(\wt)-G''(\wt)\ft} \,\dx \lesssim \norm{\zeta_{xx}}_{L^{\infty}} \int_{-\infty}^{\gamma} \ft^2 \, \dx.
		\end{eqnarray}
		We combine \eqref{1bumpI2eq1}, \eqref{1bumpI2eq2}, and Assumption \eqref{eq:torusL1boundassumpt} to bound
		\begin{align*}
		\abs{J_2} \overset{\eqref{eq:torusL1boundassumpt}}\lesssim\norm{\zeta_{xx}}_{L^{\infty}} \left( \Et + \Et^{1/2}\right).
		\end{align*}
		
		For order one times, the simple energy estimate $\Et \lesssim 1$ and Proposition~\ref{p:schauderest} lead to
		\begin{align*}
		\int_0^T \abs{J_2} \,\dt \lesssim \int_0^T \norm{\zeta_{xx}}_{L^{\infty}} \Big( \Et + \Et^{1/2} \Big) \,\dt  \overset{\eqref{eq:zetaxx}}{\lesssim}  \int_0^T  \frac{1}{(T-t)^{1/2}} \,\dt \lesssim T^{1/2}. \nonumber
		\end{align*}
		
		For $T \geq 2$,  the above argument bounds the integral over the terminal layer by $1$.
	
		For the integral over $[0,T-1]$, we apply $\Et\lesssim1$ and \eqref{eq:torusL1boundassumpt} to obtain
		\begin{align}
		\abs{J_2} \overset{\eqref{eq:odenashresult}}{\lesssim} \norm{\zeta_{xx}}_{L^{\infty}}
		\frac{\Wt_T}{t^{1/4}}. \nonumber
		\end{align}
		An integration in time yields
		\begin{align*}
		\int_0^{T-1} \vert J_2 \vert \, \dt \overset{\eqref{eq:zetaxx}}{\lesssim} & \int_0^{T-1}\frac{1}{T-t}
		\frac{\Wt_T}{t^{1/4}}\,\dt \lesssim  \frac{\ln T}{T^{1/4}}\Wt_T + 1.
		\end{align*}
			
		\underline{Term $J_3$}: For the first term in $J_3$, we again use boundedness of $G''$ on compact intervals to obtain
		\begin{align*}
		\Big\vert \zeta_x(\gamma)  \Big(G'(u(\gamma)) - G'(\wt(\gamma))\Big) \Big\vert \lesssim \norm{\zeta_{x}}_{L^{\infty}} \left\vert \int_{\wt(\gamma)}^{u(\gamma)} G''(s) \, \rm{d}s \right\vert \lesssim \norm{\zeta_{x}}_{L^{\infty}} \abs{\ft(\gamma)}.
		\end{align*}
		Substituting into $J_3$, we deduce
		\begin{align} \label{eq:I3est1}
		\vert J_3 \vert \lesssim \norm{\zeta_x}_{L^{\infty}} \left( \abs{ \ft(\gamma)} + \abs{\ft_{xx}(\gamma)} \right).
		\end{align}
		
		For small times, we  apply the simplistic energy and dissipation estimates from \eqref{eq:basiceedconsequence} to derive
		\begin{align} \nonumber
		\int_0^T \norm{\zeta_x}_{L^{\infty}}  \abs{\ft(\gamma)} \, \dt \lesssim \int_0^T \norm{\zeta_x}_{L^{\infty}}  \norm{\ft}_{L^{\infty}} \, \dt  \overset{\eqref{eq:basiceedconsequence},\eqref{eq:zetax}}{\lesssim} \int_0^T \frac{1}{(T-t)^{1/4}} \, \dt \lesssim T^{3/4}
		\end{align}
		and Cauchy-Schwarz and \eqref{diff} in
		\begin{align} \nonumber
		\int_0^T \norm{\zeta_x}_{L^{\infty}} \abs{\ft_{xx}(\gamma)} \, \dt\lesssim \int_0^T \norm{\zeta_x}_{L^{\infty}} \norm{\ft_{xx} }_{L^{\infty}} \, \dt  \overset{\eqref{eq:basiceedconsequence},\eqref{eq:zetax}}{\lesssim} \int_0^T \frac{\D^{1/2} + \exp(-L/C)}{(T-t)^{1/4}} \, \dt \lesssim T^{3/4}.
		\end{align}
For $T \geq 2$, we use the same estimates on the terminal layer $[T-1,T]$.

To estimate \eqref{eq:I3est1} for the time interval $[0,T-1]$, we 
use the support of $\varphi$ and argue as in the proof of Lemma~\ref{l:hardy}
 for $(u-v_{\tilde{a}})(\tilde{a})$ on $(-L,0)$ to deduce
		\begin{align*}
			\abs{\ft(\tilde{a})}^2\lesssim \int_{-L}^{0} \ft_x^2\,\dx.
		\end{align*}

		This together with the triangle inequality, the Cauchy-Schwarz inequality, and \eqref{bumpinequalitygammapmeps} gives
		\begin{eqnarray}
		\lefteqn{\abs{\ft(\gamma)}} \nonumber \\
		& \leq & \abs{\ft(\gamma) - \ft(\tilde{a})}+ \abs{\ft(\tilde{a})} =\left\vert\int_{\gamma}^{\tilde{a}} \ft_{x} \,\dx   \right\vert+ \abs{\ft(\tilde{a})} \nonumber \\
		&\overset{\eqref{eq:Derror}}\lesssim & \abs{\tilde{a}-\gamma}^{1/2} \left(\D^{1/2}+ \exp\left( -L/C \right)\right) + \D^{1/2} + \exp\left( -L/C\right) \nonumber \\
		& \overset{\eqref{bumpinequalitygammapmeps}}{\lesssim} & (T-t+\Lambda)^{1/8} \left(\D^{1/2}+\exp\left( -L/C \right)\right). \nonumber
		\end{eqnarray}
		Combining this with another application of \eqref{eq:basiceedconsequence} to the second term in \eqref{eq:I3est1}, we conclude
		\begin{align} \label{eq:I3est3}
		\abs{J_3} \overset{\eqref{eq:basiceedconsequence}}\lesssim \norm{\zeta_x}_{L^{\infty}}(T-t+\Lambda)^{1/8} \left(\D^{1/2}+ \exp\left( -L/C \right)\right).
		\end{align}
		
		For large times, H\"{o}lder's inequality with exponent $\frac{3}{2}$ and Lemma \ref{l:intdissipationbound} yield
		\begin{eqnarray}
		\lefteqn{\int_0^{T-1} \abs{J_3} \, \dt} \nonumber \\
		& \overset{\eqref{eq:I3est3}}{\lesssim} & \int_0^{T-1} \norm{\zeta_x}_{L^{\infty}} (T-t+\Lambda)^{1/8} \left(\D^{1/2}+\exp(-L/C)\right) \, \dt \nonumber \\
		& \overset{\eqref{eq:zetax}}{\lesssim} & \int_0^{T-1}  \frac{\D^{1/2}}{(T-t)^{3/8}} \,\dt  + \int_0^{T-1}  \frac{1}{(T-t)^{3/8}} \exp(-L/C) \,\dt \nonumber \\
		& \lesssim &  \left( \int_0^{T-1} \left( \frac{1}{(T-t)^{3/8}} \right)^3 \, \dt \right)^{1/3}  \left( \int_0^{T-1} \D^{3/4} \, {\rm{d}} t \right)^{2/3} +  T^{5/8}\exp(-L/C) \nonumber \\
		& \overset{\eqref{eq:intdissipationboundres}}{\lesssim} & \frac{\Wt_T^{2/3}}{T^{1/24}} + 1 \lesssim \Wt_T^{2/3} +1, \nonumber
		\end{eqnarray}
		where we have used $T \leq \delta L^2$.
		
		\underline{Term $J_4$}: For $J_4$, we use the elementary estimate \eqref{elementary} and \eqref{eq:torusL1boundassumpt} to deduce
		\begin{align*}
		\abs{J_4} \lesssim \norm{\zeta_{xxx}}_{L^{\infty}} \norm{\ft}_{L^{\infty}} {\lesssim}  \norm{\zeta_{xxx}}_{L^{\infty}} \Et^{1/2}.
		\end{align*}
		
		For small times, we use this together with $\Et\lesssim 1$ to conclude
		\begin{align*}
		\int_0^T \abs{J_4}\,\dt\lesssim \int_0^T \norm{\zeta_{xxx}}_{L^{\infty}}  \, \dt  \overset{\eqref{eq:zetaxxx}}{\lesssim} T^{1/4}, \nonumber
		\end{align*}
and similarly for the terminal layer.

For $2\leq T \leq \delta L^2$ and the integral over $[0,T-1]$, we use
		\begin{eqnarray}\int_0^{T-1} \abs{J_4}\,\dt
		&\lesssim& \int_0^{T-1} \norm{\zeta_{xxx}}_{L^{\infty}} \Et^{1/2} \, \dt \nonumber \\
		& \overset{\eqref{eq:torusL1boundassumpt},\,\eqref{eq:zetaxxx}}{\lesssim} &  \Wt_T \int_0^{T-1} \frac{1}{(T-t)^{3/2}} \frac{1}{t^{1/4}} \, \dt  \lesssim \frac{1}{T^{{1/4}}}\Wt_T. \nonumber
		\end{eqnarray}
		
		\underline{Term $J_5$}: Using the Cauchy-Schwarz inequality, we estimate
		\begin{eqnarray}
		\lefteqn{\left\vert \int_{-\infty}^{\gamma} \zeta_{xx} \ft_{x} \varphi_x \,\dx \right\vert + \left\vert \int_{-\infty}^{\gamma} \zeta_{xx} \ft \varphi_{xx}  \, \dx \right\vert} \nonumber \\
		 & \lesssim & \norm{\zeta_{xx}}_{L^{\infty}} \norm{\ft_{x}}_{L^2} \norm{\varphi_x}_{L^2} + \norm{\zeta_{xx}}_{L^{\infty}} \norm{\ft}_{L^2} \norm{\varphi_{xx}}_{L^2}\overset{\eqref{eq:torusL1boundassumpt},\, \eqref{etaLiL2}}{\lesssim} \frac{1}{L^{1/2}} \norm{\zeta_{xx}}_{L^{\infty}} \Et^{1/2}\nonumber
		\end{eqnarray}
		and proceed as in $J_2$ to obtain
		\begin{align*}
		\int_0^T \abs{J_5} \,\dt \lesssim \begin{cases}
		\frac{1}{L^{1/2}}T^{1/2} & \mbox{for all } T, \\  \frac{1}{L^{1/2}}  \frac{\ln T}{T^{1/4}}\Wt_T  & \mbox{for } 2 \leq T \leq \delta L^2.
		\end{cases}
		\end{align*}
	
		\underline{Term $J_6$}: We re-express $J_6$ as
		\begin{eqnarray}
			\lefteqn{\int_{-\infty}^{\gamma} \left( 2 \zeta_x \varphi_x + \zeta \varphi_{xx} \right) \left( G'(u) - G'(\wt) - \ft_{xx} \right) \, \dx} \nonumber\\
			 & = &  \int_{-\infty}^{\gamma} 2\zeta_x \varphi_x  \left( G'(u) - G'(\wt) \right) \, \dx  + \int_{-\infty}^{\gamma} \zeta \varphi_{xx} \left( G'(u) - G'(\wt) \right) \, \dx
			\nonumber \\
			& & - \int_{-\infty}^{\gamma} \zeta_x \varphi_x \ft_{xx} \, \dx- \int_{-\infty}^{\gamma} \zeta \varphi_{xx} \ft_{xx} \, \dx \eqqcolon J_{6a} + J_{6b} + J_{6c} + J_{6d}. \nonumber
		\end{eqnarray}

		To estimate $J_{6a}$, we apply the Cauchy-Schwarz inequality, the $L^\infty$ bound on $u$ and $\wt$, and the properties of $G$ to deduce
		\begin{align*}
		\abs{J_{6a}} \lesssim \norm{\zeta_x}_{L^{\infty}} \norm{\varphi_x}_{L^2} \norm{\ft}_{L^2} \overset{\eqref{eq:torusL1boundassumpt},\,\eqref{etaLiL2}}{\lesssim} \frac{1}{L^{1/2}} \norm{\zeta_x}_{L^{\infty}} \Et^{1/2}.
		\end{align*}
		
		For small times, an integration in time, Proposition \ref{p:schauderest}, and $\Et\lesssim 1$ yield
		\begin{align*}
		\int_0^T \abs{J_{6a}} \, \dt \overset{\eqref{eq:zetax}}{\lesssim} \frac{1}{L^{1/2}} \int_0^T \frac{1}{(T-t)^{1/4}} \, \dt \lesssim \frac{T^{3/4}}{L^{1/2}}.
		\end{align*}
		
		For  $2 \leq T \leq \delta L^2$, we use the time-dependent energy estimate \eqref{eq:torusL1boundassumpt} to derive
		\begin{align*}
			\int_0^T \abs{J_{6a}}\,\dx
			 \overset{\eqref{eq:torusL1boundassumpt}, \eqref{eq:torusL1boundassumpt},\eqref{eq:zetax}}{\lesssim}  \frac{\Wt_T}{L^{1/2}} \int_0^T \frac{1}{(T-t)^{1/2}} \frac{1}{t^{1/4}}\,\dt \lesssim \frac{\Wt_T}{L^{1/2}}  T^{1/4}.
		\end{align*}
		We treat $J_{6b}$ in a similar manner and obtain
		\begin{align*}
		\abs{J_{6b}} \lesssim \norm{\zeta}_{L^{\infty}} \norm{\ft}_{L^2} \norm{\varphi_{xx}}_{L^2} \overset{\eqref{eq:torusL1boundassumpt},\eqref{etaLiL2}}{\lesssim} \frac{1}{L^{\frac{3}{2}}} \norm{\zeta}_{L^{\infty}} \Et^{1/2}
		\end{align*}
		and
		\begin{align*}
		\int_0^T \abs{J_{6b}} \, \dt \overset{\eqref{eq:zeta}}{\lesssim} \begin{cases}
		\frac{T}{L^{3/2}} & \mbox{for all } T, \\ \frac{T^{3/4}}{L^{3/2}}\Wt_T & \mbox{for } 2\leq T \leq \delta L^2.
		\end{cases}
		\end{align*}
		
		For $J_{6c}$, we get
		\begin{align*}
		\abs{J_{6c}} \lesssim \norm{\zeta_x}_{L^{\infty}} \norm{\ft_{xx}}_{L^2} \norm{\varphi_x}_{L^2} \overset{\eqref{eq:basiceedconsequence}}\lesssim \frac{1}{L^{1/2}} \norm{\zeta_x}_{L^{\infty}} \left( \D^{1/2}+ \exp(-L/C) \right).
		\end{align*}
For small times, we use Proposition \ref{p:schauderest} and Cauchy-Schwarz to obtain
		\begin{eqnarray}
			\int_0^T \abs{J_{6c}} \, \dt & \overset{\eqref{eq:zetax}}{\lesssim} & \frac{1}{L^{1/2}} \int_0^T \frac{1}{(T-t)^{1/4}}  \left( \D^{1/2} + \exp(-L/C) \right) \, \dt \lesssim \frac{1}{L^{1/2}} T^{3/4}. \nonumber
		\end{eqnarray}
For $2\leq T \leq \delta L^2$, we argue similarly for the terminal layer and for $[0,T-1]$, we use Proposition \ref{p:schauderest} and the integral dissipation bound. Applying Lemma \ref{l:intdissipationbound} requires a larger exponent on $\D$, which is obtained by using H\"{o}lder's inequality as in $J_3$ and absorbing the exponential error term for times $T\lesssim L^2$:
		\begin{eqnarray}
			\int_0^{T-1} \abs{J_{6c}} \, \dt &  \overset{\eqref{eq:zetax}}{\lesssim} & \frac{1}{L^{1/2}} \int_0^{T-1} \frac{1 }{(T-t)^{1/2}} \left( \D^{1/2} + \exp(-L/C) \right)\, \dt \nonumber \\ &\lesssim& \frac{1}{L^{1/2}} \left( \int_0^{T-1} \frac{1}{(T-t)^{3/2}} \, \dt  \right)^{1/3} \left[ \left( \int_0^{T-1} \D^{3/4} \, \dt\right)^{2/3} + \exp(-L/C)T \right] \nonumber \\
			& \overset{\eqref{eq:intdissipationboundres}}{\lesssim} & \frac{1}{L^{1/2}} \frac{\Wt_T^{2/3}}{T^{1/6}}. \nonumber
		\end{eqnarray}
		
		Finally, for $J_{6d}$, we use integration by parts and $\zeta(\gamma)=0$ to estimate
		\begin{eqnarray}
		\left\vert \int_{-\infty}^{\gamma} \zeta  \ft_{xx} \varphi_{xx}  \, \dx \right\vert
			& = & \left|\zeta(\gamma) \ft_{x}(\gamma) \varphi_{xx}(\gamma) - \int_{-\infty}^{\gamma} \left( \zeta_x \varphi_{xx} + \zeta \varphi_{xxx} \right) \ft_{x} \, \dx \right|\nonumber \\
			& \lesssim & \left( \norm{\zeta_x}_{L^{\infty}} \norm{ \varphi_{xx}}_{L^2} + \norm{\zeta}_{L^{\infty}} \norm{\varphi_{xxx}}_{L^2}\right) \norm{\ft_{x}}_{L^2} \nonumber \\
			& \overset{\eqref{eq:torusL1boundassumpt},\,\eqref{etaLiL2}}{\lesssim} & \left(\frac{1}{L^{3/2}}  \norm{\zeta_x}_{L^{\infty}} + \frac{1}{L^{5/2}}\norm{\zeta}_{L^{\infty}} \right) \Et^{1/2}. \nonumber
		\end{eqnarray}
For small times we use the simple energy estimate
		\begin{align*}
		\int_0^T \abs{J_{6d}} \, \dt \lesssim \frac{1}{L^{3/2}}\int_0^T   \frac{1}{(T-t)^{1/4}} + \frac1L \, \dt \lesssim \frac{T^{3/4}}{L^{3/2}} + \frac{T}{L^{5/2}}.
		\end{align*}
For  $2 \leq T \leq \delta L^2$, Proposition \ref{p:schauderest}, and \eqref{eq:torusL1boundassumpt} lead to
		\begin{eqnarray}
			\int_0^T \abs{J_{6d}} \,\dt
			& \overset{\eqref{eq:torusL1boundassumpt},\,\eqref{eq:zeta},\, \eqref{eq:zetax} }{\lesssim} &  \frac{1}{L^{3/2}}\int_0^{T} \frac{1}{(T-t)^{1/2}}  \frac{\Wt_T}{t^{1/4}}+ \frac{\Wt_T}{Lt^{1/4}} \, \dt  \lesssim \frac{\Wt_T}{L^{3/2}} T^{1/4} + \frac{\Wt_T}{L^{5/2}} T^{3/4}. \nonumber
		\end{eqnarray}

\underline{Term $J_7$}: For $J_7$, we note that on $(-1,1)$ there holds $\abs{\wt-1} \lesssim \exp(-L/C)$ and therefore also $\abs{G'(\wt)} \lesssim \exp(-L/C)$ on $(-1,1)$. Thus
\begin{align*}
	 \left\vert \int_{[-L,\gamma)\cap (-1,1)} \zeta_{xx} \left( G'(\wt)-\wt_{xx}\right) \varphi \, \dx \right\vert  & \lesssim \norm{\zeta_{xx}}_{L^{\infty}} \Big( \sup_{\abs{\tau-1}\leq \exp(-L/C)} G'(\tau) + \norm{\wt_{xx}}_{L^2((-1,1))} \Big) \\
	 & \lesssim \norm{\zeta_{xx}}_{L^{\infty}}  \exp(-L/C).
\end{align*}
For small times, we deduce
\begin{align*}
	\int_0^T \abs{J_{7}} \, \dt \lesssim \exp(-L/C)\int_0^T   \frac{1}{(T-t)^{1/2}}  \, \dt \lesssim \exp(-L/C)T^{1/2},
\end{align*}
and similarly for the terminal layer $[T-1,T]$.
For larger times, we obtain
\begin{align*}
	\int_0^T \abs{J_{7}} \, \dt \lesssim \exp(-L/C)\int_0^T   \frac{1}{(T-t)}  \, \dt \lesssim \exp(-L/C)\ln(T).
\end{align*}

\end{proof}

\begin{proof}[Proof of Proposition~\ref{p:L1bound}]
We define the curves $\gamma_{i\pm}$, $i=1,2$ as before and use Lemma~\ref{l:u2zeros} to verify that for $T\lesssim L^2$, the zeros stay bounded between the moving curves.
We will refer to $(-\infty,\gamma_{1-}]$ and $[\gamma_{2+},\infty)$ as the ``outer regions'' and $[\gamma_{1+},\gamma_{2-}]$ as the ``inner region.'' Up to time $\delta L^2$, the proof is analogous to that of Proposition~\ref{p:L1boundtorus}.	 The inner region is controlled as in the proof on the torus, and the outer regions are even simpler (since no cut-off function is necessary), and the proof proceeds as in the proof of \cite[Proposition 3.6]{OSW}.	

To continue past $T=\delta L^2$ (and hence get the smallness claimed in \eqref{egL22}), we need a separate argument. We make use of the condition
$\Et\lesssim (\Wtz^2+1)/L$ to estimate the $L^1$ norm on the inner region
\begin{align*}
  \int_{\gamma_{1+}(T)}^{\gamma_{2-}(T)}\abs{u(T)-\wt(T)}\,dx\lesssim L^{1/2}\left(\int_{\gamma_{1+}(T)}^{\gamma_{2-}(T)}\abs{u(T)-\wt(T)}^2\,dx\right)^{1/2}\lesssim \Wtz+1.
\end{align*}
For the outer regions, on the other hand, we can again proceed as in \cite[Proposition 3.6]{OSW}, since the restriction $T\leq \delta L^2$ is only necessary for the terms $J_5$ and $J_6$, which do not occur on the outer regions since there is no need for a cut-off function. (The term $J_7$ also does not occur.)

\end{proof}

For the proof of Proposition \ref{p:L1boundsmallE}, we will make use of the following well-known parabolic estimate (see also \cite[p. 22]{OSW}). The proof is not hard and relies on Fourier transformation and the scaling of the heat kernel and the biharmonic heat kernel.
\begin{proposition}\label{p:semigroup}
Let $\zeta$ satisfy the backwards problem
	\begin{equation} \label{eq:zetapdeminus}
	\begin{cases}
	\zeta_t + G''(-1) \zeta_{xx} - \zeta_{xxxx} = 0 & \mbox{ on } \; t \in [0,T), \; x\in\R,  \\
	\zeta = \psi & \mbox{ for } \; t=T, \; x\in \R,     \end{cases}
	\end{equation}
	where $\psi$ satisfies $\norm{\psi}_{L^{\infty}} \leq 1$. There holds
	\begin{eqnarray}
	\norm{\zeta_{xx} }_{L^{\infty}} & \lesssim &  \frac{1}{T-t} \wedge \frac{1}{(T-t)^{1/2}},      \label{eq:zetaxxminus}
	\end{eqnarray}
\end{proposition}

\begin{proof}[Proof of Proposition \ref{p:L1boundsmallE}]
 	The proof is analogous to but much simpler than the one for Proposition \ref{p:L1bound}. We will work on a finite time horizon $[0,T]$ and derive estimates that hold for all $T$.
 First we deduce from Lemmas~\ref{l:basiceed-} and~\ref{l:nash} that
 	\begin{align}\label{eq:EnergyDecay-}
 	E(t)\lesssim \frac{\left(\sup_{t\leq T} \Wm (t)+1\right)^2}{t^{1/2}} \quad \mbox{for } t\in[0,T].
 	\end{align}	
	Using duality, we express
		\begin{align*}
			\int_{\R} \abs{u(t)+1}\, \dx = \sup \left\{  \int_{\R} \psi (u(t)+1)  \, \dx \mid \psi: \R \to \R \, \mbox{ such that } \norm{\psi}_{L^{\infty}} \leq 1  \right\}.
		\end{align*}
	Choosing $\psi$ as usual as the terminal data for $\zeta$ satisfying \eqref{eq:zetapdeminus}, our task is to bound
		\begin{align}
		  \ddt\int_{\R}  \zeta(u(t)+1)\,\dx.\label{turnto}
		\end{align}
An integration in time and taking the supremum over the terminal data  will then lead to
		\begin{align}
			\Wm (T)
		  	\lesssim \begin{cases} T^{1/2} + \Wm (0) & \mbox{for all }T, 
		  		\\
		  	\Wm (0) + 1 + \frac{\ln T}{T^{1/4}}\sup_{t\leq T}\Wm (t)+1 & \mbox{for } 2\leq T,
		 	\end{cases}
		\end{align}
from which we deduce \eqref{eq:W-}.

Hence we turn to \eqref{turnto}. The definition of $\zeta$ and integration by parts yield
		\begin{eqnarray}\label{eq:duality_simple}
		\lefteqn{\ddt \int_\R \zeta (u+1) \, \dx  =  \int_\R \zeta_t (u+1) \, \dx + \int_\R \zeta u_t \, \dx} \nonumber \\
		& = & \int_\R (\zeta_{xxxx}-G''(-1)\zeta_{xx}) (u+1) \, \dx + \int_\R -\zeta ((u+1)_{xx}-G'(u))_{xx} \, \dx \nonumber\\
		& = & \int_\R \zeta_{xx} \left( G'(u)-G'(-1)-G''(-1)(u+1)\right)\,\dx,
		\end{eqnarray}
		where we also used $G'(-1)=0$. Taylor's formula and the $L^\infty$ bound on $u$ lead to
		\begin{align}\label{eq:taylorest}
		\left\vert \int_\R \zeta_{xx} \left( G'(u)-G'(-1)-G''(-1)(u+1)\right)\,\dx \right\vert \lesssim \norm{\zeta_{xx}}_{L^{\infty}} \int_\R (u+1)^2\,\dx.
		\end{align}
		Since we assume $E(u)\leq 2\co-\epsilon$, we can use \eqref{eq:energyestimate-} to bound
		\begin{align}\label{eq:u+1}
		\int_\R (u+1)^2\,\dx \lesssim  E(u).
		\end{align}
		We now substitute \eqref{eq:taylorest} and \eqref{eq:u+1} into \eqref{eq:duality_simple} and argue as follows. For order one times $T$, the simple energy estimate $E(u) \lesssim 1$ and Proposition~\ref{p:semigroup} lead to
		\begin{align*}
		\int_0^T \norm{\zeta_{xx}}_{L^{\infty}} E(t) \,\dt  \overset{\eqref{eq:zetaxxminus}}{\lesssim}  \int_0^T  \frac{1}{(T-t)^{1/2}} \,\dt \lesssim T^{1/2}.
		\end{align*}
		For large $T$,  the above argument bounds the integral over the terminal layer $[T-1,T]$ by a constant.

		For the integral over $[0,T-1]$, we apply  $E(t)\lesssim 1$ in the form $E(t)\lesssim E(t)^{1/2}$, and use \eqref{eq:EnergyDecay-} to estimate
		\begin{eqnarray}
		\lefteqn{\int_0^{T-1} \norm{\zeta_{xx}}_{L^{\infty}} E(t) \,\dt} \nonumber \\
		& \lesssim & \int_0^{T-1} \norm{\zeta_{xx}}_{L^{\infty}} E(t)^{1/2} \,\dt   \overset{\eqref{eq:zetaxxminus},\eqref{eq:EnergyDecay-}}{\lesssim}  \int_0^{T-1}  \frac{1}{T-t} \frac{\sup_{t\leq T}\Wm (t)+1}{t^{1/4}}\,\dt \nonumber \\
		& \lesssim & \frac{\ln T}{T^{1/4}}\sup_{t\leq T}\Wm (t)+1. \nonumber
		\end{eqnarray}	
\end{proof}

\appendix

\section{Proofs of the EED estimates and the differential inequality for the dissipation} \label{ap:eedproofs}
Here we address the proofs of the energy gap and dissipation estimates from Subsection~\ref{ss:eed}. These estimates and their proofs are closely related to the EED inequalities from \cites{OR,OW}. For Lemma~\ref{l:eedwt}, the energy gap and dissipation estimates with respect to $\wt$ on $\R$, we give the proofs in detail. For Lemma~\ref{l:basiceed-}, the energy estimate is easy and for the  dissipation estimate, one can argue as for Lemma~\ref{l:eedwt}, but in a (much) simpler fashion. The proof of the estimates with respect to $\wt$ on the torus from Lemma~\ref{l:basiceedtorusw} is analogous to that for the corresponding estimates on $\R$, and we only comment on the most important changes that are necessary to adapt the argument. We proceed similarly for Lemma~\ref{l:basiceedtorusw}, the estimates with respect to the optimal bump $w_c$ on the torus. For more detail, we refer to \cite{B}.
\subsection{Auxiliary tools}
One tool, analogous to \cite[Lemma 2.1]{OW}, is the following Hardy-type inequality. We will use it on both $\R$ and the torus, so we state a general version.
\begin{lemma} \label{l:hardy}
There exists $\ell_1<\infty$ with the following property. Let $v$ be the centered kink and $I=[-\infty,\ell]$ or $[-\red{\tilde \ell},\ell]$ for $\red{\tilde \ell,} \ell\geq \ell_1$. Suppose $f\in H^1(I)$ satisfies
	\begin{align} \label{eq:hardyassumpt}
	\int_I f v_x\,\dx =0.
	\end{align}
	Then there holds
	\begin{align} \label{eq:hardytype}
	\int_I\frac{1}{x^2+1} f^2\, \dx \lesssim \int_I f_x^2\,\dx.
	\end{align}
\end{lemma}
\begin{proof}[Proof of Lemma~\ref{l:hardy}]
	We will prove the result on $(-\infty,\ell]$. The result on $[-\red{\tilde \ell},\ell]$ is proved similarly.
	We begin with the elementary Hardy inequality
	\begin{align} \label{eq:hardy}
	\int_I \frac{1}{x^2+1} (f-f(0))^2\, \dx \lesssim \int_I f_x^2\,\dx.
	\end{align}
	Because of the exponential tails of $v_x$ and $\ell\geq \ell_1$, we can define the order-one normalization factor
	\begin{align*}
	C_N:=\int_I v_x\,\dx \approx 1-\exp(-\ell/C)  \approx 1.
	\end{align*}
	We use this constant to write
	\begin{align*}
	f(0) = \frac{1}{C_N} f(0) \int_I v_x\,\dx 
\overset{\eqref{eq:hardyassumpt}}= \frac{1}{C_N} \int_I (f(0)-f)v_x \,\dx.
	\end{align*}
	From here, we employ Jensen's inequality with measure $\abs{v_x}\dx$, exponential decay of $v_x$, and the Hardy inequality \eqref{eq:hardy} to deduce
	\begin{align*}
	f(0)^2 \lesssim  \left( \int_I \abs{f-f(0)}|v_x|\,\dx \right)^2 \lesssim \int_I (f-f(0))^2|v_x|\,\dx \lesssim \int_I f_x^2\,\dx,
	\end{align*}
	and hence also
	\begin{align}
	\int_I \frac{1}{x^2+1}f(0)^2\,\dx \lesssim \int_I f_x^2\,\dx.\label{zeroha}
	\end{align}
The combination of \eqref{eq:hardy} and \eqref{zeroha} yields the result.
\end{proof}

	We also recall the following fact about the kernel of the linearized dissipation operator.
\begin{lemma}[{\cite[Lemma~3.4]{OW}}] \label{l:linaux2}
	If $f\in C^3(\R)$ with $f_x,\,f_{xx}\in L^2(\R)$ satisfies
	\begin{align*}
	\left( -f_{xx} + G''(v) f \right)_x =0,
	\end{align*}
	then $f = \alpha v_{x}$ for some $\alpha \in \R$.
\end{lemma}

\subsection{Proofs of the EED estimates for the two-parameter slow manifold}

We now turn to the energy gap and dissipation estimates with respect to glued kink profiles $\wt\in \N(L)$.
As in \cite[Lemma~3.1 and Lemma~3.2]{OW}, the proof of the nonlinear estimates relies on linearized estimates, which we formulate in terms of the kinks $v_{\tilde{a}}$, $v_{\tilde{b}}$ from which $\wt$ is constructed. Notice that the construction of $\tilde{w}$ assures that the contribution of the ``patching region'' to the energy gap and the dissipation is bounded by the exponentially small error term (on the right-hand side of \eqref{eq:Eerrorlow}-\eqref{eq:Derror}), so  it suffices to establish the estimates for the kinks.

We begin by collecting some linearized estimates. As usual we give two forms, one for Problem 2 and another for Problem 1. By translation invariance, we may assume that the $L^2$ closest glued kink profile belongs to $\N(0,L)$. We state our results for the left half-line or ``left half-torus''; the analogous results hold to the right of the origin.
\begin{lemma}\label{l:linER}
There exists $\ell_1<\infty$ such that for all $\red{\tilde\ell},\ell\ge \ell_1$ the following holds true. Let $I=(-\infty,\ell]$ or $[-\red{\tilde \ell},\ell]$. Suppose that $f\in H^1(I)$ satisfies
\begin{align*}
	\int_{I} fv_{x}\,\dx=0.
\end{align*}
Then there holds
\begin{align} \label{eq:linER}
	\int_{I} f^2\,\dx \lesssim\int_{I} f_x^2+G''(v)f^2\,\dx.
\end{align}
\end{lemma}
For the dissipation estimate, we need to increase our region of integration slightly on the right-hand side. This ``extra error'' will be absorbed later in the nonlinear estimates; see the proof of Lemma \ref{l:eedwt}.
\begin{lemma}\label{l:linDR}
There exists $\ell_1<\infty$ such that for all $\red{\tilde\ell},\ell,\red{\tilde{r}},r\ge \ell_1$ the following holds true. Let $I=(-\infty,\ell]$ or ${[-\red{\tilde\ell},\ell]}$ and let $I_+:={[-\red{\tilde\ell-\tilde r},-\red{\tilde\ell}] \cup I\cup [\ell,\ell+r]}$. Suppose that $f\in C^3(I_+)\cap \dot{H}^1(I_+)$ satisfies
\begin{align*}
	\int_{I} fv_{x}\,\dx=0.
\end{align*}
Then there holds
\begin{align} \label{eq:linDR}
	&\int_{I} \frac{1}{1+x^2}f^2+f_x^2+f_{xx}^2+f_{xxx}^2\,\dx \nonumber \\
	\lesssim & \;\int_{I_+} \left[(-f_{xx}+G''(v)f)_x\right]^2\,\dx+\frac{1}{{\ell\wedge r}}\int_{I_+}\red{\frac{1}{1+x^2} f^2+}f_x^2\,\dx.
\end{align}
\end{lemma}

\begin{proof}[Proof of Lemma \ref{l:linER}]The proofs on $\R$ and the torus are essentially identical; for notational simplicity, we will give the argument for the half-line.
Suppose for a contradiction that the statement is false. Then there exists sequences $\ell_n\to \infty$ and $f_n\in H^1((-\infty,\ell_n))$ such that
\begin{align}
\int_{-\infty}^{\ell_n} f_n^2\,\dx &=1,\label{l2nrm}\\
\int_{-\infty}^{\ell_n} f_nv_{x}\,\dx&=0,\notag\\
\limsup_{n\to\infty}\int_{-\infty}^{\ell_n} f_{nx}^2+G''(v)f_n^2\,\dx&\leq 0.\label{liminfco}
\end{align}
Since $f_n$ is uniformly bounded in $H^1_{\loc}$, there exists a subsequence and $f:\R\to\R$ such that
\begin{align}
f_n&\rightharpoonup f \text{ in }H^1_\loc(\R), \notag\\
f_n&\to f \text{ in }C^0_\loc(\R), \notag\\
\int_\R f^2\,\dx&\le 1,\label{ine1}\\
\int_\R fv_x\,\dx&=0,\notag\\
\int_\R f_x^2+G''(v)f^2\,\dx&\le 0,\label{ine2}
\end{align}	
where for the last inequality we use a combination of $C_{\loc}^0$ convergence and lower semicontinuity. From \eqref{ine1} and \eqref{ine2} we deduce in addition that $f\in H^1(\R)$.
Since $v$ is a minimizer of the energy for $\pm 1$ boundary conditions, the second variation of the energy around $v$ in direction $f\in H^1(\R)$ is positive, so that \eqref{ine2} improves to the equality
\begin{align*}
\int_\R f_x^2+G''(v)f^2\,\dx= 0.
\end{align*}
Consequently $f$ minimizes the linearized energy gap functional, and thus satisfies the Euler-Lagrange equation
\begin{align*}
 f_{xx}+G''(v)f= 0.
\end{align*}
This implies $f\in C^2(\R)$ and $f_{xx}\in L^2(\R)$. According to Lemma~\ref{l:linaux2}, we deduce $f=\alpha v_x$ for some $\alpha\in \R$. Together with $\int_\R fv_x=0$, this implies $f=0$. We will now show that
\begin{align*}
  \int_{-\infty}^{L_n}f_n^2\,\dx\to 0,
\end{align*}
contradicting \eqref{l2nrm}.

On the one hand, $C_{\loc}^0$ convergence to $f=0$ means that for any $X<\infty$, there holds
\begin{align}
  \lim_{n\to\infty}\int_{-X}^X f_n^2\,\dx=\lim_{n\to\infty}\int_{-X}^X G''(v)  f_n^2\,\dx=0.\label{tozero}
\end{align}
On the other hand, choosing $X$ large enough so that
\begin{align*}
  \inf_{(-\infty,L_n)\setminus (-X,X)}G''(v)\geq\frac{1}{2}G''(1)>0,
\end{align*}
we obtain from the combination of \eqref{liminfco} and \eqref{tozero} that
\begin{align*}
 \limsup_{n\to\infty} \int_{-\infty}^{-X} f_n^2\,\dx+\int_{X}^{\ell_n}f_n^2\,\dx\lesssim
 \limsup_{n\to\infty}
  \int_{-\infty}^{-X} G''(v)f_n^2\,\dx+\int_{X}^{\ell_n}G''(v)f_n^2\,\dx\leq 0.
\end{align*}
\end{proof}

\begin{proof}[Proof of Lemma \ref{l:linDR}]
\red{As before we give the proof on the half-line.} As usual, we use the Hardy-type inequality from Lemma~\ref{l:hardy} and the form of the dissipation to argue that it is enough to show the inequality for the left-hand side $\int_{-\infty}^\ell f_x^2\,\dx$.  (See for instance the proof of \cite[Lemma 3.2]{OW}.) Hence suppose that such an estimate does not hold. Then there exist sequences $\ell_n, r_n\to \infty, f_n\in \dot{H}^1((-\infty,\ell_n+r_n))$ such that
\begin{align}
\int_{-\infty}^{\ell_n} f_nv_{x}\,\dx&=0,\notag\\
\int_{-\infty}^{\ell_n} f_{nx}^2\,\dx &=1,\label{562}\\
\int_{-\infty}^{\ell_n+r_n} \left[(-f_{nxx}+G''(v)f_n)_x\right]^2\,\dx&\to 0,\label{eq:Diss_conv}\\
\frac{1}{\ell_n\wedge r_n}\int_{-\infty}^{\ell_n+r_n}\red{\frac{1}{1+x^2}f_n^2+} f_{nx}^2\,\dx&\to 0.\label{eq:extra_term}
\end{align}
By the usual argument (see for instance the proof of \cite[Lemma 3.2]{OW}), \eqref{562} - \eqref{eq:extra_term} yield
\begin{align}
\int_{-\infty}^{\ell_n} \frac{1}{1+x^2}f_n^2+f_{nx}^2+f_{nxx}^2+f_{nxxx}^2\,\dx &\lesssim1,\label{eq:improved_f}
\end{align}
Since $f_n$ is uniformly bounded in $H^3_{\loc}(-X,X)$, there exists a subsequence and a limit function $f$ such that
\begin{align}
f_n\rightharpoonup f \text{ in }H^3_\loc(\R), \;\;
f_n & \to f \text{ in }C^2_\loc(\R), \notag\\
\int_\R fv_x\,\dx&=0,
\notag \\
\int_\R f_{x}^2+f_{xx}^2+f_{xxx}^2\,\dx&\lesssim 1,\notag\\
\int_\R \left[(-f_{xx}+G''(v)f)_x\right]^2\,\dx&\le 0.\label{fxs}
\end{align}	
From \eqref{fxs}, we obtain
\begin{align*}
 (-f_{xx}+G''(v)f)_x= 0,
\end{align*}
from which we deduce $f\in C^3(\R)$. Since also $f_x,\,f_{xx}\in L^2(\R)$, Lemma~\ref{l:linaux2} gives $f=\alpha v_x$ for some $\alpha\in \R$. Together with $\int_\R fv_x\,\dx=0$ this implies $f=0$.

We claim now that
\begin{align} \label{linDissclaim}
\int_{-\infty}^{\ell_n+r_n} \left[(-f_{nxx}+G''(1)f_n)_x\right]^2\,\dx&\to 0.
\end{align}
By \eqref{eq:Diss_conv} it is enough to show that
\begin{align*}
\int_{-\infty}^{\ell_n+r_n} \left[([G''(v)-G''(1)]f_n)_x\right]^2\,\dx&\to 0.
\end{align*}
For any $X<\infty$, the convergence  on $(-X,X)$ follows from the uniform convergence $f_n\to 0$. For the integral over $[\ell_n,\ell_n+r_n]$ we estimate
\begin{align*}
\int_{\ell_n}^{\ell_n+r_n}\left(G''(v)-G''(1)\right)^2f_{nx}^2+G'''(v)^2v_x^2f_n^2\,\dx&\to 0,
\end{align*}
where we used \eqref{eq:extra_term} and
\[\abs{G''(v)-G''(1)}\lesssim \frac{1}{\ell_n}\qquad\text{and}\qquad \abs{G'''(v)v_x}\lesssim \frac{1}{{(1+x^2)}\ell_n}\]
for $x\ge \ell_n$. Finally we observe that we can use the properties of $v$ and \eqref{eq:improved_f}
to make the integrals over $(-\infty,-X)$ and $(X,\ell_n)$ arbitrarily small for $X$ large enough.

Now let $g_n\coloneqq f_{nx}$ and for an arbitrary $\eps>0$, choose $n$ large enough so that we can express \eqref{562}, \eqref{eq:extra_term}, and \eqref{linDissclaim} in terms of $g_n$ and $\eps$ as
\begin{align}
\int_{-\infty}^{\ell_n} g_n^2\,\dx&=1,\label{eq:g_ident}\\
\int_{-\infty}^{\ell_n+r_n} \left(-g_{nxx}+G''(1)g_n\right)^2\,\dx&\le \eps,\label{eq:sec}\\
\int_{-\infty}^{\ell_n+r_n}g_n^2\,\dx&\le \eps(\ell_n\wedge r_n).
\end{align}
 According to the last inequality, there is a point $x_n\in[\ell_n,\ell_n+r_n-1]$ such that $\int_{x_n}^{x_n+1}g_n^2\le \eps$.  Let $\varphi\in C^\infty(\R;[0,1])$ such that $\varphi\equiv 1$ on $(-\infty,x_n)$, $\varphi\equiv 0$ on $(x_n+1,\infty)$ and  $\abs{\varphi_{xx}}\lesssim 1$.
Then, according to \eqref{eq:sec} there holds
\begin{eqnarray}
\eps&\ge& \int_{-\infty}^{\ell_n+r_n} \left[-g_{nxx}+G''(1)g_n\right]^2\,\dx \notag\\
&\ge& \int_{-\infty}^{\ell_n+r_n} \left[-g_{nxx}+G''(1)g_n\right]^2\varphi\,\dx \notag\\
&\ge& -2G''(1)\int_{-\infty}^{\ell_n+r_n}  g_{nxx}g_n\varphi\,\dx+(G''(1))^2\int_{-\infty}^{\ell_n}g_n^2\,\dx\notag\\
&\overset{\eqref{eq:g_ident}}=&-2G''(1)\int_{-\infty}^{\ell_n+r_n} g_{nxx}g_n\varphi\,\dx+(G''(1))^2.\label{subone1}
\end{eqnarray}
For the first term on the right-hand side, we integrate by parts to estimate
\begin{align}
-\int_{-\infty}^{\ell_n+r_n} g_{nxx}g_n\varphi\,\dx=\int_{-\infty}^{\ell_n+r_n}  g_{nx}g_n \varphi_x+ g_{nx}^2\varphi\,\dx\ge -\frac{1}{2}\int_{x_n}^{x_n+1} g_n^2\varphi_{xx}\,\dx.\label{sub1}
\end{align}
Recalling $\int_{x_n}^{x_n+1} g_n^2\abs{\varphi_{xx}}\,\dx\lesssim \eps$, we obtain from \eqref{subone1} and \eqref{sub1} a contradiction for $\eps$ small enough in relation to $G''(1)$.
\end{proof}

We are now ready to establish the nonlinear estimates.
\begin{proof}[Proof of Lemma \ref{l:eedwt}]
As explained at the beginning of this subsection, the kinks $v_{\tilde{a}}$, $v_{\tilde{b}}$ out of which $\tilde{w}$ is constructed will play an important role, and we may assume without loss of generality that $\wt\in\N(0,L)$, so that $\tilde{a}<0$, $\tilde{b}>0$, and $-\tilde{a}\sim\tilde{b}\sim L$. Note also that according to our assumptions and Lemma~\ref{l:u2zeros}, $u$ has zeros $a,\,b$ such that
\begin{align*}
  \abs{a-\tilde{a}}+\abs{b-\tilde{b}}\lesssim \Wt+\Wt^{1/2},
\end{align*}
and hence
\begin{align*}
  -\tilde{a}\sim -a\sim\tilde{b}\sim b\sim L.
\end{align*}
For notational simplicity, we will focus on the infinite line case; the estimates for the torus are proved in the same way.

We begin with the energy gap, i.e., by establishing the energy gap estimates \eqref{eq:Eerrorlow}-\eqref{eq:Eerrorup}.
We will split the proof into the following steps.

\textit{Step} 1: If $|\Et| \ll 1$, then
$\norm{u-\wt}_{L^{\infty}} \ll 1$.

\textit{Step} 2: If $\norm{u-\wt}_{L^{\infty}} \ll 1$, then $\int_\R \ft^2\,\dx\lesssim |\Et|+\exp(-L/C)$.

\textit{Step} 3: If $E(u)\lesssim 1$ and $\int_\R \ft^2\,\dx\lesssim |\Et|+\exp(-L/C)$, then $\int_{\R}\ft^2+\ft_x^2\,\dx\lesssim|\Et|+\exp(-L/C)$.

\textit{Step} 4: If $\Et>0$, $\Et \gtrsim 1$,  and $E(u_0) \leq 4\co-\epsilon$, then $\int_{\R}\ft^2+\ft_x^2\,\dx\lesssim\Et$.

\textit{Step} 5: $E(u)\lesssim 1$ implies $\Et\lesssim\int_{\R}\ft^2+\ft_x^2\,\dx+\exp(-L/C)$.

Notice that $\Et$ is not necessarily positive, but that $\Et<0$ implies $\abs{\Et}\ll 1$, so we lose no generality in Step 3 by assuming positivity of the energy gap. By the same reasoning applied to each infinite half-line, smallness of the energy gap in Step 1 implies smallness of $E_{(-\infty,0)}(u)-E_{(-\infty,0)}(v_{\tilde{a}})$, and the same on the positive half-line.
In the following we will sometimes make use of the energy identity
\begin{align}
E_{(-\infty,0)}(u)-E_{(-\infty,0)}(v_{\tilde{a}})  &= \int_{\mid (-\infty,0]} \frac{1}{2} (u-v_{\tilde{a}})_{x}^2 + G(u) - G(v_{\tilde{a}}) - G'(v_{\tilde{a}})(u-v_{\tilde{a}}) \,\dx\notag\\
&\qquad+ v_{\tilde{a}x}(0)(u(0)-v_{\tilde{a}}(0)), \label{EidentityR}
\end{align}
which follows from  $-v_{xx}+G'(v)=0$ and integration by parts.

\underline{Step 1}:  We begin by comparing to the kinks $v_a$ and $-v_b$ on $(-\infty,0]$ and $[0,\infty)$, respectively. To be concrete, we focus on $v_a$ and the negative half-line. According to \eqref{EEDassumptwt}, there exists $X<\infty$ and a point $x_*\in(-X,0)$ such that
\begin{align*} 
u(x_*) \approx 1\qquad\text{while also}\qquad v_a(x_*)\approx 1.
\end{align*}
Therefore, using the Modica-Mortola trick, we obtain
\begin{align} \label{approx1R}
\int_{(-\infty,x_*]} u_x \sqrt{2G(u)} \,\dx \approx \int_{(-\infty,x_*]} v_{ax} \sqrt{2G(v_{a})} \,\dx = E_{(-\infty,x_*]}(v_{a})
\end{align}

We now use smallness of the energy gap on $(-\infty,x_*)$ (see above) to deduce
\begin{eqnarray}
1 & \gg &  E_{(-\infty,x_*]}(u) - E_{(-\infty,x_*]}(v_a) \nonumber \\
& = & \half \int_{(-\infty,x_*]} \left(u_x - \sqrt{2G(u)}\right)^2 + 2\left(u_x\sqrt{2G(u)}\right)\,\dx -  E_{(-\infty,x_*]}(v_a)\nonumber \\
&  \overset{\eqref{approx1R}}\approx &  \frac{1}{2} \int_{(-\infty,x_*]} \left( u_x - \sqrt{2G(u)} \right)^2 \,\dx.  \nonumber
\end{eqnarray}
In other words, $u$ satisfies
\begin{align*}
u_x = \sqrt{2G(u)} + r \quad\mbox{on } (-\infty,x_*)\qquad\text{and}\qquad
u(a) = 0,
\end{align*}
where  $r$ is small in $L^2$. Since
 $v_a$ satisfies
\begin{align*}
v_{a x} = \sqrt{ 2G(v_{a})} \quad\mbox{on } (-\infty,x_*)\qquad\text{and}\qquad v_a(a)=0,
\end{align*}
ODE theory  yields $L^{\infty}$ closeness of $u$ to $v_a$ on a large interval $(a-X,a+X)$ around $a$. In the same way, one obtains $L^{\infty}$ closeness of $u$ to $-v_b$ on a large interval around $b$.

It remains to show closeness to $v_a$ on $(-\infty,a-X)$ and $(a+X,0]$ (and similarly for $-v_b$). For this, we use
\begin{align*}
  v_a\approx -1\quad\text{on}\;(-\infty,a-X),\qquad v_a\approx 1\quad\text{on}\;(a+X,0],
\end{align*}
together with
\begin{align*}
  \lim_{x\to-\infty}u(x)=-1,\qquad u(a-X)\approx -1,
\end{align*}
smallness of the energy gap, and the trick of Modica and Mortola.

Finally, to deduce the desired closeness to $\wt$ on $\R$, it suffices to improve from closeness to $v_a$ and $v_b$ to closeness to $v_{\tilde{a}}$ and $v_{\tilde{b}}$. For closeness of $u$ to $v_{\tilde{a}}$ on $(-\infty,0]$, we argue as in \cite[Proof of (1.18), Lemma~1.3]{OW}. Indeed, orthogonality gives
\begin{align*}
0 = \int_{(-\infty,0]} \Big( (u- v_{a}) + (v_{a}-v_{\tilde{a}}) \Big)v_{\tilde{a}x} \,\dx,
\end{align*}
and thus
\begin{align*}
\left\vert \int_{(-\infty,0]}\left(v_{\tilde{a}}- v_a \right) v_{\tilde{a}x} \,\dx \right\vert = \left\vert \int_{(-\infty,0]}\left( u-v_a\right) v_{\tilde{a}x} \,\dx \right\vert \leq \norm{u-v_a}_{L^{\infty}((-\infty,0])}\int \abs{v_{\tilde{a}x}}\,\dx \ll 1,
\end{align*}
from which we deduce $\abs{a-\tilde{a}} \ll 1$. The same argument on the positive half-line yields $\abs{b-\tilde{b}}\ll 1$, and combining these ingredients implies
\begin{align*}
\norm{u-v_{\tilde{a}}}_{L^{\infty}((-\infty,0])} + \norm{u+v_{\tilde{b}}}_{L^{\infty}([0,\infty))} \ll 1.
\end{align*}

\underline{Step 2}:
The energy gap identity \eqref{EidentityR},
Taylor expansion of $G(u)$ around $v_{\tilde{a}}$, and $\norm{u-v_{\tilde{a}}}_{L^{\infty}((-\infty,0])} \ll 1$ yield a lower bound on the energy gap:
\begin{align*}
E_{(-\infty,0)}(u)-E_{(-\infty,0)}(v_{\tilde{a}}) & \geq \int_{(-\infty,0]} \frac{1}{2} \left( (u-v_{\tilde{a}})_{x}^2 + G''(v_{\tilde{a}})(u-v_{\tilde{a}})^2 \right) \,\dx \\
&\qquad - C \norm{ u-v_{\tilde{a}}}_{L^{\infty}((-\infty,0])} \int_{(-\infty,0]} (u-v_{\tilde{a}})^2 \,\dx-\exp(-L/C).
\end{align*}
The linear bound from \eqref{eq:linER} (after translation), nondegeneracy of $G$ (cf.\ Assumption~\eqref{ass:G}), and smallness of $\norm{u-v_{\tilde{a}}}_{L^{\infty}((-\infty,0])} $ yield
\begin{align}
\int_{(-\infty,0]} (u-v_{\tilde{a}})^2 \,\dx\lesssim \abs{E_{(-\infty,0)}(u)-E_{(-\infty,0)}(v_{\tilde{a}})} +\exp(-L/C).\label{sub585}
\end{align}
The analogous argument on $[0,\infty)$ and the construction of $\wt$ yield the result.

\underline{Step 3}: 
From the identity \eqref{EidentityR}, the uniform bound on $u$, and Taylor's formula, we obtain
\begin{eqnarray}
\int_{(-\infty,0]} [(u-v_{\tilde{a}})_{x}]^2 \, \dx &\overset{\eqref{EidentityR}}{\lesssim}& \abs{E_{(-\infty,0)}(u)-E_{(-\infty,0)}(v_{\tilde{a}})}\notag\\
&&\qquad - \int_{(-\infty,0]} G(u) - G(v_{\tilde{a}}) - G'(v_{\tilde{a}})(u-v_{\tilde{a}}) \, \dx+\exp(-L/C)\notag\\
& \lesssim&\abs{ E_{(-\infty,0)}(u)-E_{(-\infty,0)}(v_{\tilde{a}})} + \int_{(-\infty,0]} (u-v_{\tilde{a}})^2 \, \dx +\exp(-L/C). \nonumber
\end{eqnarray}
Substituting \eqref{sub585} completes the estimate on $(-\infty,0]$ and the analogous argument on the positive half-line together with the construction of $\wt$ yields the result.

\underline{Step 4}: As in \cite[Step 3, Proof of Lemma~1.3, (1.18)]{OW}, we say that $u$ has a $\delta$-transition layer on $(x_-,x_+)$ if
$u(x_-)=-1+\delta$ and $u(x_+) =1-\delta$, or vice versa, and
$(x_-,x_+)$ is minimal in the sense that the same is not true for any proper subset of
$(x_-,x_+)$. By using $E(u)\leq 4\co-\epsilon$, choosing $\delta = \delta(\epsilon)$ sufficiently small, applying the Modica-Mortola
trick, and recalling the boundary conditions of $u$, we deduce that $u$ has exactly one $\delta$-transition on $(-\infty,0]$ that is positioned around the zero $a$. We argue as before to derive existence of a constant $C(\delta) < \infty$ such that
\begin{align}
\begin{cases} (u(x)-(-1))^2\leq C(\delta) G(u(x)) & \mbox{for } x \in (-\infty,a], \\ (u(x)-1)^2 \leq C(\delta) G(u(x)) & \mbox{for } x \in (a,0]. \end{cases} \label{CdeltaR}
\end{align}
Let $\chi: (-\infty,0] \to \{-1,1\}$ denote the characteristic function
\begin{align*}
\chi (x) = \begin{cases} -1, & x \in (-\infty,a], \\ 1, & x \in (a,0]. \end{cases}
\end{align*}
From \eqref{CdeltaR} it follows that
\begin{align} \label{eq:uchideltaR}
\int_{(-\infty,0]} \abs{u - \chi}^2 \,\dx \leq C(\delta) \int_{(-\infty,0]} G(u) \,\dx \lesssim C(\delta).
\end{align}
The properties of $v_a$ yield
\begin{align} \label{eq:wctildechi1R}
\int_{(-\infty,0]} \abs{v_{a} - \chi}^2 \,\dx \lesssim 1.
\end{align}
Since $v_{\tilde{a}}$ is the $L^2$ closest shifted kink to $u$ on $(-\infty,0]$, a triangle inequality yields
\begin{eqnarray}
\lefteqn{\int_{(-\infty,0]} \abs{u - v_{\tilde{a}}}^2 \,\dx } \nonumber \\
& \lesssim & \int_{(-\infty,0]} \abs{u - v_a}^2 \,\dx \lesssim \int_{(-\infty,0]} \abs{u - \chi}^2 \,\dx + \int_{(-\infty,0]} \abs{\chi - v_a}^2 \,\dx \nonumber \\
& \overset{\eqref{eq:uchideltaR},\eqref{eq:wctildechi1R}}{\lesssim} & C(\delta) + 1 \lesssim \left( C(\delta) + 1 \right) \Et,  \nonumber
\end{eqnarray}
where the last inequality follows from $\Et \gtrsim 1$. We improve from this $L^2$ to an $H^1$ estimate (up to exponential errors) using \eqref{EidentityR}. The analogous estimate on $[0,\infty)$ and the definition of $\wt$ completes the proof.

\underline{Step 5}: Once again we will deduce the estimate on $\R$ (up to exponential errors) from two estimates on the half-lines and the definition of $\wt$. From \eqref{EidentityR} we deduce
\begin{eqnarray}
\lefteqn{\abs{E_{(-\infty,0)}(u)-E_{(-\infty,0)}(v_{\tilde{a}})}}\notag\\
& \overset{\eqref{EidentityR}}{=} & \left|\int_{(-\infty,0]} \frac{1}{2} [(u-v_{\tilde{a}})_{x}]^2 + G(u) - G(v_{\tilde{a}}) - G'(v_{\tilde{a}})(u-v_{\tilde{a}}) \, \dx \right|+\exp(-L/C)\nonumber \\
& \lesssim & \int_{(-\infty,0]} [(u-v_{\tilde{a}})_{x}]^2 \,\dx + \sup_{\vert s \vert \leq 1 + \norm{u}_{L^{\infty}}} \abs{G''(s)} \int_{(-\infty,0]} (u-v_{\tilde{a}})^2 \,\dx +\exp(-L/C) \nonumber \\
&\lesssim& \int_{(-\infty,0]}[(u-v_{\tilde{a}})_{x}]^2  +(u-v_{\tilde{a}})^2 \,\dx+\exp(-L/C). \nonumber
\end{eqnarray}

We now turn to the proof of the dissipation estimate~\eqref{eq:Derror}, once again focusing for notational simplicity on the infinite line case.
We carry out the proof via these steps:

\textit{Step} 1:  If $\D \ll 1$, then $\norm{ \ft}_{L^{\infty}(\R)} \ll 1$. 

\textit{Step} 2: If  $\norm{\ft}_{L^{\infty}(\R)}\ll 1$, then $\int_{\R}  \ft_{x}^2+\ft_{xx}^2+\ft_{xxx}^2\,\dx \lesssim \D+\exp(-L/C)$.

\textit{Step} 3: If $\D \gtrsim 1$ and $E(u_0) \leq 4\co-\epsilon$, then $\int_{\R}  \ft_{x}^2\,\dx \lesssim \D$.

\textit{Step} 4: If $\int_{\R}  \ft_{x}^2\,\dx \lesssim \D$ and $E(u) \lesssim 1$, then $\int_{\R}  \ft_{x}^2+\ft_{xx}^2+\ft_{xxx}^2\,\dx \lesssim \D+\exp(-L/C)$.

We will make use of the half-line identity
\begin{eqnarray}
\lefteqn{\int_{(-\infty,X]} \left[ \left(-u_{xx} + G'(u)\right)_x\right]^2\,\dx}\notag\\
& \overset{\eqref{elv}}=& \int_{(-\infty,X]} \left[ \left(- (u-v_{\tilde{a}})_{xx} + G'(u) -G'(v_{\tilde{a}}) \right)_x\right]^2\,\dx \nonumber \\
& = &\int_{(-\infty,X]} \Big[ \big(  - (u-v_{\tilde{a}})_{xx} + G''(v_{\tilde{a}}) (u-v_{\tilde{a}}) \notag\\
&&\qquad + G'(u) -G'(v_{\tilde{a}}) -G''(v_{\tilde{a}})(u-v_{\tilde{a}}) \big)_x\Big]^2\,\dx. \label{DidentityR}
\end{eqnarray}

\underline{Step $1$}: To fix ideas, we will work on the left half-line, noticing that small dissipation implies that its restriction to either half-line is small as well. Arguing as in \cite[Proof of (1.19), Lemma~1.3]{OW}, we bound the so-called discrepancy $\xi \coloneqq \frac{1}{2}u_x^2-G(u)$ via
\begin{align} \label{smallnessxiR}
\norm{\xi}_{L^{\infty}((-\infty,0])} \lesssim \max\{\D^{1/6},L^{-1}\}.
\end{align}
We use this to deduce that  $u$ satisfies
\begin{align*}
u_x = \pm \sqrt{2G(u)+\xi} \quad \mbox{on } (-\infty,0]\quad\text{and}\quad u(a)=0,
\end{align*}
where $\xi$ is small in $L^\infty$. We  argue as in  \cite{OW} to derive closeness of $u$ to the kink $v_a$ and $-v_b$ on $(-\infty,0]$ and $[0,\infty)$, respectively, and use orthogonality as in the argument from Step 1 of the energy gap estimate to prove $L^{\infty}$ closeness of $u$ to $v_{\tilde{a}}$ on $(-\infty,0]$ and $-v_{\tilde{b}}$ on $[0,\infty)$. Closeness to $\wt$ on $\R$ follows as usual.

\underline{Step 2}: We use~\eqref{DidentityR} to bound
\begin{eqnarray}
\D & \gtrsim & \frac{1}{2} \int_{(-\infty,X]}  \left[\Big(-(u-v_{\tilde{a}})_{xx} + G''(v_{\tilde{a}})(u-v_{\tilde{a}}) \Big)_x \right]^2 \,\dx \nonumber \\
& & - \int_{(-\infty,X]} \left[\Big(G'(u) -G'(v_{\tilde{a}}) -G''(v_{\tilde{a}})(u-v_{\tilde{a}}) \Big)_x \right]^2 \,\dx\notag\\
&\gtrsim & \int_{(-\infty,0]} \red{\frac{1}{1+x^2}(u-v_{\at})^2}  + [(u-v_{\tilde{a}})_x]^2+[(u-v_{\tilde{a}})_{xx}]^2+[(u-v_{\tilde{a}})_{xxx}]^2\,\dx\notag\\
&&\quad -\eps\int_{(-\infty,X]}\red{\frac{1}{1+x^2}(u-v_{\at})^2}+[(u-v_{\tilde{a}})_x]^2\,\dx\notag\\
&&\quad-
\int_{(-\infty,X]} \left[\Big(G'(u) -G'(v_{\tilde{a}}) -G''(v_{\tilde{a}})(u-v_{\tilde{a}}) \Big)_x \right]^2 \,\dx,
\label{Dstep2toruswt0R}
\end{eqnarray}
where we have applied the linearized estimate \eqref{eq:linDR} and we can make the $\eps$ in the first error term as small as we like by choosing $X$ large enough. For the second integral, we use Taylor's theorem and $\norm{u-v_{\tilde{a}}}_{L^{\infty}((-\infty,0])} \ll 1$ to estimate
\begin{align*}
\lefteqn{\left\vert \Big( G'(u)-G'(v_{\tilde{a}})-G''(v_{\tilde{a}})(u-v_{\tilde{a}}) \Big)_x \right\vert}\\
&\lesssim \abs{(u-v_{\tilde{a}}) (u-v_{\tilde{a}})_x} + \abs{(u-v_{\tilde{a}})^2 v_{\tilde{a}x}} + \abs{(u-v_{\tilde{a}})^2 (u-v_{\tilde{a}})_x }.
\end{align*}	
We use exponential decay of $v_{\tilde{a}}$ at infinity and the Hardy-type inequality from Lemma~\ref{l:hardy} to estimate
\begin{eqnarray}
\lefteqn{\int_{(-\infty,X]} \left[\Big( G'(u) -G'(v_{\tilde{a}}) -G''(v_{\tilde{a}})(u-v_{\tilde{a}})\Big)_x \right]^2 \, \dx} \nonumber \\
& \lesssim & \norm{u-v_{\tilde{a}}}_{L^{\infty}((-\infty,X])} \int_{(-\infty,X]} (u-v_{\tilde{a}})^2 v_{\tilde{a}x}^2 + [(u-v_{\tilde{a}})_x]^2\, \dx\nonumber \\
& \overset{\eqref{eq:hardytype}}\lesssim & \norm{u-v_{\tilde{a}}}_{L^{\infty}((-\infty,X])} \int_{(-\infty,X]} [(u-v_{\tilde{a}})_{x}]^2\,\dx. \label{Dsteptoruswt3R}
\end{eqnarray}
We now substitute \eqref{Dsteptoruswt3R} into \eqref{Dstep2toruswt0R}, which yields
\begin{align*}
\lefteqn{  \int_{(-\infty,0]} \frac{1}{1+x^2}(u-v_{\at})^2 + [(u-v_{\tilde{a}})_x]^2+\ldots+[(u-v_{\tilde{a}})_{xxx}]^2\,\dx}\\
&\lesssim  \D+\hat{\eps}\int_{(-\infty,X]}\frac{1}{1+x^2}(u-v_{\at})^2 + [(u-v_{\tilde{a}})_x]^2\,\dx
\end{align*}
for an $\hat{\eps}$ that we can make arbitrarily small by choosing $X$ large (but order-one with respect to $L$) and $\norm{u-\wt}_{L^{\infty}(\R)}$.

In the same way, we obtain the analogous estimate on $[-X,\infty)$. Adding the two  and recalling the definition of $\wt$, we deduce
\begin{align*}
  \int_{\R} \red{ \frac{1}{1+x^2}\ft^2}+ \ft_x^2+\ft_{xx}^2+\ft_{xxx}^2\,\dx\lesssim \D+\hat{\eps}\int_{\R}\red{\frac{1}{1+x^2}\ft^2+}\ft_x^2\,\dx+\exp(-L/C).
\end{align*}
Absorbing the $\hat{\eps}$-dependent term in the left-hand side completes the proof.

\underline{Step 3}: We use $D\gtrsim 1$ and the bounds on the energy gap to derive
\begin{align*}
\int_{\R} \ft_x^2\,\dx \overset{\eqref{eq:Eerrorlow}}\lesssim \abs{\Et} + \exp(-L/C) \lesssim 1 \lesssim \D.
\end{align*}

\underline{Step 4}: According to the interpolation inequality
\begin{align*}
\int_{(-\infty,0]} [(u-v_{\tilde{a}})_{xx}]^2 \,\dx \lesssim \int_{(-\infty,0]} [(u-v_{\tilde{a}})_{x}]^2 + [(u-v_{\tilde{a}})_{xxx}]^2 \,\dx  +\exp(-L/C)  \end{align*}
it is sufficient to establish
\begin{align*}
\int_{(-\infty,0]}  [(u-v_{\tilde{a}})_{xxx}]^2 \,\dx \lesssim \D + \int_{(-\infty,0]}  [(u-v_{\tilde{a}})_{x}]^2 \,\dx.
\end{align*}
Since the first equality in \eqref{DidentityR} implies
\begin{align*}
\int_{(-\infty,0]} \left[ -(u-v_{\tilde{a}})_{xxx} + \Big(G'(u)-G'(v_{\tilde{a}})\Big)_x  \right]^2 \,\dx \lesssim \D,
\end{align*}
it suffices to show 	
\begin{align} \label{Dstep4toruswt2R}
\int_{(-\infty,0]} \left[ \Big(G'(u)-G'(v_{\tilde{a}})\Big)_x  \right]^2 \,\dx \lesssim \int_{(-\infty,0]} [(u-v_{\tilde{a}})_{x}]^2 \,\dx.
\end{align}
We use Taylor's formula to write
\begin{eqnarray}
\lefteqn{(G'(u)-G'(v_{\tilde{a}}))_x}  \nonumber \\
& = & \ft_{x} \int_0^1 G''\Big(v_{\tilde{a}}+\theta(u-v_{\tilde{a}})\Big)(1-\theta)\, {\rm{d}}\theta \nonumber \\
& &  +\, (u-v_{\tilde{a}}) \int_0^1 G'''(v_{\tilde{a}}+\theta (u-v_{\tilde{a}})) (v_{\tilde{a}x}+ \theta (u-v_{\tilde{a}})_{x})(1-\theta) \, {\rm{d}}\theta. \nonumber
\end{eqnarray}
Boundedness of $\abs{u-v_{\tilde{a}}}$ and the properties of $G$ then lead to
\begin{align} \nonumber
\int_{(-\infty,0]} \left[\Big(G'(u)-G'(v_{\tilde{a}})\Big)_x  \right]^2 \,\dx \lesssim \int_{(-\infty,0]} [(u-v_{\tilde{a}})_{x}]^2 \,\dx + \int_{(-\infty,0]} v_{\tilde{a}x}^2 (u-v_{\tilde{a}})^2 \, \dx.
\end{align}
From the exponential decay of $v_{\tilde{a}x}^2$ and Lemma~\ref{l:hardy}, we recover \eqref{Dstep4toruswt2R}.
\end{proof}
\subsection{EED estimates relative to $-1$}
\begin{proof}[Proof of Lemma \ref{l:basiceed-}]
	The energy bound $E(u)\leq 2\co - \epsilon$ implies {by the Modica-Mortola trick and the boundary conditions} that there is a $\delta>0$ such that $u<1-\delta$. Combined with the bound on $\norm{u}_{L^\infty}$, this yields
	\begin{align*}
	(u+1)^2\sim G(u),
	\end{align*}
	which suffices to deduce \eqref{eq:energyestimate-}.  To obtain the dissipation estimate, we first establish  the linear estimate via integration by parts:
	\begin{align}
	\int_\R \left((u_{xx}-G''(-1)u)_x\right)^2\,\dx&=\int_\R u_{xxx}^2-2G''(-1)u_xu_{xxx}+(G''(-1))^2 u_x^2\,\dx \notag \\
	&=\int_\R u_{xxx}^2+2G''(-1)u_{xx}^2+(G''(-1))^2 u_x^2\, \notag \dx\\
	&\gtrsim \int_\R u_{xxx}^2+u_{xx}^2+u_x^2\,\dx. \label{eq:LinDiss}
	\end{align}
	If the dissipation is order one, we argue as in the proof of Lemma \ref{l:eedwt} by bounding $\int_\R u_x^2\lesssim E\lesssim 1\lesssim \D$ and then improving  to the full inequality. For $\D\ll 1$, we write
	\begin{align}
	\int_\R u_x^2\,dx\overset{\eqref{eq:LinDiss}}\lesssim\int_\R \left((u_{xx}-G''(-1)u)_x\right)^2\,\dx&\lesssim \D+\int_\R \left(G''(u)u_x-G''(-1)u_x\right)^2\,\dx\notag\\
	&=\D+\int_\R \left(G''(u)-G''(-1)\right)^2u_x^2\,\dx\notag
	\end{align}
	and observe that it is enough to prove $\norm{u+1}_{L^\infty(\R)}\ll 1$, since then $\norm{G''(u)-G''(-1)}_{L^{\infty}(\R)}\ll 1$ and the second term on the right-hand side can be absorbed in the left-hand side. To see that $\norm{u+1}_{L^\infty(\R)}\ll 1$ is small if $\D\ll 1$, we use \eqref{smallnessxiR} to obtain smallness of $\xi=\frac 12 u_x^2-G(u)$ in $L^\infty$. By the boundary conditions, $u$ must have a maximum at some point $x_0$. There, $u_x(x_0)=0$ and hence $G(u(x_0))=\abs{\xi(x_0)}\ll 1$. This implies that either $u(x_0)\approx -1$ or $u(x_0)\approx 1$ and in view of $E\le 2e_\ast-\epsilon$ and the boundary conditions we infer $u(x_0)\approx -1$ if the dissipation is small enough (again using the trick of Modica and Mortola). Thus, $u(x_0)\approx -1$ and arguing analogously for the minimum, we obtain $\norm{u+1}_{L^\infty(\R)}\ll 1$.
\end{proof}
	
\subsection{Proofs of the EED estimates for the one-parameter slow manifold}


On the torus (cf. Lemma~\ref{l:basiceedtorusw}), once the energy gap $\Et$ is small enough, we will use energy gap and dissipation estimates with respect to $w_c$.
As for EED with respect to $\wt$, we begin with lower bounds for the linearized energy gap and dissipation.
\begin{lemma}\label{lem:linest}
There exists $L_0<\infty$ such that for all $L\ge L_0$ the following holds true. Suppose $f\in H^1([-L,L])$ satisfies
	\begin{align*}
	\int_{[-L,L]}f\,\dx=0, \quad \mbox{and}\quad
	\int_{[-L,L]} fw_{cx}\,\dx=0.
	\end{align*}
	Then
	\begin{align} \label{eq:linEestwc}
	\int_{[-L,L]} f^2\,\dx \lesssim \int_{[-L,L]} f_x^2+G''(w_c)f^2\,\dx
	\end{align}
	and
	\begin{align} \label{eq:linDestwc}
	\int_{[-L,L]} f_x^2 \,\dx \lesssim \int_{[-L,L]} \left[ \Big(-f_{xx} + G''(w_c)f\Big)_x \right]^2 \,\dx.
	\end{align}
\end{lemma}
	The proofs are very similar to the proofs of Lemmas~\ref{l:linER} and~\ref{l:linDR} and rely on the Euler-Lagrange equation~\eqref{el} and the orthogonality condition~\eqref{elL2}.
For details, see \cite{B}.

	We now turn to the nonlinear  estimates.
\begin{proof}[Proof of Lemma~\ref{l:basiceedtorusw}] We begin with the energy gap estimates~\eqref{eq:energyestimatewctorus}. We will split the proof into the following steps.
	
	\textit{Step} 1: If $\norm{u-w_c}_{L^2([-L,L])} \ll \frac{1}{L}$ and $\int_{[-L,L]} f_c \,\dx = 0$, then $\norm{ f_c}_{L^{\infty}} \ll 1$.
	
	\textit{Step} 2: If $\norm{f_c}_{L^{\infty}} \ll 1$, then $\int_{[-L,L]}  f_c^2 \, \dx \lesssim \E$.
	
	\textit{Step} 3: If $\int_{[-L,L]}  f_c^2 \, \dx \lesssim \E$ and $\E \lesssim 1$, then $\int_{[-L,L]}  f_c^2 + f_{cx}^2 \, \dx \lesssim \E$.
	
	\textit{Step} 4: $\E \lesssim 1$ implies $ \E \lesssim \int_{[-L,L]}  f_c^2 + f_{cx}^2 \, \dx$.
	
	Steps 2-4 can be derived as in the proof of Lemma~\ref{l:eedwt}, using the energy gap identity
	\begin{align*}
	E(u) - E(w_c)\overset{\eqref{el}}= \int_{[-L,L]} \frac{1}{2} f_{cx}^2 + G(u) - G(w_c) - G'(w_c)f_c \,\dx,
	\end{align*}
 	Taylor expansion, and the linearized energy gap estimate \eqref{eq:linEestwc}.
 	
 	The derivation of Step 1 is of a different spirit.  The $L^2$-closest bump $w_c$ has zeros $a_c, b_c$, and let $\wt(a_c,b_c)$ denote the glued kink profile with $\alpha=a_c$, $\beta=b_c$. From $\norm{u-w_c}_{L^2([-L,L])} \ll \frac{1}{L}$ and the properties of $w_c$, it follows that $\norm{u-\wt(a_c,b_c)}_{L^2([-L,L])} \ll \frac{1}{L}$ and hence also for the $L^2$-closest kink profile $\wt$, there holds
 $\norm{u-\wt}_{L^2([-L,L])} \ll \frac{1}{L}$.

Let the distance between the zeros of a bump $w$ be denoted by $\ell$, while the distance between the zeros of $\wt$ is denoted by $\tilde{\ell}$. As in the proof of Lemma \ref{l:u2zeros}, we shift the centered bump $w$ such that it shares its left zero with $\wt$ and denote the shifted bump by $w_d$.
We then use $\int_{[-L,L]} u\,\dx = \int_{[-L,L]} w_d\,\dx$ and the Cauchy-Schwarz inequality  to estimate
	\begin{eqnarray}
	\abs{\ell-\tilde{\ell}} &\sim& \left\vert \int_{[-L,L]} \wt-w_d\,\dx \right\vert =  \left\vert \int_{[-L,L]} \wt-u\,\dx \right\vert
	 \lesssim  L^{1/2} \left(\int_{[-L,L]} \ft^2\,\dx \right)^{1/2}
\ll 1, \nonumber
	\end{eqnarray}
where we have invoked the above estimate on the $L^2$-distance.	It follows that
	\begin{align*}
	\norm{\wt-w_d}_{L^{\infty}([-L,L])} \ll 1\qquad\text{and}\qquad \Et\ll \frac{1}{L},
	\end{align*}
where we have recalled Remark \ref{r:expt} and $\E\ll \frac{1}{L}$.
	From these facts and the triangle inequality, we obtain
	\begin{eqnarray}
	\norm{u-w_d}_{L^{\infty}([-L,L])} &\leq& \norm{u-\wt}_{L^{\infty}([-L,L])} + \norm{\wt-w_d}_{L^{\infty}([-L,L])} \notag\\
&\overset{\eqref{eq:basiceedconsequence}}\lesssim& \abs{\Et}^{1/2} + \exp(-L/C)+ \norm{\wt-w_d}_{L^{\infty}([-L,L])} \ll 1.\label{wd1}
	\end{eqnarray}
	We use this together with orthogonality in the form
\begin{align} \label{orthprop}
	0=\int_{[-L,L]} (u-w_c)w_{cx}\,\dx = \int_{[-L,L]} (u-w_d+w_d-w_c)w_{cx}\,\dx.
	\end{align}
to establish smallness of $\norm{f_c}_{L^{\infty}([-L,L])}$. Indeed,
denoting the zeros of
$w_d$ by
$\{a_d,\,b_d\}$, we observe that
	\begin{eqnarray*}
\left\vert \int_{[-L,L]} (w_d-w_c)w_{cx}\,\dx \right\vert
&\overset{\eqref{orthprop}}=&\left\vert \int_{[-L,L]} (u-w_d)w_{cx}\,\dx \right\vert\\
&\leq& \norm{u-w_d}_{L^{\infty}([-L,L])}\int_{[-L,L]} \abs{w_{cx}}\,dx\ll 1.
	\end{eqnarray*}
Using that $\int (w_d-w_c)w_{cx} \gtrsim \min {1, |d-c|}$, and that \emph{$w_c$ and $w_d$ have the same distance between zeros}, it follows that
	\begin{align}
	\abs{a_c-a_d}+\abs{b_c-b_d}\ll 1\qquad\text{and}\qquad\norm{w_d-w_c}_{L^{\infty}([-L,L])} \ll 1.\label{wd2}
	\end{align}
From here we obtain
	\begin{align*}
	\norm{u-w_c}_{L^{\infty}([-L,L])} \leq \norm{u-w_d}_{L^{\infty}([-L,L])} + \norm{w_d-w_c}_{L^{\infty}([-L,L])}\overset{\eqref{wd1},\eqref{wd2}} \ll 1.
	\end{align*}
	
We now turn to the dissipation estimate~\eqref{eq:dissipationestimatewctorus}.	We will follow these steps:
	
	\textit{Step} 1:  If $\abs{\E} \ll \frac{1}{L}$ and $\int_{[-L,L]} f_c \,\dx = 0$, then $\norm{ f_c}_{L^{\infty}([-L,L])}  \ll 1$. 
	
	\textit{Step} 2: If $\norm{  f_{c}}_{L^{\infty}([-L,L])}  \ll 1$ then $\int_{[-L,L]}  f_{cx}^2\,\dx \lesssim \D$.
	
	
	\textit{Step} 3: If $\int_{[-L,L]}  f_{cx}^2\,\dx\lesssim \D$ and $\E\lesssim 1$, then $\int_{[-L,L]}  f_{cx}^2 + f_{cxx}^2 + f_{cxxx}^2\,\dx\lesssim \D$.
	
	Step 1 is analogous to Step 1 for the energy gap estimate. For Steps 2 and 3, we  argue as in the proof of the dissipation estimate on $\R$ with respect to $\wt$.
	In both Steps~2 and~3, we use the identity
	\begin{eqnarray}
	\lefteqn{-u_{xx} + G'(u)}\nonumber \\
	& \overset{\eqref{el}}=& - f_{cxx} + G'(u) -G'( w_{c}) \nonumber \\
	& = & \left( - f_{cxx} + G''( w_{c}) f_{c} \right) + \left(G'(u) -G'( w_{c}) -G''(w_{c}) f_{c}\right), \nonumber
	\end{eqnarray}
	Taylor expansion, the linearized dissipation estimate \eqref{eq:linDestwc}, and Lemma~\ref{l:hardy}.
\end{proof}

\subsection{Proof of Lemmas \ref{l:nash}~--~\ref{l:dissy}}
\begin{proof}[Proof of Lemma~\ref{l:nash}]
	The elementary inequalities
	\begin{align}
	\sup  \abs{f}^2 \lesssim \left( \int  f_{x}^2 \,\dx \int  f^2 \,\dx \right)^{1/2} \quad \mbox{and} \quad
	\int  f^2 \,\dx \lesssim  \sup  \, \abs{f}\int  \abs{f} \,\dx \label{bumpinterpol0}
	\end{align}
	yield
	\begin{align} \label{bumpinterpol1}
	\sup  \abs{f} \lesssim \left( \int  \abs{f} \, \dx \int  f_{x}^2 \, \dx \right)^{1/3}.
	\end{align}
	Inserting \eqref{bumpinterpol1} into \eqref{bumpinterpol0} and combining this with~\eqref{eq:nashD}, we obtain
	\begin{equation} \label{bumpproofnasheq1}
	\int  f^2 \, \dx \overset{\eqref{bumpinterpol0},\eqref{bumpinterpol1}}\lesssim \left( \int  f_{x}^2 \,\dx \right)^{1/3} \left( \int  \abs{f} \,\dx \right)^{4/3} \overset{\eqref{eq:nashD},\eqref{eq:nashW}}{\lesssim} \D^{1/3} \,\W^{4/3}.
	\end{equation}
	Next we use $\E \lesssim 1$ for the na\"{\i}ve estimate of the gradient term:
	\begin{equation} \label{bumpproofnasheq2}
	\int  f_{x}^2\,\dx \overset{\eqref{eq:nashE},\,\eqref{eq:nashD}}{\lesssim} \D^{1/3} \E^{2/3} \overset{\eqref{eq:nashE1}}\lesssim \D^{1/3} \leq \D^{1/3}(\W+1)^{4/3}.
	\end{equation}
	Combining \eqref{bumpproofnasheq1} and~\eqref{bumpproofnasheq2} with~\eqref{eq:nashE} implies
	\begin{align*}
	\E \overset{\eqref{eq:nashE}}{\lesssim} \int  f^2 + f_{x}^2 \,\dx \overset{\eqref{bumpproofnasheq1},\, \eqref{bumpproofnasheq2}}{\lesssim} \D^{1/3}(\W+1)^{4/3}.
	\end{align*}
\end{proof}
\begin{proof}[Proof of Lemma~\ref{l:odenash}]
	We obtain \eqref{eq:odenashresult} from \eqref{eq:odenashpre1}, \eqref{eq:definitionWT}, \eqref{eq:odenashpre2}, and an integration in time.
\end{proof}

\begin{proof}[Proof of Lemma~\ref{l:intdissipationbound}]
We consider separately the cases (i) $T\leq \W_T^4$ and (ii) $T>\W_T^4$.
In case (i), we  deduce from H\"{o}lder's inequality that
	\begin{align*}
	\int_0^{T} \D^{\gamma}(t) \,{\rm{d}}t \leq T^{1-\gamma} \left( \int_0^{T} \D(t) \, {\rm{d}}t \right)^{\gamma} \lesssim T^{1-\gamma} \E_0^{\gamma} \lesssim T^{1-\gamma}\leq \W_T^{4(1-\gamma)}.
	\end{align*}
In case (ii), we let $\tau \in (0,T)$ be a constant to be optimized below. The same H\"{o}lder's inequality as above on $(0,\tau)$ gives
\begin{align*}
	\int_0^{\tau} \D^{\gamma}(t) \,{\rm{d}}t
\lesssim \tau^{1-\gamma}.
\end{align*}
On $(\tau,T)$ we pick $\theta \in (\frac{1}{\gamma} - 1, \frac{1}{2})$ and use
	\begin{align} \label{eq:intdissproof2}
	t^{\theta}\D(t) = - t^{\theta} \frac{{\rm{d}}}{{\rm{d}}t} \E(t) = \frac{{\rm{d}}}{{\rm{d}}t} \left( - t^{\theta} \E(t) \right) + \theta t^{\theta -1} \E(t)
	\end{align}
to estimate
	\begin{eqnarray}
	\lefteqn{\int_{\tau}^T \D^{\gamma}(t) \, {\rm{d}}t} \nonumber \\
	& = & \int_{\tau}^T \D^{\gamma}(t) \cdot t^{\theta\gamma} \cdot t^{-\theta\gamma} \, {\rm{d}}t \leq \left( \int_{\tau}^T \left( t^{-\theta\gamma} \right)^{\frac{1}{1-\gamma}} \, {\rm{d}}t \right)^{1-\gamma} \left( \int_{\tau}^T t^{\theta} \D(t) \, {\rm{d}}t \right)^{\gamma} \nonumber \\
	& \overset{\eqref{eq:intdissproof2}}{=} & \left( \left[ t^{-\frac{\theta\gamma}{1-\gamma}+1} C_{\theta \gamma} \right]_{\tau}^T \right)^{1-\gamma} \left( \int_{\tau}^T \ddt (-t^{\theta} \E(t)) \, {\rm{d}}t + \int_{\tau}^T \theta t^{\theta -1} \E(t) \, {\rm{d}}t \right)^{\gamma} \nonumber \\
	& \overset{\eqref{eq:intdisspre}}{\lesssim} & \frac{1}{\tau^{\theta\gamma - (1-\gamma)}} \left( \tau^{\theta} \frac{\W_T^2 }{\tau^{1/2}}
	+ \W_T^2 \int_{\tau}^T t^{\theta -1-\frac12}  \, {\rm{d}}t \right)^{\gamma} \lesssim \frac{ \W_T^{2\gamma}}{\tau^{\frac{3\gamma}{2}-1}}, \nonumber
	\end{eqnarray}
where $\lesssim$ is allowed to depend on $C_{\theta\gamma}$.
Combining the two estimates gives
	\begin{align*}
	\int_0^T \D^{\gamma}(t) \, {\rm{d}}t \lesssim \tau^{1-\gamma} + \frac{ \W_T^{2\gamma}}{\tau^{\frac{3\gamma}{2}-1}}
	\end{align*}
and optimizing in $\tau$ yields \eqref{eq:intdissipationboundres}.
\end{proof}

\begin{proof}[Proof of Lemma~\ref{l:dissy}]
	
	We adapt the argument from \cite[Lemma~1.4, (1.24)]{OW} to our setting. In order to be self-contained, we provide the proof of \eqref{eq:dissyexp} on $\R$ in full detail. The proof of \eqref{eq:dissyexp} on the torus follows similarly. In Step 2, we comment on the proof of \eqref{eq:dissywc}, in which things simplify.
	
	\underline{Step 1:} For convenience, we define
	\begin{align*}
		g \coloneqq u_{xx} -G'(u) \quad \mbox{and} \quad h \coloneqq g_{xx}.
	\end{align*}
Without loss of generality, assume $\wt\in\N(0,L)$.
	We calculate
	\begin{eqnarray}
		\frac12 \ddt \D(t)
		& = & \ddt \frac12 \int_\R \Big[ \Big( -u_{xx} + G'(u) \Big)_x \Big]^2 \,\dx = \int_\R g_{xx} \Big( -\partial_x^2+ G''(u)\Big) u_t\,\dx \nonumber \\
		& = &  -\int_{(-\infty,0]} h_x^2+G''(v_{\tilde{a}})h^2\,\dx + \int_{(-\infty,0]} \red{\Big(G''(v_{\tilde{a}})-G''(u)\Big)} h^2\,\dx \nonumber \\
		& & -\int_{[0,\infty)} h_x^2+G''(\red{-}v_{\tilde{b}})h^2\,\dx + \int_{[0,\infty)} \red{\Big(G''(-v_{\tilde{b}})- G''(u)\Big)} h^2\,\dx. \label{D321}
	\end{eqnarray}
	We will establish the estimate on $(-\infty,0]$. The estimate on $[0,\infty)$ follows analogously.

	We begin by considering the first term on the right-hand side of \eqref{D321} (which should be a good term, up to some error). The idea is to capitalize on the linearized energy gap estimate for functions that are orthogonal to $v_{\tilde{a}x}$. With this goal in mind, we decompose
	\begin{align}\label{eq:alpha}
		g_{xx}=h = h_0 + \alpha v_{\tilde{a}x} \quad \mbox{with} \quad \alpha \coloneqq \frac{\int_{(-\infty,0]} hv_{\tilde{a}x} \,\dx}{ \int_{(-\infty,0]} v_{\tilde{a}x}^2\,\dx },
	\end{align}
from which it follows that
\begin{align}\label{eq:orth}
	\int_{(-\infty,0]} h_0v_{\tilde{a}x}\,\dx = 0.
\end{align}
and hence, according to Lemma \ref{l:linER}, there holds
\begin{align} \label{eq:linEh0}
	\int_{(-\infty,0]} h_{0x}^2 + G''(v_{\tilde{a}})h_0^2 \,\dx \gtrsim \int_{(-\infty,0]} h_0^2 + h_{0x}^2 \,\dx.
	\end{align}
	
In order to be able to integrate by parts over most of the domain, we introduce a smooth partition of unity $\varphi_-,\,\varphi: (-\infty,0]\to[0,1]$ with
	\begin{align*}
		\varphi = \begin{cases}
			1 & \mbox{for } x \in (-1,0], \\ 0 & \mbox{for } x\in (-\infty,-2],
		\end{cases} \quad \varphi_- = 1-\varphi,
	\quad	\mbox{and} \quad \norm{(\varphi_-)_{x}}_{L^2((-\infty,0])} +  \norm{\varphi_{x}}_{L^2((-\infty,0])} \lesssim 1.
	\end{align*}
As a preliminary step, we will need	
control of $\alpha$. Using the partition of unity, we estimate
	\begin{eqnarray}
	\alpha^2 & \lesssim &	\left( \int_{(-\infty,0]} g_{xx} v_{\tilde{a}x} \varphi_-\,\dx \right)^2+\left( \int_{(-\infty,0]} g_{xx} v_{\tilde{a}x} \varphi\,\dx \right)^2 \nonumber \\
		&  \lesssim & \left(-\int_{(-\infty,0]} g_x(v_{\tilde{a}xx}\varphi_-+v_{\tilde{a}x}(\varphi_-)_{x}) \,\dx\right)^2 +  \int_{(-\infty,0]} g_{xx}^2\varphi^2\,\dx \int_{(-2,0)} v_{\tilde{a}x}^2  \,\dx  \nonumber \\
		& \overset{\eqref{eq:alpha}}\lesssim & \int_{(-\infty,0]} g_x^2\,\dx + \exp(-L/C) \int_{(-\infty,0]} (h_0^2+\alpha^2(v_{\tilde{a}x})^2)\varphi^2\,\dx \nonumber \\
		& \lesssim &  \D + \exp(-L/C) \left(  \int_{(-\infty,0]} h_0^2\,\dx + \alpha^2\right). \nonumber
	\end{eqnarray}
Absorbing the last term gives
\begin{align} \label{alpha2}
	\alpha^2 \lesssim  \D + \exp(-L/C)  \int_{(-\infty,0]} h_0^2\,\dx .
\end{align}

With this estimate in hand, we expand the first term on the right-hand side of \eqref{D321} as
\begin{eqnarray}
	\lefteqn{\int_{(-\infty,0]} h_x^2+G''(v_{\tilde{a}})h^2\,\dx} \nonumber \\
	& \overset{\eqref{eq:alpha}}= & \int_{(-\infty,0]} h_{0x}^2 + G''(v_{\tilde{a}})h_0^2 \,\dx \red{+} 2\alpha\int_{(-\infty,0]}  \Big( v_{\tilde{a}xx}h_{0x} + G''(v_{\tilde{a}})v_{\tilde{a}x}h_{0} \Big) (\varphi_-+\varphi)\,\dx \nonumber \\
	& & + \alpha^2\int_{(-\infty,0]}  \Big( (v_{\tilde{a}xx})^2 + G''(v_{\tilde{a}})(v_{\tilde{a}x})^2 \Big) (\varphi_-+\varphi)\,\dx. \label{hlinestini}	
\end{eqnarray}
For the terms of order $\alpha$ and $\alpha^2$ on the right-hand side (which vanished in \cite{OW}), we integrate by parts and use
$-v_{\tilde{a}xx} + G'(v_{\tilde{a}}) = 0$ to simplify. Exponential smallness of the derivatives of $v_{\tilde{a}}$ on the support of $(\phi_-)_x$ and $\phi$ yield
\begin{align}
\lefteqn{\left|  \alpha\int_{(-\infty,0]}  \Big( v_{\tilde{a}xx}h_{0x} + G''(v_{\tilde{a}})v_{\tilde{a}x}h_{0} \Big) (\varphi_-+\varphi)\,\dx\right|}\notag\\
  &+\alpha^2\left|\int_{(-\infty,0]}  \Big( (v_{\tilde{a}xx})^2
  + G''(v_{\tilde{a}})(v_{\tilde{a}x})^2 \Big) (\varphi_-+\varphi)\,\dx\right|\notag\\
  &\qquad\lesssim \exp(-L/C)\left(\alpha^2+\int_{(-\infty,0]} h_0^2 + h_{0x}^2\,\dx \right)\overset{\eqref{alpha2}}\lesssim \exp(-L/C)\left(D+\int_{(-\infty,0]} h_0^2 + h_{0x}^2\,\dx \right).\notag
\end{align}

Substituting into \eqref{hlinestini} for $L_0$ large enough gives
	\begin{eqnarray}
		\lefteqn{\int_{(-\infty,0]} h_x^2+G''(v_{\tilde{a}})h^2\,\dx} \nonumber \\
		& \ge & \int_{(-\infty,0]} h_{0x}^2 + G''(v_{\tilde{a}})h_0^2 \,\dx  - \exp(-L/C)\left(D+\int_{(-\infty,0]} h_0^2 + h_{0x}^2\,\dx\right).\nonumber\\
		& \overset{\eqref{eq:linEh0}}\gtrsim & \int_{(-\infty,0]} h_{0x}^2 +h_0^2  \,\dx  - \exp(-L/C) D,
		 \label{D32summand1}
	\end{eqnarray}
which completes our estimate of the first term on the right-hand side of \eqref{D321}.

We claim that it suffices to show
\begin{eqnarray}\label{D32summand2}
\lefteqn{\int_{(-\infty,0]} \abs{ G''(u)-G''(v_{\tilde{a}})}h^2\,\dx}  \nonumber \\
&\lesssim &\D^{3/2}+\left(\eps+\exp(-L/C)\right)\int_{(-\infty,0]} h_0^2+h_{0x}^2\,\dx+\exp(-L/C).
\end{eqnarray}
Indeed, inserting \eqref{D32summand1} and \eqref{D32summand2} and the analogous estimates on $[0,\infty)$ into \eqref{D321} and applying Young's inequality, we obtain
	\begin{eqnarray}
		\ddt \D \lesssim \D^{3/2} +\exp(-L/C). \nonumber
	\end{eqnarray}
	
	To establish \eqref{D32summand2}, we use $\norm{u-v_{\tilde{a}}}_{L^{\infty}((-\infty,0])} \lesssim 1$ and Young's inequality to estimate
	\begin{eqnarray}
		 \int_{(-\infty,0]} &\abs{G''(u)-G''(v_{\tilde{a}})}h^2\,\dx
		\leq  \sup_{\abs{\tau} \leq 1+\norm{u-v_{\tilde{a}}}_{L^\infty((-\infty,0])}} \abs{G'''(\tau)} \int_{(-\infty,0]} \abs{u-v_{\tilde{a}}} h^2 \,\dx \nonumber \\
		&\overset{\eqref{eq:alpha}}\lesssim  \int_{(-\infty,0]} \abs{u-v_{\tilde{a}}} h_0^2\,\dx + \alpha^2 \int_{(-\infty,0]} \abs{u-v_{\tilde{a}}} (v_{\tilde{a}x})^2\,\dx.
		 \label{D32eq2}
	\end{eqnarray}
For the first term on the right-hand side of \eqref{D32eq2}, we use the assumption $\norm{u-\wt}_{L^\infty}\le \eps$ in the form $\norm{u-v_{\tilde a}}_{L^\infty((-\infty,0])}\le \eps+\exp(-L/C)$ to obtain
	\begin{align}\label{D32eq6}
		\int_{(-\infty,0]} \abs{u-v_{\tilde{a}}} h_0^2\,\dx &\leq 
(\eps+\exp(-L/C))\int_{(-\infty,0]} h_0^2\,\dx.
	\end{align}

For the second term on the right-hand side of \eqref{D32eq2}, we use exponential decay of $v_{\tilde{a}x}$, the Cauchy-Schwarz inequality, and  the Hardy-type inequality from Lemma~\ref{l:hardy} to write
	\begin{eqnarray}
		\lefteqn{\alpha^2 \int_{(-\infty,0]} \abs{u-v_{\tilde{a}}} v_{\tilde{a}x}^2\,\dx} \nonumber \\
		& \lesssim &  \alpha^2 \left( \int_{(-\infty,0]} \frac{(u-v_{\tilde{a}})^2}{1+(x-\tilde{a})^2} \,\dx \right)^{1/2} \overset{\eqref{eq:hardy}}\lesssim \alpha^2\left( \int_{(-\infty,0]} (u_x-v_{\tilde{a}x})^2\,\dx \right)^{1/2} \nonumber \\	
			& \overset{\eqref{alpha2},\eqref{eq:Derror},\eqref{eq:Eerrorlow}}\lesssim &  \D^{3/2}+ \exp(-L/C)D+\exp(-L/C)  \left(\int_{(-\infty,0]}h_0^2 \,\dx\right) \left(\abs{\Et}+\exp(-L/C)\right)\nonumber \\
			& \lesssim & \D^{3/2} +\exp(-L/C)D+ \exp(-L/C) \int_{(-\infty,0]}h_0^2 \,\dx.   \label{D32eq4}
	\end{eqnarray}	
Substituting \eqref{D32eq6} and \eqref{D32eq4} into \eqref{D32eq2} and applying Young's inequality yields \eqref{D32summand2}.
	
	\underline{Step 2:}
	For the estimate \eqref{eq:dissywc} on the torus, the derivation simplifies because we compare to $w_c$, which satisfies an Euler-Lagrange equation on the whole torus, and integrations by parts yields no boundary terms, so there is no need for a cut-off function. We define $g$ and $h$ as above and write
	\begin{align}
		\frac12 \ddt \D(t)
		= -\int_{[-L,L]} h_x^2+G''(w_c)h^2\,\dx + \int_{[-L,L]} \Big( G''(u)-G''(w_c)\Big) h^2\,\dx. \label{D321wc}
	\end{align}
	We decompose
	\begin{align*}
		h = h_0 + \alpha w_{cx} \quad \mbox{with} \quad \alpha \coloneqq \frac{\int_{[-L,L]} hw_{cx} \,\dx}{ \int_{[-L,L]} w_{cx}^2\,\dx }.
	\end{align*}
	An integration by parts and the Cauchy Schwarz inequality yield
	\begin{align*}
		\alpha^2
		\lesssim \frac{\int_{[-L,L]} w_{cxx}^2\,\dx}{\left( \int_{[-L,L]} w_{cx}^2\,\dx\right)^2} \D \lesssim \D. \nonumber
	\end{align*}
For the first term in \eqref{D321wc}, we use the decomposition of $h$,
the Euler-Lagrange equation
$-w_{cxx} + G'(w_c) = \lambda$,
integration by parts, and Lemma~\ref{lem:linest} to obtain
	\begin{align*}
		\int_{[-L,L]} h_x^2+G''(w_c)h^2\,\dx
		=  \int_{[-L,L]} h_{0x}^2 + G''(w_c) h_0^2 \,\dx \overset{\eqref{eq:linEestwc}}\gtrsim\int_{[-L,L]} h_0^2 + h_{0x}^2 \,\dx.
	\end{align*}
For the error term, we show
	\begin{align*}
		\left|\int_{[-L,L]} ( G''(u)-G''(w_c))h^2\,\dx\right| \lesssim \D^{3/4} \left( \int_{[-L,L]} h_0^2+h_{0x}^2\,\dx \right)^{1/2} + \D^{3/2},
	\end{align*}
	which we establish  as in Step 1, using the decomposition of $h$ and the energy gap and dissipation estimates from Lemma~\ref{l:basiceedtorusw}.
\end{proof}
\begin{proof}[Proof of Lemma~\ref{l:dissy_bound}]
The  differential inequality from \eqref{assumeddiffq} and integrating in time give
	\begin{align} \label{eq:dissipationdecay2}
	\D(s) \gtrsim \frac{\D(t)}{\left( 1 + \D^{\frac{1}{2}}(t)(t-s) \right)^2},\qquad\text{for }t\geq s.
	\end{align}
	On the other hand, $\ddt \E = -\D$ (together with positivity of the energy gap) yields
	\begin{align} \label{eq:dissipationdecay1}
	\int_\tau^t \D(s) \,{\rm{d}}s = - \E(t) + \E(\tau) \le \E(\tau)\qquad\text{for }t\geq \tau.
	\end{align}
From here we obtain
	\begin{eqnarray}
	\E(\tau)& \overset{ \eqref{eq:dissipationdecay1} }\ge & \int_\tau^t \D(s) \, {\rm{d}}s \overset{\eqref{eq:dissipationdecay2}} \gtrsim \int_\tau^t \frac{\D(t)}{\left( 1 + \D^{1/2}(t)(t-s) \right)^2} \, {\rm{d}}s  = \int_0^{\D^{1/2}(t)(t-\tau)} \frac{\D(t)^{1/2}}{(1+\sigma)^2} \, {\rm{d}}\sigma \nonumber \\
	& \gtrsim & \min \left\{ \D^{1/2}(t),\D(t)(t-\tau)  \right\},   \nonumber
	\end{eqnarray}
	from which the dissipation estimate follows by choosing $\tau=\frac t2$.
\end{proof}

\section*{Acknowledgements}
We gratefully acknowledge interesting discussions with and valuable insights from Felix Otto, as well as the hospitality of the Max Planck Institute for Mathematics in the Sciences, where some of the ideas for this work originated.
S. Biesenbach was funded by the Deutsche Forschungsgemeinschaft (DFG, German Research Foundation) through the RTG Energy, Entropy, and Dissipative Dynamics (EDDy) [Projektnummer 320021702 /GRK2326].
R. Schubert gratefully recognizes partial funding from DFG Grant WE 5760/1-1 and travel support from EDDy.

\end{document}